\documentclass[reqno,11pt]{amsart}
\usepackage{amsmath, amsfonts,amsthm,amssymb,amsbsy,upref,color,graphicx,amscd,hyperref,enumerate}
\usepackage{color}
\usepackage[active]{srcltx}
\usepackage[latin1]{inputenc}
\usepackage{bbm}
\usepackage{graphicx}
\usepackage{subfig} 
\usepackage{xfrac}
\usepackage{comment} %added by Po-Lam

\usepackage{dsfont}

\def\ds{\displaystyle}
\def\eps{{\varepsilon}}
\def\N{\mathbb{N}}
\def\O{\Omega}
\def\om{\omega}
\def\R{\mathbb{R}}
\def\C{\mathbb{C}}

\def\HH{\mathcal{H}}

\def\vf{\varphi}

\newcommand{\be}{\begin{equation}}
\newcommand{\ee}{\end{equation}}

\newcommand{\cp}{\mathop{\rm cap}\nolimits}

\newcommand{\ind}{\mathbbm{1}}
\newcommand{\reg}{Reg(\partial\Omega^\ast)}

\newcommand{\regu}{Reg(\partial\Omega_u)}
\newcommand{\singu}{Sing(\partial\Omega_u)}
\newcommand{\adm}{\mathcal{A}(u,x_0,r)}

\theoremstyle{plain}
\newtheorem{teo}{Theorem}[section]
\newtheorem{lm}[teo]{Lemma}
\newtheorem{prop}[teo]{Proposition}
\newtheorem{coro}[teo]{Corollary}

\theoremstyle{definition}
\newtheorem{definition}[teo]{Definition}

\newtheorem{oss}[teo]{Remark}

\def\Xint#1{\mathchoice
   {\XXint\displaystyle\textstyle{#1}}%
   {\XXint\textstyle\scriptstyle{#1}}%
   {\XXint\scriptstyle\scriptscriptstyle{#1}}%
   {\XXint\scriptscriptstyle\scriptscriptstyle{#1}}%
   \!\int}
\def\XXint#1#2#3{{\setbox0=\hbox{$#1{#2#3}{\int}$}
     \vcenter{\hbox{$#2#3$}}\kern-.5\wd0}}

\def\aver#1{\Xint-_{#1}}

\DeclareMathOperator{\dive}{div}

\usepackage{refcount}

\newcounter{cte}

\numberwithin{equation}{section}

\begin{document}

%\date{}
\title[Existence and regularity of optimal shapes]{Existence and regularity of optimal shapes for elliptic operators with drift}
%\title{\bf{Shape optimization problems for elliptic operators with drift}}
\author{Emmanuel Russ, Baptiste Trey, Bozhidar Velichkov}

\address {Emmanuel Russ: \newline \indent
Universit\'e Grenoble Alpes, CNRS UMR 5582, Institut Fourier
\newline\indent
100 rue des Math\'ematiques, F-38610 Gi\`eres, France}
\email{emmanuel.russ@univ-grenoble-alpes.fr}

\address {Baptiste Trey: \newline \indent
Universit\'e Grenoble Alpes, CNRS UMR 5582, Institut Fourier
\newline\indent
100 rue des Math\'ematiques, F-38610 Gi\`eres, France}
\email{baptiste.trey@etu.univ-grenoble-alpes.fr}

\address {Bozhidar Velichkov: \newline \indent
Universit\'e Grenoble Alpes, CNRS UMR 5224\newline\indent
700 avenue Centrale, F-38401 Domaine Universitaire de Saint-Martin-d'H\`eres, France}
\email{bozhidar.velichkov@univ-grenoble-alpes.fr}

\date{\today}

\begin{abstract}
This paper is dedicated to the study of shape optimization problems for the first eigenvalue of the elliptic operator with drift $L = -\Delta + V(x) \cdot \nabla$ with Dirichlet boundary conditions, where $V$ is a bounded vector field. In the first instance, we prove the existence of a principal eigenvalue $\lambda_1(\O,V)$ for a bounded quasi-open set $\O$ which enjoys similar properties to the case of open sets. Then, given $m>0$ and $\tau\geq 0$, we show that the minimum of the following non-variational problem
\begin{equation*}
\min\Big\{\lambda_1(\Omega,V)\ :\ \Omega\subset D\ \text{quasi-open},\ |\Omega|\leq m,\ \|V\|_{L^\infty}\le \tau\Big\}.
\end{equation*}
is achieved, where the box $D\subset \R^d$ is a bounded open set. The existence when $V$ is fixed, as well as when $V$ varies among all the vector fields which are the gradient of a Lipschitz function, are also proved.

The second interest and main result of this paper is the regularity of the optimal shape $\O^\ast$ solving the minimization problem
\begin{equation*}
\min\Big\{\lambda_1(\Omega,\nabla\Phi)\ :\ \Omega\subset D\ \text{quasi-open},\ |\O|\leq m\Big\},
\end{equation*}
where $\Phi$ is a given Lipschitz function on $D$. We prove that the optimal set $\Omega^\ast$ is open and that its topological boundary $\partial\Omega^\ast$ is composed of a {\it regular part}, which is locally the graph of a $C^{1,\alpha}$ function, and a {\it singular part}, which is empty if $d<d^\ast$, discrete if $d=d^\ast$ and of locally finite $\mathcal{H}^{d-d^\ast}$ Hausdorff measure if $d>d^\ast$, where $d^\ast \in \{5,6,7\}$ is the smallest dimension at which there exists a global solution to the one-phase free boundary problem with singularities. Moreover, if $D$ is smooth, we prove that, for each $x\in \partial\Omega^{\ast}\cap \partial D$, $\partial\O^\ast$ is $C^{1,\sfrac12}$ in a neighborhood of $x$. 
%This last result is optimal in the sense that $C^{1,\sfrac12}$ is the best regularity that one can expect.
\end{abstract}


\maketitle
\tableofcontents

\section{Introduction and main results}

Let $D$ be a bounded connected open set in $\R^d$, $d\ge 2$. For any  bounded vector field $V:D\to\R^d$ and any connected open set $\Omega\subset D$, we consider the elliptic operator with drift $L=-\Delta+V(x)\cdot\nabla$\,. In this paper we study variational optimization problems in which the variables are both the domain $\Omega$ and the drift $V$, and the cost functional is defined through the operator $L$. The aim of the present paper is twofold. From one side, we develop an existence theory for shape optimization problems for operators with drift. On the other hand, we study the regularity of the optimal shapes for vector fields $V$ that are gradients of potentials $\Phi:D\to\R$. 
We focus on the model problem 
\begin{equation}\label{e:minOV}
\min \big\{ \lambda_1(\O,V) \ : \ \O \subset D, \ |\O| \leq m,\  \|V\|_{L^\infty} \leq \tau\big\},
\end{equation}
where $m>0$  and $\tau\ge0$ are fixed constants, and $\lambda_1(\Omega,V)$ is the principal eigenvalue of the operator $L$. Our main results are the following.
%We notice that in this case the lack of specific information on the geometry of $D$ does not allow the application of a symmetrization technique in the spirit of \cite{hamel nadirashvili russ annals} nor the explicit identification of the shape of the optimal domains and the precise analytic expression of the optimal vector fields. 

\begin{teo}\label{t:th1}
Let $D\subset\R^d$ be a bounded open set, and $0<m<|D|$  and $\tau\ge0$ be fixed constants. Then, there exist a quasi-open set $\Omega\subset D$ and a vector field $V:D\to\R^d$ such that the couple $(\Omega,V)$ is a solution to the shape optimization problem \eqref{e:minOV}. 
%\begin{equation}\label{e:optO}
%\min\Big\{\lambda_1(\Omega,V)\ :\ \Omega\subset D\ \text{quasi-open},\ |\Omega|\le m \Big\}. 
%\end{equation}
\end{teo}
\noindent In particular, we prove in Theorem \ref{thm prop lambda1} below, that the principal eigenvalue $\lambda_1(\Omega,V)$ of the (non-self-adjoint) operator $L$ is well-defined on any quasi-open set $\Omega\subset D$. Preciesly, we will show that for any quasi-open set $\Omega$, there is a real eigenvalue $\lambda_1(\Omega,V)$ of the operator $L$ such that $\lambda_1(\Omega,V)\le \text{Re}\,\lambda$, for any other eigenvalue $\lambda\in\C$ of $L$.
\begin{teo}\label{t:th2}
Let $D\subset\R^d$ be a bounded open set, and $0<m<|D|$  and $\tau\ge0$ be fixed constants.  Then the shape optimization problem
\begin{equation}\label{e:minOPhi}
\min\big\{\lambda_1(\Omega,\nabla\Phi)\ :\ \Omega\subset D\ \text{quasi-open},\ |\O|\leq m,\ \Phi\in W^{1,\infty}(D),\ \|\nabla\Phi\|_{L^\infty(D)}\le\tau\big\}
\end{equation}
admits a solution $(\Omega^\ast,\nabla\Phi^\ast)$. Moreover, if $D$ is connected, then any optimal set $\Omega^\ast$ has the following properties: 
\begin{enumerate}
\item $\Omega^\ast$ is an open set; 
\item $\Omega^\ast$ has finite perimeter; 
\item $\Omega^\ast$ saturates the constraint, that is, $|\Omega^\ast|=m$; 
\end{enumerate}
The free boundary $\partial\Omega^\ast\cap D$ can be decomposed in the disjoint union of a regular part $\text{Reg}(\partial\O^\ast\cap D)$ and a singular part $\text{Sing}(\partial\O^\ast\cap D)$, where: 
\begin{enumerate}
\item[(4)] $\text{Reg}(\partial\O^\ast\cap D)$ is locally the graph of a $C^{1,\alpha}$-regular function for any $\alpha<1$; 
\item[(5)] for a universal constant $d^\ast\in\{5,6,7\}$ (see Definition \ref{def:dstar}), $\text{Sing}(\partial\O^\ast\cap D)$ is: 
\begin{itemize}
\item empty if $d<d^*$;
\item discrete if $d=d^*$; 
\item of Hausdorff dimension at most $(d-d^\ast)$ if $d>d^*$.\end{itemize}
\end{enumerate}
If the boundary $\partial D$ is $C^{1,1}$, then the boundary $\partial\Omega^\ast$ can be decomposed in the disjoint union of a regular part  $\text{Reg}(\partial\O^\ast)$ and a singular part $\text{Sing}(\partial\O^\ast)$, where: 
\begin{enumerate}
\item[(6)] $\text{Reg}(\partial\O^\ast)$ is an open subset of $\partial\Omega^\ast$ and locally the graph of a $C^{1,\sfrac12}$ function; moreover, $\text{Reg}(\partial\O^\ast)$ contains both $\text{Reg}(\partial\O^\ast\cap D)$ and $\partial\Omega^\ast\cap\partial D$; 
\item[(7)] $\text{Sing}(\partial\O^\ast)=\text{Sing}(\partial\O^\ast\cap D)$.
\end{enumerate}
\end{teo}

\noindent In fact, our result is more general. Precisely, we prove the regularity of the optimal sets for $\lambda_1(\cdot,\nabla\Phi)$ with fixed vector field $\nabla\Phi$ (see Theorem \ref{thm main} and Remark \ref{rem:th2-main}). 
\medskip

%Assume now that $\Omega$ is a $C^2$ bounded open set of Lebesgue measure $|\Omega|=m$ and let $V:\Omega\to\R^d$ be a vector field such that $\|V\|_{L^\infty}=\tau$ (where $\|V\|_{L^\infty}$ stands for the $L^{\infty}$-norm of the Euclidean norm of $V$). 
For $m,\tau,\Omega$ and $V$ as in \eqref{e:minOV}, Hamel, Nadirashvili and Russ \cite{hamel-nadirashvili-russ-annals},  proved the lower bound
\begin{equation}\label{e:hnr}
\lambda_1\left(B,\tau\frac{x}{|x|}\right)\le \lambda_1(\Omega,V),
\end{equation}
where $B$ is the ball of Lebesgue measure $m$ centered in zero; moreover, there is an equality in \eqref{e:hnr}, if and only if, up to translation, $\Omega=B$ and $V(x)=\tau\frac{x}{|x|}$. In other words, the couple $\big(B,\tau \frac{x}{|x|}\big)$ is (up to translation) the unique solution of the shape optimization problem
\begin{equation}\label{e:hnr:minpb}
\min\big\{\lambda_1(\Omega,V)\ :\ \Omega\subset\R^d,\ |\Omega|=m,\ \|V\|_{L^\infty}\le \tau\big\}.
\end{equation}

We notice that a symmetrization technique in the spirit of \cite{hamel-nadirashvili-russ-annals} cannot be applied to the problem \eqref{e:minOV}. In fact, the presence of the constraint $D$ makes impossible to determine explicitly the shape of the optimal domains or the precise analytic expression of the optimal vector fields, except in the trivial case when a ball of measure $m$ fits into $D$. Thus, we first establish the existence of an optimal domain $\Omega$ in the larger (relaxed) class of {\it quasi-open} sets and we then study the regularity of the optimal shapes through variational free boundary techniques. 
We stress that, in the case of a generic vector field $V$, the principal eigenvalue $\lambda_1(\Omega,V)$ does {\it not} have a variational formulation but is only determined trough the solution of a certain PDE on $\Omega$. In particular, the shape cost functional in \eqref{e:minOV} cannot be written in terms of a variational minimization problem involving integral cost functionals on $\Omega$. This makes the extension of the functional $\lambda_1(\cdot,V)$ to a ($\gamma$-)continuous functional on the class of quasi-open sets a non trivial problem. 
%Moreover, in this case the existence of an optimal quasi-open set cannot be obtained the continuity, with respect to the ($\gamma$-)convergence of the domains, is not a direct consequence of the (Mosco-)convergence of the underlying Sobolev spaces. 
%Up to our knowledge, Theorem \ref{t:th1} is the first instance in the literature in which the existence of optimal shapes in the class of quasi-open sets is obtained for a shape cost functional, which does not have a closed variational expression on every domain (see Remark \ref{rem:existence}).  
\medskip

In the case $\tau=0$, \eqref{e:minOV} and \eqref{e:minOPhi} are reduced to the classical shape optimization problem 
\begin{equation}\label{e:min_classical}
\min \big\{ \lambda_1(\O) \ : \ \O \subset D, \ |\O| \leq m\big\},
\end{equation}
where $\lambda_1(\Omega)$ is the first eigenvalue of the Dirichlet Laplacian on $\Omega$. For the problem \eqref{e:min_classical}, the existence of an optimal (quasi-open) set was proved by Buttazzo and Dal Maso in \cite{buttazzo-dal-maso}, the fact that the optimal sets are open (Theorem \ref{t:th2} (1)) was proved by Brian\c con and Lamboley in \cite{briancon-lamboley}, the estimate on the perimeter of the optimal set (Theorem \ref{t:th2} (2)) is due to Bucur (see \cite{bucur}), the regularity of the free boundary $\text{Reg}(\partial\Omega^\ast\cap D)$ was again proved in \cite{briancon-lamboley}; the estimate on the dimension of the singular set $\text{Sing}(\partial\Omega^\ast\cap D)$ was obtained in \cite{mazzoleni-terracini-velichkov}. Even for the classical problem \eqref{e:min_classical} the regularity up to the boundary of the box $D$ (Theorem \ref{t:th2} (6) and (7)) is new. 

\begin{oss}[On the regularity of the optimal shapes for spectral functionals]
The regularity of the optimal shapes for the eigenvalues of the Laplacian was an object of an intense study in the last years. As mentioned above, a regularity result, for the optimal sets for the first eigenvalue of the Laplacian, was proved Brian\c con and Lamboley in \cite{briancon-lamboley}. The regularity of the optimal sets for more general spectral functionals was studied in \cite{bucur-mazzoleni-pratelli-velichkov}, \cite{mazzoleni-terracini-velichkov}, \cite{kriventsov-lin-17} and \cite{kriventsov-lin-18}. An alternative approach in dimension two, based on the epiperimetric inequality from \cite{spolaor-velichkov}, was recently introduced in \cite{spolaor-trey-velichkov}, where Theorem \ref{t:th2} (6) is proved in the case $\tau=0$ and $d=2$. We notice that the method from \cite{spolaor-trey-velichkov} can be applied to give an alternative proof of Theorem \ref{t:th2} (6) in the case $\tau>0$, but the restriction on the dimension is required by the epiperimetric inequality and for now cannot be removed.
\end{oss}

\begin{oss}[On the existence of optimal shapes]\label{rem:existence}
The existence of optimal shapes in a bounded open set (box) $D\subset\R^d$ is a consequence of the theory of Buttazzo and Dal Maso (see \cite{buttazzo-dal-maso} and the books \cite{bucur-buttazzo} and \cite{henrot-pierre}) for general shape optimization problems of the form 
\begin{equation}\label{e:minJ}
\min\big\{\mathcal F(\Omega)\ :\ \Omega\subset D\ \text{quasi-open}\,,\ |\Omega|\le m\big\},
\end{equation}
for {\it shape cost functionals} $\mathcal F$ with the following properties:  
\begin{itemize}
\item $\mathcal F$ is decreasing with respect to the set inclusion; 
\item $\mathcal F$ is lower semi-continuous with respect to the ($\gamma$-)convergence of sets.
\end{itemize}
We notice that in the case when $\mathcal F$ is a function of the spectrum of the Dirichlet Laplacian on $\Omega$, the existence of an optimal set can be obtained directly (see for instance \cite{velichkov}). In fact, if $\mathcal F(\Omega)=\lambda_1(\Omega)$, then given a minimization sequence of quasi-open sets $\Omega_n$ for \eqref{e:minJ}, and setting $u_n$ to be the first eigenfunction of the Dirichlet Laplacian on $\Omega_n$, it is not hard to check that, up to a subsequence, $u_n$ converges weakly in $H^1_0(D)$ to a function $u\in H^1_0(D)$ and that the (quasi-open) set $\Omega:=\{u>0\}$ is a solution to \eqref{e:minJ}. This elementary argument works not only for $\lambda_1$, but can also be reproduced for general spectral functionals of the form $\mathcal F(\Omega)=F\big(\lambda_1(\Omega),\dots,\lambda_k(\Omega)\big)$, and also for most of the shape cost functionals present in the literature. We stress that this is not the case of the functional $\mathcal F(\Omega)=\lambda_1(\Omega,V)$. Even if $\lambda_1(\cdot,V)$ is still monotone and $\gamma$-continuous (as we will prove in Section \ref{s:existence}), its non-variational nature does not allow to use the elementary argument described above; thus, the only way to obtain the existence of an optimal set is through the Buttazzo-Dal Maso theory.   
\end{oss}

%More precisely, we show that the shape functional $\Omega\mapsto\lambda_1(\Omega,V)$ is lower semi-continuous with respect to the so-called $\gamma$-convergence of sets, which corresponds to a special topology on the quasi-open domains that takes into account the convergence of the solutions of the elliptic PDEs on these domains. We then obtain the existence of optimal sets by a general result of Buttazzo and Dal Maso \cite{buttazzo dal maso}. 

%Let $\O\subset \R^d$ be a bounded domain (here and after, ``domain'' means open connected set) and $V\in L^{\infty}(\O,\R^d)$ be a given vector field. We consider the elliptic operator with drift $L = -\Delta + V(x) \cdot \nabla$ with Dirichlet boundary conditions on $\partial\O$. In \cite{berestycki-nirenberg-varadhan} Berestycki, Nirenberg and Varadhan showed that there is a positive real eigenvalue, not greater than the real part of any other eigenvalue of $L$. This eigenvalue is called {\it principal} or {\it first eigenvalue} of $L$ and is denoted by $\lambda_1(\Omega,V)$. 

\subsection*{Optimal shapes for a fixed vector field} In this paper, we also study the case in which only the shape $\Omega$ is variable, while the vector field $V$ is fixed. Precisely, we consider the shape optimization problem 
\begin{equation}\label{e:minO}
\min \Big\{ \lambda_1(\O,V) \ : \ \O \subset D, \ |\O| \leq m \Big\},
\end{equation}
where both the upper bound $m$ of the Lebesgue measure of the domain $\Omega$ and the vector field $V$ are fixed. In this case the geometry of the optimal sets is affected both by the geometric constraint $\Omega\subset D$ and the vector field $V$. We notice that in this case it is the inclusion constraint that provides the compactness necessary for the existence of an optimal set. We show that the shape functional $\Omega\mapsto\lambda_1(\Omega,V)$ is lower semi-continuous with respect to the so-called $\gamma$-convergence of sets and then we obtain the existence of optimal sets by the general result discussed in Remark \ref{rem:existence}. Furthermore, when the vector field is the gradient of a Lipschitz function, we prove a regularity result for the optimal sets. 
Our main result is the following. 
\begin{teo}[Existence and regularity of optimal shapes for a fixed vector field]\label{thm main}
Let $D$ be a bounded open set in $\R^d$. Let $m \in (0,|D|)$ and let the vector field $V:D\to\R^d$ be such that $\|V\|_{L^\infty}=\tau<+\infty$. Then the shape optimization problem
\begin{equation}\label{e:optO}
\min\Big\{\lambda_1(\Omega,V)\ :\ \Omega\subset D\ \text{quasi-open},\ |\Omega|\le m \Big\} 
\end{equation}
admits a solution $\Omega^\ast\subset D$. Moreover, if $D$ is connected and the vector field $V$ is of the form $V=\nabla\Phi$, where $\Phi:D\to\R$ is a given Lipschitz function, then any solution $\Omega^\ast$ of \eqref{e:optO} has the following properties: 
\begin{enumerate}
\item $\Omega^\ast$ is an open set; 
\item $\Omega^\ast$ has finite perimeter; 
\item $\Omega^\ast$ saturates the constraint, that is, $|\Omega^\ast|=m$; 
\end{enumerate}
The free boundary $\partial\Omega^\ast\cap D$ can be decomposed in the disjoint union of a regular part $\text{Reg}(\partial\O^\ast\cap D)$ and a singular part $\text{Sing}(\partial\O^\ast\cap D)$, where: 
\begin{enumerate}
\item[(4)] $\text{Reg}(\partial\O^\ast\cap D)$ is locally the graph of a $C^{1,\alpha}$-regular function for any $\alpha<1$; 
\item[(5)] for a universal constant $d^\ast\in\{5,6,7\}$ (see Definition \ref{def:dstar}), $\text{Sing}(\partial\O^\ast\cap D)$ is: 
\begin{itemize}
\item empty if $d<d^*$;
\item discrete if $d=d^*$; 
\item of Hausdorff dimension at most $(d-d^\ast)$ if $d>d^*$.\end{itemize}
\end{enumerate}
If the boundary $\partial D$ is $C^{1,1}$, then the boundary $\partial\Omega^\ast$ can be decomposed in the disjoint union of a regular part  $\text{Reg}(\partial\O^\ast)$ and a singular part $\text{Sing}(\partial\O^\ast)$, where: 
\begin{enumerate}
\item[(6)] $\text{Reg}(\partial\O^\ast)$ is an open subset of $\partial\Omega^\ast$ and locally the graph of a $C^{1,\sfrac12}$ function; moreover, $\text{Reg}(\partial\O^\ast)$ contains both $\text{Reg}(\partial\O^\ast\cap D)$ and $\partial\Omega^\ast\cap\partial D$; 
\item[(7)] $\text{Sing}(\partial\O^\ast)=\text{Sing}(\partial\O^\ast\cap D)$.
\end{enumerate}
\end{teo}

\begin{oss}[On the optimal regularity of the free boundary]
The regularity of the boundary of an optimal set $\O^\ast$ to the problem \eqref{e:optO} at contact points of the free boundary with the box cannot exceed $C^{1,\sfrac12}$ even if the vector field is smooth. Indeed, Chang-Lara and Savin proved in \cite{chang-lara-savin} that the boundary of $\Omega_u$, where $u$ is a solution of the free boundary problem \eqref{eq opt cond} in $\O_u=\O^\ast$, is at most $C^{1,\sfrac12}$ regular.
\end{oss} 

\begin{oss}[Regularity of the optimal shapes for variable vector field]\label{rem:th2-main}
We notice that if the couple $(\Omega,V)$ is a solution to the shape optimization problem \eqref{e:minOV} or \eqref{e:minOPhi}, then fixing $V$, we obtain that $\Omega$ is a solution to \eqref{e:optO}. In particular, the regularity part of Theorem \ref{t:th2} is a consequence of Theorem \ref{thm main}.
\end{oss}

%As a consequence of Theorem \ref{thm main} we get the regularity of every shape solution to the analogous problem where $V$ varies among all the vector fields which are the gradient of a Lipschitz function.

%\begin{coro}
%Let $D$ be an open, bounded and connected set in $\R^d$, $m \in (0,|D|)$ and $\tau>0$. Then the shape optimization problem
%\begin{equation}
%\min\Big\{\lambda_1(\Omega,\nabla\Phi)\ :\ \Omega\subset D\ \text{quasi-open},\ |\O|\leq m,\ \Phi\in W^{1,\infty}(D),\ \|\nabla\Phi\|_{L^\infty}\le\tau\Big\}
%\end{equation}
%admits a solution $(\Omega^\ast,\nabla\Phi^\ast)$. Moreover, $\O^\ast$ is an open set of locally finite perimeter such that $|\Omega^\ast|=m$. Furthermore, the boundary $\partial\Omega^\ast$ can be decomposed in the disjoint union of a regular set $\text{Reg}(\partial\O^\ast)$ and a singular set $\text{Sing}(\partial\O^\ast)$ with the properties stated in Theorem \ref{thm main}.
%\end{coro}

\subsection*{Outline of the proof and plan of the paper}
Throughout the paper the bounded open set $D\subset\R^d$ is fixed and is assumed to be (at least $C^{1,1}$) smooth. 

In the sections \ref{s:preliminaries}, \ref{s:def_lambda_1} and \ref{s:existence}, we prove our main existence results (Theorem \ref{t:th1}) and the existence of an optimal domain for a fixed vector field (Theorem \ref{thm main}), as well as the existence of an optimal domain in Theorem \ref{t:th2}. 

In Section \ref{s:preliminaries} we recall several central definitions and results in the $\gamma$-convergence theory of quasi-open sets. In particular, we show that the (classical) $\gamma$-convergence of a sequence of quasi-open sets is equivalent to the strong convergence of the sequence of resolvent operators for $L=-\Delta+V(x)\cdot \nabla$ on each of the sets.

 In Section \ref{s:def_lambda_1}, Theorem \ref{thm prop lambda1} and Corollary \ref{cor:lambda_1}, we prove that the principal eigenvalue is well-defined on every quasi-open set $\Omega\subset D$, that is, there exists a (real) eigenvalue $\lambda_1(\Omega,V)\in\R$ of the operator $L=-\Delta+V(x)\cdot\nabla$, such that for any other (complex) eigenvalue $\lambda\in \C$ we have $\lambda_1(\Omega,V)\le \text{Re}\,\lambda$. In the same section, we establish the continuity of the functional $\Omega\mapsto\lambda_1(\Omega,V)$ with respect to the $\gamma$-convergence (Proposition \ref{p:gamma_continuity}) and the fact that the principal eigenvalue is decreasing with respet to the set inclusion (Remark \ref{oss:monotonicity}).

 In Section \ref{s:existence} we prove our main existence results. The existence of the optimal set for a fixed vector field $V$ (Theorem \ref{thm main}) follows by the classical Buttazzo-Dal Maso theorem (Theorem \ref{t:budm}). We give the precise statement in Theorem \ref{t:optO}. The proof of Theorem \ref{t:th1} requires a more refined argument. The reason is the following: 
consider a (minimizing) sequence $(V_n,\Omega_n)$ of vector fields $V_n$ and quasi-open sets $\Omega_n$ with eigenfunctions $u_n\in H^1_0(\Omega_n)$ of $L_n=-\Delta +V_n\cdot \nabla $, solutions of 
$$-\Delta u_n+V_n\cdot\nabla u_n=\lambda_1(\Omega_n,V_n) u_n\quad\text{in}\quad\Omega_n,\qquad\int_D u_n^2\,dx=1,\qquad u_n\in H^1_0(\Omega_n).$$
Let us suppose for simplicity that: $\Omega_n$ $\gamma$-converge to a quasi-open set $\Omega$; $u_n$ converge to a function $u\in H^1_0(\Omega)$ both strongly $L^2(D)$ and weakly $H^1_0(D)$; $V_n$ converge weakly (in $L^2(D)$) to some $V\in L^\infty(D;\R^d)$. Now, the limit function $u$ solves a PDE in $\Omega$, which involves the (weak) limit of  the term $V_n\cdot \nabla u_n$, but a priori this might be different from $V\cdot \nabla u$.  In order to solve this issue, in Section \ref{s:existence}, we first prove that, on any fixed quasi-open set $\Omega$, there exists an optimal vector field (see Theorem \ref{thm min omega fixed}). We then replace the vector fields $V_n$ of the minimizing sequence $(V_n,\Omega_n)$ by the optimal vector field $V_n^\ast$ on each domain. Finally, we use the precise expression of $V_n^\ast$ to prove that the limit function $u$ is an eigenfunction of $-\Delta+V\cdot\nabla$ on $\Omega$ and we obtain that $\lambda_1(\Omega_n,V_n^\ast)$ converges to $\lambda_1(\Omega,V^\ast)$, which concludes the proof (see Theorem \ref{thm min on omega and v}). We cannot apply the same argument for Theorem \ref{t:th2}, since the optimal vector field might not be a gradient. On the other hand, for gradient vector fields the first eigenvalue is a variational functional, namely $$\lambda_1(\O,\nabla\Phi)=\displaystyle\min_{u\in H^1_0(\O)\setminus\{0\}}\frac{\int_D e^{-\Phi}|\nabla u|^2\,dx}{\int_D e^{-\Phi} u^2\,dx},$$ 
 and the existence of an optimal set can be obtained directly (see Theorem \ref{t:optPhi}).

\medskip

In Section \ref{s:regularity}, for a fixed drift $V=\nabla\Phi$, we prove the regularity of the optimal sets for $\lambda_1(\cdot,\nabla \Phi)$ (Theorem \ref{thm main}). In particular, this implies the regularity of the optimal sets in the case when both the set $\Omega$ and the vector field $\nabla \Phi$ may vary (see Theorem \ref{t:th2}). Our argument relies in an essential way on the variational formulation of $\lambda_1(\O,V)$. More precisely, we show (see Lemma \ref{l:fbp}) that, if $V=\nabla\Phi$ is fixed and $\O\subset D$ is a solution of \eqref{e:optO}, then the corresponding eigenfunction solves the free boundary problem
\begin{equation}\label{e:opteigen}
\min\Big\{\int_De^{-\Phi}|\nabla u|^2\,dx\ :\ u\in H^1_0(D), \ u\ge 0,\ \big|\{u\neq 0\}\big|\le m, \ \int_D e^{-\Phi}u^2\,dx=1\Big\}.
\end{equation}
This is a one-phase free boundary problem, similar to the one studied in the seminal paper of Alt and Caffarelli \cite{alt-caffarelli} on the local minimizers of the one-phase functional 
$$u\mapsto \int |\nabla u|^2\,dx+|\{u>0\}|.$$
Nevertheless, there are four differences with repect to the classical one-phase problem \cite{alt-caffarelli}.
\begin{enumerate}[(i)]
	\item the presence of the variable coefficient $e^{-\Phi}$ in the functional; 
		\item the presence of the integral constraint $\ds\int e^{-\Phi}u^2\,dx=1$;
	\item the presence of the measure constraint $|\{u>0\}|\le m$;  
\vspace{1.8mm}
	\item the presence of the inclusion constraint $\{u>0\}\subset D$ (equivalent to $u\in H^1_0(D)$). 
	\end{enumerate}

The variable coefficient $e^{-\Phi}$ introduces several technical difficulties, but does not have an influence on the overall strategy. The issues with the integral constraint are of similar nature. In fact, we are able to deal with this term (see Subsection \ref{sub:fbp} and Remark \ref{oss:jujv}) by reformulating the free boundary problem \eqref{e:opteigen} in terms of the functional
$$J(v): = \int_D |\nabla v|^2 e^{-\Phi}dx - \lambda_m \int_D v^2 e^{-\Phi}dx,$$
where $\lambda_m$ is the value of the minimum in  \eqref{e:opteigen}. In fact, one easily checks that, if $u$ is a solution of \eqref{e:opteigen}, then $u$ is also a solution to the free boundary problem
\begin{equation}\label{e:free_j}
J(u)\leq J(v)\qquad\text{for every}\qquad  v\in H^1_0(D)\qquad\text{such that}\qquad  \left\vert \O_v\right\vert\leq m,
\end{equation}
where, for any function $v$, we set  $\Omega_v:=\{v>0\}=\left\{x\in \Omega\,:\, v(x)>0\right\}$. 

The measure constraint in free boundary problems first appeared in the work of Aguilera, Alt and Caffarelli  \cite{aguilera-alt-caffarelli}. In fact, it is not hard to check that, at least formally, the solution $u$ should satisfy the optimality condition 
$$|\nabla u|=\sqrt{\Lambda_ue^{\Phi}}\quad\text{on the free boundary}\quad \partial\Omega_u\cap D,$$
where $\Lambda_u$ is a Lagrange multiplier formally arising in the minimization of the functional $J(u)$ under the constraint $|\Omega_u|=m$ (see Subsection \ref{sub:optimality_condition}). Thus, at least formally, there is no difference between the classical one-phase free boundary problem and the problem with a measure constraint. In practice, dealing with the measure constraint is an hardeous task. In fact, the Lagrange multiplier $\Lambda_u$ arises by applying internal variation to the function $u$, which by itself cannot be used to deduce even the basic qualitative properties of the solution $u$ as, for instance, the Lipschitz continuity and the non-degeneracy (in other words, at the moment, the regularity of the stationary free boundaries is not known). Our approach is different from the one in \cite{aguilera-alt-caffarelli} and is inspired by the works of Brian\c con-Lamboley \cite{briancon-lamboley} and Brian\c con \cite{briancon}.  In fact, we aim to tranform the problem \eqref{e:free_j} into 
\begin{equation}\label{e:free_j_2}
J(u)+\Lambda_u|\Omega_u|\leq J(v)+\Lambda_u|\Omega_v|\qquad\text{for every}\qquad  v\in H^1_0(D).
\end{equation}
Now, it is not possible to re-write \eqref{e:free_j} precisely in this form. Instead, we prove that  
\begin{equation}\label{e:free_j_3}
J(u)-J(v)\le \begin{cases}
(\Lambda_u+\eps)\big(|\Omega_v|-|\Omega_u|\big)\quad\text{for every}\quad  v\in H^1_0(D) \quad\text{such that}\quad |\Omega_v|\ge m;\\
(\Lambda_u-\eps)\big(|\Omega_v|-|\Omega_u|\big)\quad\text{for every}\quad  v\in H^1_0(D) \quad\text{such that}\quad |\Omega_v|\le m;
\end{cases}
\end{equation}
where the constant $\eps$ improves at small scales, that is, if we consider competitors $v$ that differ from $u$ only in a small ball of radius $r$, then $\eps$ can be chosen in a function of $r$, $\eps=\eps(r)$, which is such that $\eps(r)\to 0$ as $r\to 0$. In this part of the proof (Subsection \ref{sub:small_scales}) we follow the analysis of  \cite{briancon-lamboley}, except in one fundamental point. In fact, the approach of Brian\c con and Lamboley requires that the Lagrange multiplier $\Lambda_u$ is not vanishing, which is not a priori known (see Proposition \ref{p:lagrange}); in \cite{briancon-lamboley} the issue is solved by the method in \cite{briancon}. In this paper, we give a different argument to prove that the Lagrange multiplier is non trivial. Our approach is based on the Almgren monotonicity formula, and the fact that it implies the non-degeneracy of the solution $u$. We give the proof of the appendix, since the argument is very general (based only on the stationarity condition) and might be of independent interest.  We also notice that this simplifies the proof of \eqref{e:free_j_3}  and reduces it to three fundamental steps (see Theorem \ref{thm mu+-}).

Our proof of Theorem \ref{thm main} is general and can be applied to the classical one-phase problem \cite{alt-caffarelli}, to the one-phase problem with measure constraint \cite{aguilera-alt-caffarelli} and to shape optimization problems as for instance the one of \cite{briancon-lamboley}. Our approach is different from (and alternative to) the one of \cite{alt-caffarelli}, \cite{aguilera-alt-caffarelli} and \cite{briancon-lamboley}, as we do not use the regularity result of Alt and Caffarelli \cite{alt-caffarelli}. In fact, in order to prove the regularity of the flat free boundaries (Subsection \ref{sub:regularity}), we prove that the optimality condition on the free boundary holds in viscosity sense (see Lemma \ref{l:visc}) and then we apply the general results of De Silva \cite{de-silva} (for the regularity of the free boundary $\partial\Omega_u\cap D$) and the recent result of Chang-Lara and Savin \cite{chang-lara-savin} (for the regularity at the contact points $\partial\Omega_u\cap\partial D$). Finally, the estimate on the dimension of the singular set (Subsection \ref{sub:weiss}) is a consequence of the Weiss' (quasi-)monotonicity formula (Lemma \ref{l:weiss}).

\section{Preliminaries}\label{s:preliminaries}
In this section we recall the main definitions and the properties of the quasi-open sets, the $\gamma$-convergence and the weak-$\gamma$-convergence. 
\subsection{Capacity, quasi-open sets and quasi-continuous functions}\label{sub:quasi-open}$ $

The {\it capacity} of a set $E \subset \R^d$ is defined as
$$\cp(E):= \inf \big\{ \|u\|^2_{H^1} \ : \ u\in H^1(\R^d),\ u\geq 1 \text{ in a neighborhood of } E \big\},$$
where $H^1(\R^d)$ is the Sobolev space equipped with the norm 
$\ds \|u\|_{H^1}^2 = \int_{\R^d} \big(|\nabla u|^2+u^2\big)\,dx.$

We say that a property holds {\it quasi-everywhere} (q.e.) if it holds on the complementary of a set of zero capacity.

\medskip
%\begin{definition}\label{def qo func qc} [Quasi-open sets, quasi-continuous functions]
%\begin{enumerate}
%\item 
A set $\O \subset \R^d$ is said to be {\it quasi-open} if there exists a decreasing sequence  $(\om_n)_{n\geq1}$ of open sets such that, for every $n\ge 1$, $\Omega\cup\omega_n$ is an open set and 
$\ds \lim_{n\rightarrow \infty} \cp(\omega_n) = 0$.
	%$\ds \lim_{n\rightarrow \infty} cap(\om_n) = 0, \quad \text{and} \quad \forall n, \O \cup \om_n \text{ is open.}$

 A function $u : \R^d \rightarrow \R$ is said to be {\it quasi-continuous} if there exists a decreasing sequence $(\om_n)_{n\ge 1}$ of open sets such that $\ds \lim_{n\rightarrow \infty} \cp(\om_n)=0$ and the restriction of $u$ to $\R^d\setminus \om_n$ is continuous.

%\end{enumerate}
%\end{definition}

It is well-known (see for instance \cite[Theorem 1, Section 4.8]{evans-gariepy}) that, for every $u \in H^1(\R^d)$, there exists a quasi-continuous representative $\tilde{u}$ of $u$, which is unique up to a set of zero capacity. From now on we will identify a function $u\in H^1(\R^d)$ with its quasi-continuous representative. We note that, by definition of a quasi-open set and a quasi-continuous function, for every $u\in H^1(\R^d)$, the set $\O_{u} := \{u>0\}=\{ x \in \R^d \ | \ u(x)>0 \}$ is a quasi-open set (\cite[Proposition 3.3.41]{henrot-pierre}). On the other hand, for every quasi-open set $\Omega$, there exists a function $u \in H^1(\R^d)$ such that $\O = \O_u$ up to a set of zero capacity 
%(see Proposition \ref{prop w_O} below) 
that is, the quasi-open sets are superlevel sets of Sobolev functions.

\medskip
For any set $E\subset \R^d$, the Sobolev space $H^1_0(E)\subset H^1(\R^d)$ is defined as
$$H^1_0(E):= \big\{u \in H^1(\R^d) \ : \ u=0\ \text{ q.e. in }\ \R^d \setminus E \big\}.$$
Note that, whenever $E$ is open, this definition coincides with the usual definition of $H^1_0(E)$ as the closure of $C^\infty_c(E)$ with respect to the norm $\|\cdot\|_{H^1}$, $C^\infty_c(E)$ being the set of smooth functions compactly supported in $E$ (see for instance \cite[Theorem 3.3.42]{henrot-pierre}). 
For any set $E\subset\R^d$ there is a quasi-open set $\tilde E\subset\R^d$ such that $\cp(\tilde E\setminus E)=0$ and $H^1_0(\tilde E)=H^1_0(E)$. Roughly speaking, the quasi-open sets are the natural domains for the Sobolev space $H^1_0$. We notice that, for every quasi-open set $E$, $H^1_0(E)$ is a closed subspace of  $H^1(\R^d)$  ; if $E_1\subset E_2$ are two quasi-open sets, then $H^1_0(E_1)\subset H^1_0(E_2)$ and the   two sets $E_1$ and $E_2$   coincide q.e. if and only if $H^1_0(E_1)=H^1_0(E_2)$. 
%Then define $H^1_0(E)$ for any set $E \subset D$ as
%\[ H^1_0(E) = \{u \in H^1_0(D) \ | \ \tilde{u}=0 \text{ q.e. in } D \backslash E \}. \]
%Note that this definition coincides with the usual one for open sets (see \cite[Theorem 3.3.42]{Henrot Pierre}) and that if $\O \subset D$ is a quasi-open set, then $H^1_0(\O)$ is closed in $H^1_0(D)$.
%Moreover, two sets $E_1,E_2 \subset D$ such that the capacity of the symmetric difference $E_1 \Delta E_2$ is $0$ define the same space $H^1_0(E_1) = H^1_0(E_2)$.

\subsection{PDEs on quasi-open sets}\label{sub:quasi-open_pde}

Let $D \subset \R^d$ be a given open set and $\O \subset D$ be a quasi-open set of finite Lebesgue measure. For every quasi-open set $\Omega\subset D$ and every function $f\in L^2(\Omega)$,   the Lax-Milgram theorem and the Poincar\'e inequality ensure that   there is a unique solution $u\in H^1_0(\Omega)$ of the problem 
\[ -\Delta u = f\quad\text{in}\quad\Omega, \qquad u \in H^1_0(\O), \]
where the PDE is intended in the weak sense 
\[ \int_\O \nabla u \cdot \nabla \varphi \, dx = \int_\Omega f \varphi \,dx\ ,\quad\text{for every}\quad \varphi \in H^1_0(\O). \]
In particular, taking $u=\varphi$, we notice that $\|\nabla u\|_{L^2(\Omega)}\le \|f\|_{L^2(\Omega)}\|u\|_{L^2(\Omega)}$. Now since $\Omega$ has a finite Lebesgue measure, there is a constant $C_\Omega$ such that $\|u\|_{H^1}\le C_\Omega\|\nabla u\|_{L^2}$ for every $u\in H^1_0(\Omega)$. Thus, we get that $\|u\|_{H^1}\le C_\Omega\|f\|_{L^2}$. 

The resolvent operator $R^{-\Delta}_\O : L^2(D) \rightarrow L^2(D)$ is defined as $R^{-\Delta}_\O(f):=u$ and is a linear, continuous, self-adjoint, positive operator such that $R^{-\Delta}_\O(L^2(D))\subset H^1_0(\Omega)$. Moreover, thanks to the compact embedding $H^1_0(\O) \hookrightarrow L^2(\O)$, the resolvent $R^{-\Delta}_\O$ is also compact. 

The usual comparison and weak maximum principles hold in this setting. Precisely, we have:

$\bullet$ if $f \in L^2(D)$ is a positive function and $\O_1\subset\O_2 \subset D$ are two quasi-open   sets  , then $w_{\O_1} \leq w_{\O_2}$. 
	
$\bullet$ if $\Omega$ is a quasi-open set and $f,g \in L^2(\O)$ are such that $f \leq g$ in $\O$, then $R_{\O}^{-\Delta}(f) \leq R_{\O}^{-\Delta}(g)$.

\medskip

In the sequel we denote by $w_\O$ (and sometimes also by $R^{-\Delta}_\O(1)$) the solution of 
\[ -\Delta w_\O = 1\quad\text{in}\quad\Omega, \qquad \quad w_\O \in H^1_0(\O). \]
%Equivalently, we have $w_\O = R^{-\Delta}_\O(1)$. 
This function is sometimes called {\it torsion} or {\it energy} function and is useful, in particular, to define the topology of the $\gamma$-convergence on the family of quasi-open sets, which is the purpose of the next section.
In the following proposition we summarize the main properties of the function $w_\O$ (see for instance \cite[Proposition 3.50, Remark 3.53, Lemma 3.125, Proposition 3.72]{velichkov}).
\begin{prop}[Properties of the torsion function $w_\Omega$] \label{prop w_O}$ $
\begin{enumerate}
	%\item (Weak maximum principle). Let $f \in L^2(\O)$ be a positive function and $\O_1,\O_2 \subset D$ be two quasi-open sets such that $\O_1 \subset \O_2$. Then $w_{\O_1} \leq w_{\O_2}$. 
	
	%\noindent Moreover, if $f,g \in L^2(\O)$ are such that $f \leq g$ in $\O$, where $\O \subset D$ is a quasi-open set, then $w_\O^f = R_\O^{-\Delta}(f) \leq w_\O^g = R_\O^{-\Delta}(g)$.
	
\item There is a dimensional constant $C_d>0$ such that 
\begin{equation}\label{eq estim infty w}
\|\nabla w_\O\|_{L^2} \leq C_d |\O|^{\frac{d+2}{2d}} \qquad\text{and}\qquad  \|w_\O \|_{L^\infty} \leq C_d |\O|^{\sfrac2d}.
\end{equation} 
	
	\item Let $\O_1,\O_2 \subset D$ be two quasi-open sets. Then we have the estimate
	\begin{equation}\label{eq estim gamma dist}
	\int_D (w_{\O_1} - w_{\O_1 \setminus \O_2}) \, dx \leq \cp(\O_2)\,\|w_{\O_1}\|^2_{L^\infty(\O_1)} .
	\end{equation}
	
	\item $H^1_0(\O) = H^1_0(\{ w_\O>0 \})$. In particular, $\O = \{ w_\O>0 \}$ up to a set of zero capacity.
\end{enumerate}
\end{prop}
In the sequel we make the convention to extend to $D$ any vector field $V \in L^\infty(\O,\R^d)$ and any function $u \in H^1_0(\O)$ by letting it equal to $0$ on $D\setminus\O$ so that $V \in L^\infty(D,\R^d)$ and $u \in H^1_0(D)$. 

\medskip

We notice that, given a drift $V\in L^\infty(\O,\R^d)$, the bilinear form associated to the operator $L = -\Delta + V \cdot \nabla$ may not be coercive on $H^1_0(\O)$. Thus, in order to define the resolvent of $L = -\Delta + V \cdot \nabla$, we consider a large enough constant 
%We can not easily define a resolvent for the operator $L = -\Delta + v \cdot \nabla$ since its bilinear form is not coercive on $H^1_0(\O)$. However, for 
$c>0$ (depending only on   $\|V\|_{L^\infty(\Omega)}$),   for which there exists a positive constant $\delta>0$ such that
%for all $u \in H^1_0(\O)$, we have
\begin{equation}\label{eq coerci L'}
\delta  \int_D \big(|\nabla u |^2 + u^2) \,dx\le \int_D (|\nabla u |^2 + (V \cdot \nabla u) \,u  + c\, u^2\big) \,dx\ , \quad\text{for every}\quad u \in H^1_0(\O).
\end{equation}
The bilinear form associated to the operator $L'=L + c$ is hence coercive on $H^1_0(\O)$. Note that 
$$\text{if}\quad \|V\|_{L^\infty}\le \tau\ ,\quad\text{then we can take any}\quad 0<\delta<1\quad\text{and}\quad  c\ge\delta+\frac{\tau^2}{4(1-\delta)}.$$
%we can choose $c$ such that $\delta \in (1/2,1)$.
%Indeed, using $ab \leq \varepsilon a^2 + \frac{1}{\varepsilon}b^2$, $\varepsilon>0$, it follows that for all $u \in H^1_0(\O)$ we have
%\begin{align*}
%\int_D (|\nabla u|^2 + v \cdot \nabla u \ u) \ dx &\geq \int_D (|\nabla u |^2 - \tau |\nabla u | | u |) \ dx \\
%&\geq \int_D (|\nabla u |^2 - \tau (\varepsilon |\nabla u |^2 + \frac{1}{\varepsilon} u^2) ) \ dx\\
%&\geq (1-\tau \varepsilon)\int_D (|\nabla u |^2 + u^2) \ dx - ( \frac{\tau}{\varepsilon} +1-\tau \varepsilon ) \int_D u^2 \ dx.
%\end{align*}
%This proves \eqref{eq coerci L'} for $c = \frac{\tau}{\varepsilon} +1-\tau \varepsilon$ and $\varepsilon > 0$ small enough such that $\delta = 1-\tau \varepsilon > 0$.
Therefore, thanks to Lax Milgram theorem, we define the resolvent ${R^{L'}_\O : L^2(D) \rightarrow L^2(D)}$ as the compact (non self-adjoint) operator, which maps $f \in L^2(\O)$ to the unique solution of the problem 
\[ L'u = f \quad\text{in}\quad \O\ , \qquad u \in H^1_0(\O),\] 
which is intended in the weak sense 
\[ \int_\O \big(\nabla u \cdot \nabla \varphi +(V\cdot\nabla u )\varphi+c\, u\,\varphi\big)\, dx = \int_\Omega f \varphi \,dx\ ,\quad\text{for every}\quad \varphi \in H^1_0(\O). \]

\subsection{The $\gamma$-convergence and the weak-$\gamma$-convergence}\label{sub:gamma}
In this subsection we briefly recall the definition and the main properties of the $\gamma$-convergence of (quasi-open) sets.  
\begin{definition}[$\gamma$-convergence and weak-$\gamma$-convergence]
Let $D\subset\R^d$ be a given open set of finite Lebesgue measure, $(\O_n)_{n\geq 1}$ be a sequence of quasi-open sets and let $\O$ be a quasi-open set, all included in $D$. We say that
\begin{itemize}
\item $\O_n$ $\gamma$-converges to $\O$, if  $w_{\O_n}$ converges to $w_\O$ strongly in $L^2(D)$;
\item $\O_n$ weak-$\gamma$-converges to $\O$, if there exists $w\in H^1_0(D)$ such that $\O = \{w>0\}$ and   $w_{\O_n}$   converges to  $w$ in $L^2(D)$.
\end{itemize}
\end{definition} 

Though the $\gamma$-convergence is not compact on the family of quasi-open sets (see for instance \cite{cioranescu-murat} and \cite[$\S$ 3.2.6]{henrot-pierre} for an example), it is easy to see that the weak-$\gamma$-convergence is: by \eqref{eq estim infty w}, up to a subsequence, $w_{\O_n}$ weakly converges in $H^1_0(D)$ to some $w\in H^1_0(D)$ and hence $\O_n$ weak-$\gamma$-converges to the quasi-open set $\O := \{w>0\}$. To deal with the non-compactness of the $\gamma$-convergence we will use the following Lemma (see for example \cite{buttazzo} and \cite[Lemma 4.7.11]{henrot-pierre}).

\begin{lm}\label{lem sub seq gamma conv}
Let $\left(\O_n\right)_{n\geq 1} \subset D$ be a sequence of quasi-open sets that weak-$\gamma$-converges to the quasi-open set $\O\subset D$. 
Then there exists a subsequence of $\left(\O_n\right)_{n\geq 1}$, still denoted by $\left(\O_n\right)_{n\geq 1}$, and a sequence $(\tilde{\O}_n)_{n\geq 1} \subset D$ of 
quasi-open sets satisfying $\O_n \subset \tilde{\O}_n$, such that $\tilde{\O}_n$ $\gamma$-converges to $\O$.
\end{lm}

%We will also need the following lemmas (see \cite[Lemma 2.2.21]{velichkov} for Lemma \ref{lem semi cont Leb meas}).
The following lemma is a direct consequence of the definition of the weak-$\gamma$-convergence and the fact that for every quasi-open set $\Omega=\{w_\Omega>0\}$ (the detailed proof can be found for example in \cite{buttazzo} and \cite[Lemma 2.2.21]{velichkov}).   
\begin{lm}[Lower semi-continuity of the Lebesgue measure]\label{l:sc_meas}
Let $(\O_n)_{n\geq 1}$ be a sequence of quasi-open sets in $D$ weak-$\gamma$-converging to $\O\subset D$, then $\ds |\O| \leq \liminf_{n \rightarrow +\infty} |\O_n|. $
\end{lm}

  As was shown   in \cite{bucur-buttazzo} and \cite{buttazzo}, the following theorem, first proved in \cite{buttazzo-dal-maso}, is an immediate consequence of Lemma \ref{lem sub seq gamma conv} and Lemma \ref{l:sc_meas}. 
\begin{teo}[Buttazzo-Dal Maso \cite{buttazzo-dal-maso}]\label{t:budm}
Let $\mathcal F$ be a functional on the quasi-open sets, which is:

$\bullet$ decreasing with respect to the inclusion of sets; 

$\bullet$ lower semi-continuous with respect to the $\gamma$-convergence.

\noindent Then, for every bounded open set $D\subset\R^d$ and every $0<m\le |D|$, the shape optimization problem 
$$\min\Big\{\mathcal F(\Omega)\ :\ \Omega\ \text{quasi-open},\ \Omega\subset D,\ |\Omega|\le m\Big\}$$
has a solution.
\end{teo}
We will not be able to apply directly Theorem \ref{t:budm} to establish the existence of optimal sets for both the problems \eqref{e:minO} and \eqref{e:minOV} in the class of quasi-open sets. Instead, in Section \ref{s:existence}, we will use an argument based only on Lemma \ref{lem sub seq gamma conv} and Lemma \ref{l:sc_meas}, but before that we will need to extend the definition of $\lambda_1(\Omega,V)$ to the class of quasi-open sets. We do this in Section \ref{s:def_lambda_1}, where we will use several times the following approximation result. 
%In order to extend the definition of $\lambda_1(\Omega,V)$ to this class we will need the following result.
\begin{lm}[Approximation with open and smooth sets]\label{lem approx qo}
Let $\O \subset D$ be a quasi-open set. Then:

\noindent {\it(1)} there is a sequence of open sets   $\left(\O_n\right)_{n\geq 1}$   that $\gamma$-converges to $\O$ and is such that $\O \subset \O_n \subset D$   and $\ds\lim_{n\rightarrow +\infty}\left\vert \O_n\right\vert=\left\vert \O\right\vert$;  
    
\noindent {\it(2)} there is a sequence   $\left(\O_n\right)_{n\geq 1}$   of smooth ($C^\infty$) open sets contained in $D$, that $\gamma$-converges to $\O$.
%\end{enumerate}
\end{lm}   

\begin{proof}
The result is well-known; here we give the proof for the readers' convenience. 

\textit{(1)} Let   $\left(\om_n\right)_{n\geq 1}$   be a sequence of open sets such that $\lim_{n\to\infty}\cp(\omega_n)=0$ and $\O_n = (\O \cup \om_n) \cap D$ is an open set. Then, \eqref{eq estim gamma dist} applied to the sets $\O_n$ and $\om_n \setminus \O$ together with the second estimate in \eqref{eq estim infty w} show that $w_{\O_n}$ converges to $w_\O$ in $L^1(D)$. Moreover, up to a subsequence, $w_{\O_n}$ weakly converges in $H^1(D)$ thanks to the first estimate in \eqref{eq estim infty w}. Since the embedding $H^1_0(D) \hookrightarrow L^2(D)$ is compact, there is a subsequence which converges strongly in $L^2(D)$. By uniqueness of the limit in $L^1(D)$, it has to be $w_\O$. Thus, $w_{\O_n}$ converges in $L^2(D)$ to $w_\O$ and so, $\O_n$ $\gamma$-converges to $\O$.   Observe also that one has $\ds\lim_{n\rightarrow +\infty}\left\vert \O_n\right\vert=\left\vert \O\right\vert$ since $\ds\lim_{n\rightarrow +\infty} \left\vert \om_n\right\vert=0$.  
		
\textit{(2)} Firstly, assume that $\O$ is an open set.	
Let   $\left(\O_n\right)_{n\geq 1}$   be an increasing sequence of smooth open sets included in $\O$ which Hausdorff converges to $\O$. Then, up to a subsequence, $w_n := w_{\O_n}$ weakly converges in $H^1_0(D)$ to some $w \in H^1_0(D)$. But $\O_n,\O$ are open sets such that $\O_n \subset \O$, and since the convergence of $\O_n$ to $\O$ is Hausdorff, we can pass to the limit in the equation 
\[ -\Delta w_n = 1\quad \text{in}\quad \O_n \]
to see that $w$ satisfies 
\[ -\Delta w = 1 \quad\text{in}\quad \O. \]
This also shows that the sequence of norms $\|w_n\|_{H^1(D)}$ converges to $\|w\|_{H^1(D)}$, so that the convergence of $w_n$ to $w$ is strong in $H^1(D)$. Finally, since $\O_n \subset \O$, we get that $w \in H^1_0(\O)$ and hence that $w = w_\O$. Therefore, the sequence of smooth open sets $\O_n$ $\gamma$-converges to $\O$.
		
If now $\O$ is merely a quasi-open set, we can approximate $\O$ by a sequence of open sets which $\gamma$-converges to $\O$ thanks to {\it(1)}. Hence, by approximating these open sets by open smooth sets as above, we get a sequence of smooth open sets which $\gamma$-converges to $\O$. Recall that the topology of the $\gamma$-convergence is metrizable (see for example \cite{bucur-buttazzo}).
\end{proof}

\begin{oss}[The quasi-open sets cannot be $\gamma-$approximated with bigger smooth open sets]
In general, we cannot approximate a quasi-open set   (or even an open set)   $\O \subset D$ by a sequence of smooth (say of class $C^1$) open sets   $\left(\O_n\right)_{n\geq 1}$   which $\gamma$-converges to $\O$ and such that $\O_n \supset \O$.
Indeed,
%, as the following counter example shows. 	
let $(\xi_n)_{n\geq 1}$ be a dense sequence in $D=(0,1)^2 \subset \R^2$  and pick a sequence $(r_n)_{n\geq 1}$ of positive numbers such that $\sum_{n\geq1}\pi r_n^2 <1$. Set $\ds\O := \cup_{n \geq 1} B_{r_n}(\xi_n) \subset D$. We now claim that if $\O_n \supset \O$ is a smooth open set, then necessarily $\O_n \supset D$. To see this, let $x_0 \in D   \subset   \overline{\O} \subset \overline{\O}_n$. Then if $x_0 \in \partial \O_n$, there exist $r>0$ and a smooth, say of class $C^1$, function $f : \R^d \rightarrow \R$ such that, up to reorienting the axis, we have $\O_n \cap B_r(x_0) = \big\{x \in B_r(x_0) \ : \   x_d>f(x_1,\cdots,x_{d-1})  \big\}$. It follows that $B_r(x_0)\setminus \overline{\O}_n \subset D$ is   a nonempty   open set which does not intersect $\O_n$. This is in contradiction with $\O_n \supset \O$ since $\O$ is a dense open set in $D$. Hence $x \in \O_n$ and this shows that $D \subset \O_n$. Now, suppose that $\left(\Omega_n\right)_{n\geq 1}$   is a sequence of smooth sets such that $  D\supset\Omega_n\supset\Omega$.   Then $\Omega_n=D$ for every $n\ge 1$. Furthermore, the weak maximum principle implies $w_\O < w_D = w_{\O_n}$ in $D$, where the first inequality is strict since $|\O| < |D| = 1$. Therefore, $w_{\O_n}$ cannot strongly converge to $w_\O$ in $L^2(D)$.
%that is, $\O_n$ cannot be $\gamma$-convergent to $\O$.
\end{oss}

We now give a characterization of the $\gamma$-convergence in   terms   of convergence of resolvent operators.
The following theorem is a generalization of \cite[Lemma 4.7.3]{henrot-pierre} for the operator $L$.

\begin{teo}[$\gamma$-convergence and operator convergence]\label{thm charac gamma conv}
Let $D\subset \R^d$ be a bounded open set, $  \left(\Omega_n\right)_{n\geq 1}  \subset D$ be a sequence of quasi-open sets and $\Omega\subset D$ be a quasi-open set. The following assertions are equivalent :
%\begin{enumerate}

\noindent{\it(1)} the sequence   $\left(\O_n\right)_{n\geq 1}$   $\gamma$-converges to $\O$;

\noindent{\it(2)} for every sequence $  \left(f_n\right)_{n\geq 1}  \in L^2(D)$ weakly converging in $L^2(D)$ to $f \in L^2(D)$, the sequence $  \left(R^{L}_{\O_n}(f_n)\right)_{n\geq 1}$   converges to $R^{L}_\O(f)$ strongly in $L^2(D)$;

\noindent{\it(3)} the sequence of operators $  \left(R^{L}_{\O_n}\right)_{n\geq 1}   \in {\mathcal L}(L^2(D))$ converges to $R^{L}_\O$ in the operator norm $\|\cdot\|_{{\mathcal L}(L^2(D))}$.
%\end{enumerate}
\end{teo}

\begin{proof}
It is plain to see that the equivalence between \textit{(2)} and \textit{(3)} holds for all sequence of compact operators defined on Hilbert spaces. It then remains to prove that \textit{(1)} and \textit{(2)} are equivalent.

\noindent\textit{(1)}$\Rightarrow$\textit{(2)}.
Let $f_n \in L^2(D)$ be a sequence $L^2(D)$-weakly   converging   to $f \in L^2(D)$. Then $\|f_n\|_{L^2}$ is uniformly bounded. Moreover, writing $u_n = R^{L}_{\O_n}(f_n)$ we have
\[ \int_D f_nu_n \ dx = \int_D \big(|\nabla u_n|^2 + (V\cdot \nabla u_n) \, u_n +c u_n^2\big) \, dx. \]
Thanks to \eqref{eq coerci L'} this gives
\[ \frac{1}{2} \int_D (f_n^2 +  u_n^2 ) \, dx 
\geq \delta \int_D (|\nabla u_n|^2 + u_n^2 ) \, dx, \]
and therefore
\[ \int_D f_n^2 \ dx \geq (2\delta -1) \int_D(|\nabla u_n|^2 + u_n^2 ) \, dx. \]
Taking $\delta\in(\sfrac12,1)$, this shows that the sequence $\|u_n\|_{H^1(D)}$ is bounded.\par
  Assume now that the conclusion of \textit{(2)} does not hold. Then there exists $\varepsilon>0$ such that, up to a subsequence, $\left\Vert R^{L}_{\O_n}(f_n)-R^{L}_\O(f)\right\Vert_{L^2(D)}\geq \varepsilon$. Moreover,   up to a subsequence, $u_n$ weakly converges in $H^1(D)$ to some $u \in H^1_0(D)$, and therefore $g_n = f_n - V \cdot \nabla u_n - cu_n$ weakly converges in $L^2(D)$ to $g = f - V \cdot \nabla u - cu$. Theorem \ref{thm charac gamma conv} being true   for the Laplacian   (see \cite[Proposition 3.4]{buttazzo-velichkov}), we conclude that $R^{-\Delta}_{\O_n}(g_n)$ strongly converges in $L^2(D)$ to $R^{-\Delta}_\O(g)$. Thus $R^{L}_{\O_n}(f_n) = R_{\O_n}(g_n)$ and $R^{L}_{\O}(f) = R_{\O}(g)$ imply that $R^{L}_{\O_n}(f_n)$ strongly converges in $L^2(D)$ to $R^{L}_{\O}(f)$,   which yields a contradiction and therefore proves \textit{(2)}  . 

\noindent\textit{(2)}$\Rightarrow$\textit{(1)}. Let $(f_n)_{n\geq 1}\in L^2(D)$ be a sequence weakly converging in $L^2(D)$ to $f\in L^2(D)$. Set $w_n:=R^{-\Delta}_{\Omega_n}(f_n)$ and $w:=R^{-\Delta}_{\Omega}(f)$. We claim that $w_n\rightarrow w$ strongly in $L^2(D)$, which, according to \cite{buttazzo-velichkov} and \cite[Lemma 4.7.3]{henrot-pierre}, implies that $\Omega_n$ $\gamma$-converges to $\Omega$. Assume by contradiction that it is not the case, and pick up $\varepsilon>0$ and an increasing function $\varphi:\N^{\ast}\rightarrow \N^{\ast}$ such that 
\begin{equation} \label{wnw}
\left\Vert w_{\varphi(n)}-w\right\Vert_{L^2(D)}\geq \varepsilon\qquad\text{for every}\qquad n\ge 1.
\end{equation}
Since the sequence $(w_n)_{n\geq 1}$ is bounded in $H^1_0(D)$, up to a subsequence, there exists a function $z\in H^1_0(D)$ such that $w_{\varphi(n)}$ converges to $z$ weakly in $H^1_0(D)$   and   strongly in $L^2(D)$. 
%and 
%\begin{equation} \label{wzl2}
%w_{\varphi(n)}\rightarrow z\quad\mbox{strongly in}\quad L^2(D).
%\end{equation}
Now, since
$$
L\, w_n=f_n+V\cdot \nabla w_n+cw_n:=g_n\quad\mbox{ in }\quad\Omega_n,
$$
and $w_n\in H^1_0(\Omega_n)$, $w_n=R^L_{\Omega_n}(g_n)$. But $g_{\varphi(n)}\rightharpoonup g:=f+V\cdot \nabla z+cz$ weakly in $L^2(D)$, so that, by assumption $(2)$, $w_{\varphi(n)}\rightarrow R^{L}_{\Omega}(g)$ strongly in $L^2(D)$. Then the convergence of $w_{\varphi(n)}$ to $z$   yields   that $z=R^{L}_{\Omega}(g)$, is a solution of 
$$
L\,z=f+V\cdot \nabla z+cz\quad\mbox{ in }\quad\Omega,\qquad z\in H^1_0(\Omega),
$$
or, in other words, 
%$$-\Delta z=f\quad\mbox{ in }\quad\Omega.$$
%Finally, 
$z=R^{-\Delta}_{\Omega}(f)=w$. Thus, \eqref{wnw} provides a contradiction, therefore showing that $w_n\rightarrow w$ strongly in $L^2(D)$, which means that {\it(1)} holds.
\end{proof}

\section{The principal eigenvalue on quasi-open sets}\label{s:def_lambda_1}
%\subsection{The principal eigenvalue on open sets}
For a bounded domain $\O\subset \R^d$ and $V\in L^{\infty}(\Omega,\R^d)$, the principal eigenvalue $\lambda_1(\Omega,V)$, of the (non self-adjoint) elliptic operator $L = -\Delta + V \cdot \nabla$ on $\O$ with Dirichlet boundary condition on $\partial\Omega$, was defined in \cite{berestycki-nirenberg-varadhan} by
$$\lambda_1(\O,V)=\sup\big\{\lambda\in\R\ :\  \exists \phi\in W^{2,d}(\Omega)\quad\text{such that}\quad \phi>0\quad\text{and}\quad -L\phi+\lambda\phi\le0\quad \text{in}\quad \Omega\big\},$$
where it was proved that $\lambda_1(\Omega,V)\in\R$ has the following properties:
%(Theorems 2.1 and 2.3)  
%\begin{itemize}

    $(i)$ There is a positive eigenfunction $u:\Omega\to\R$ such that $u \in W^{2,p}_{loc}(\O)$, for all $p\in [1,+\infty)$, and 
    $$Lu = \lambda_1(\O,V)u\quad\text{in}\quad \Omega,\qquad u\in H^1_0(\Omega),\qquad \int_\Omega u^2\,dx=1,$$
    (see \cite[Theorem 2.1]{berestycki-nirenberg-varadhan}).
    
    $(ii)$ $\lambda_1(\O,V) < \text{Re}\,(\lambda)$ for every eigenvalue $\lambda \neq \lambda_1(\O,V)$ of $L$ in $\O$ (see \cite[Theorem 2.3]{berestycki-nirenberg-varadhan}).
    %the inequality being strict in the case when $\Omega$ is connected.
    
    $(iii)$ The functional $\Omega\mapsto\lambda_1(\Omega,V)$ is decreasing with respect to the domain inclusion. 
%\end{itemize}

\medskip

\noindent In the sequel we extend the definition of $\lambda_1(\Omega,V)$ to quasi-open sets. 
We first recall that the definition can be extended to an arbitrary open set $\O\subset D$ by
\[ \lambda_1(\O,V) = \inf \lambda_1(\O_n,V), \]
where the infimum is taken over all the connected component $\O_n$ of $\O$.
Now, in view of property $(iii)$ above, for any quasi-open set $\O \subset D$, we define
\begin{equation}\label{eq def lambda1}
\lambda_1(\O,V) := \sup \big\{ \lambda_1(\tilde{\O},V) \ : \ \tilde{\O} \text{ open,} \ \O \subset \tilde{\O} \subset D \big\}.
\end{equation} 
\begin{oss}
Notice that, these two definitions coincide for open sets. 
\end{oss}

\begin{oss}\label{oss:monotonicity}
The functional $\O\mapsto \lambda_1(\O ,V)$, defined on the family of quasi-open sets, is still non-increasing with respect to the set inclusion, that is $\lambda_1(\Omega_2,V)\le \lambda_1(\Omega_1,V)$, whenever $\Omega _1\subset\Omega_2$.
%$$\lambda_1(\Omega_2,V)\le \lambda_1(\Omega_1,V)\quad\text{whenever}\quad \Omega _1\subset\Omega_2. $$
\end{oss}
We will show that $\lambda_1(\O,V)$ is finite and is an eigenvalue of $L$ in $\Omega$ satisfying the minimality property $(ii)$. 
Recall that, for a quasi-open set of finite Lebesgue measure $\Omega\subset D$, we say that $\lambda\in \C$ is an eigenvalue of the operator $L=-\Delta+V\cdot \nabla$ in $\Omega$ if there is  an eigenfunction $u:\R^d\to \C$, (weak) solution to the problem 
\begin{equation}\label{e:eigenvalue}
-\Delta u+V\cdot\nabla u=\lambda u\quad\text{in}\quad\Omega,\qquad u\in H^1_0(\Omega;\C),\qquad \int_\Omega |u|^2\,dx=1.
\end{equation}
Let now $c>0$ be the constant from Subsection \ref{sub:quasi-open_pde} and $L' = L + c$. 
Note that $\lambda\in \C$ is an eigenvalue of $L$ in $\O$, if and only if, $\lambda+c$ is an eigenvalue of $L'$ in $\O$. By the argument from  Subsection \ref{sub:quasi-open_pde}, we have that the bilinear form associated to the operator $L'$ is coercive and so, $R_{\O}^{L'}$ is a compact operator on $L^2(D)$. In particular, the spectrum is   a discrete set   of eigenvalues with no accumulation points except zero and $\lambda\in\C$ is an eigenvalue of $L$ in the sense of \eqref{e:eigenvalue} if and only if $(\lambda+c)^{-1}$ is an eigenvalue of $R_\Omega^{L'}$.

The following theorem shows that most of the properties of the principal eigenvalue on an open set still hold for $\lambda_1(\O,V)$ if $\O \subset D$ is merely a quasi-open set.

\begin{teo}[Definition of the principal eigenvalue on quasi-open sets]\label{thm prop lambda1}
Let $D$ be a bounded open set, $V \in L^\infty(D,\R^d)$ and $\O \subset D$ be a non-empty quasi-open set. Then

%{\it(1)} The functional $\O\mapsto \lambda_1(\O,v)$ is decreasing with respect to the set inlcusion.
{\it (1)} $\lambda_1(\O,V)$ is well-defined that is, $\lambda_1(\O,V)<+\infty$. 

{\it (2)} $\lambda_1(\O,V)$ is an eigenvalue of $L$ in $\O$; there is a (non-trivial)   real-valued   eigenfunction $u$ such that 
$$Lu = \lambda_1(\O,V)u\quad\text{in}\quad \Omega\ ,\qquad u \in H^1_0(\O),\qquad \int_\Omega u^2\,dx=1.$$ 

{\it (3)} If $\lambda \in \C$ is an eigenvalue of $L$ in $\O$, then $\lambda_1(\O,V) \leq \text{Re}\,(\lambda)$.
%\end{enumerate}
\end{teo}
In order to prove Theorem \ref{thm prop lambda1} we will need the following two lemmas. The key estimate for the proof of Theorem \ref{thm prop lambda1} {\it(1)} is contained in the following lemma inspired by \cite[Proposition 5.1]{berestycki-nirenberg-varadhan}.
\begin{lm}\label{l:lambda_bound}
Let $V \in L^\infty(D,\R^d)$ and $\O \subset D$ be an open set. Suppose that there is $\tau>0$ such that $\|V\|_{L^\infty(\Omega)}\le \tau< 2\sqrt{\lambda_1(\Omega,V)}$. Then 
\begin{equation}\label{eq comp lambda1}
\lambda_1(\O,0) \geq \lambda_1(\O,V) - \tau \sqrt{\lambda_1(\O,V)}.
\end{equation}
\end{lm}
\begin{proof}
%Let $\O_n$ be a maximizing sequence for \eqref{eq def lambda1}. We can assume that $\O_n$ $\gamma$-converges to $\O$. We argue by contradiction and suppose that $\lambda_1(\O_n,V)$ tends to $+\infty$ as $n$ goes to $+\infty$. The proof is a consequence of the following claim: for $n$ large enough, we have the estimate

%where we have set $\tau = \|V\|_{L^\infty}$. Here $\lambda_1(\O_n,0)$ is the first eigenvalue of $-\Delta$ in $\O_n$. Indeed, from this inequality, it follows that $\lambda_1(\O_n,0)$ tends to $+\infty$. But, by Theorem \ref{thm charac gamma conv}, the sequence $R_{\O_n}^{-\Delta}$ converges to $R_{\O}^{-\Delta}$ in norm of operator. This hence implies (because $-\Delta$ is self-adjoint, see \cite[Theorem 4.10]{Kato}) the convergence of the spectrum, that is, $\lambda_1(\O_n,0) \rightarrow \lambda_1(\O,0)$ as $n$ goes to $+\infty$. In particular, $\lambda_1(\O_n,0)$ is bounded, which is a contradiction.

%We reason as in \cite[Proposition 5.1]{berestycki nirenberg varadhan}. 
Let us first suppose that $\Omega$ is connected. For convenience, set $\lambda := \lambda_1(\O,V)$. By the definition of the first eigenvalue of $-\Delta$ on domains, it is enough to find some $\phi>0$ in $\O$ such that $-\Delta \phi \geq (\lambda -\tau \sqrt{\lambda})\phi$ in $\O$. Since $\O$ is an open set, from \cite[Theorem 2.1]{berestycki-nirenberg-varadhan}, there exists a positive eigenfunction $\phi_V$ for the first eigenvalue of $L$ in $\O$, that is, $\phi_V>0$ in $\O$ and $L \phi_V= \lambda \phi_V$. Set $\phi := \phi_V^\alpha$ for some $\alpha \in (0,1)$ to be chosen later. Then, in $\Omega$, we have
\begin{align*}
-\Delta \phi - \lambda \phi &= -\alpha (\Delta \phi_V) \phi_V^{\alpha-1} - \alpha(\alpha-1)|\nabla \phi_V|^2 \phi_V^{\alpha-2} - \lambda \phi_V^\alpha \\
&= \left[ \lambda(\alpha-1) - \alpha V \cdot \frac{\nabla \phi_V}{\phi_V} + \alpha(1-\alpha)\frac{|\nabla \phi_V|^2}{\phi_V^2} \right] \phi_V^\alpha \\
&\geq \left[ \lambda(\alpha-1) - \alpha \tau \frac{|\nabla \phi_V|}{\phi_V} + \alpha(1-\alpha)\frac{|\nabla \phi_V|^2}{\phi_V^2} \right] \phi_V^\alpha.
\end{align*}
The function $x \mapsto -\alpha \tau x + \alpha(1-\alpha)x^2$   reaches its minimum   at $x = \tau/(2(1-\alpha))$. Therefore, we get
\[ -\Delta \phi - \lambda \phi \geq \left[ \lambda(\alpha-1) - \alpha \frac{\tau^2}{4(1-\alpha)} \right] \phi_V^\alpha=\left[ \lambda(\alpha-1) - \alpha \frac{\tau^2}{4(1-\alpha)} \right] \phi. \]
Since $\alpha \in (0,1)$ is arbitrary, we can choose it so that it maximizes the term in the brackets of the above estimate, that is, such that $1-\alpha = \tau /(2 \sqrt{\lambda})$. Note that, by hypothesis   on $\tau$,   we have $\alpha \in (0,1)$. It follows
\[  -\Delta \phi - \lambda \phi \geq \left[-\tau \sqrt{\lambda} + \frac{\tau^2}{4}\right] \phi \geq -\tau \sqrt{\lambda}\, \phi, \]
which proves the claim in the case when $\Omega$ is connected.

  In the general case,   let $(\Omega_n)_{n\ge1}$ be the connected components of $\Omega$. Then, for every $V$, we have 
$$\lambda_1(\Omega,V)=\inf_{n}\lambda_1(\Omega_n,V).$$
%Suppose, without loss of generality, that $\lambda_1(\Omega,0)$ is realized on $\Omega_1$ that is, $\lambda_1(\Omega,0)=\lambda_1(\Omega_1,0)$. 
Then, we have,   for all $n$,  
$$  \lambda_1(\Omega_n,0)\ge \lambda_1(\Omega_n,V)-\tau \sqrt{\lambda_1(\Omega_n,V)}\ge \lambda_1(\Omega,V)-\tau \sqrt{\lambda_1(\Omega,V)},  $$
where the last inequality is due to the fact that $x\mapsto x-\tau\sqrt{x}$ is a   non-increasing   function on the interval $[\lambda_1(\Omega,V),+\infty)$. 
\end{proof}

The next lemma is a direct consequence of the classical result \cite[Theorem 3.16]{kato} on the convergence of a spectrum of closed operators   with suitable properties.   We will use it in the proof of Theorem \ref{thm prop lambda1} {\it (3)}.
\begin{lm}[Convergence of the spectra]\label{l:kato}
Let $H$ be a separable Hilbert space and $  \left(T_n\right)_{n\geq 1}   \in\mathcal L(H)$ a sequence of compact operators converging to the compact operator $T\in\mathcal L(H)$ in the operator norm $\|\cdot\|_{\mathcal L(H)}$. Suppose that   $\lambda\in\C\setminus\{0\}$   is an (isolated) eigenvalue of $T$ and let $r>0$ be such that   $B_r(\lambda)\cap \sigma(T)=\left\{\lambda\right\}$.   Then, there is $n_0\ge 1$ such that for every $n\ge n_0$ there is an eigenvalue $\lambda_n\in \sigma(T_n)\cap B_{r/2}(\lambda)$. 
\end{lm}

We are now in position to prove Theorem \ref{thm prop lambda1}.

\begin{proof}[Proof of Theorem \ref{thm prop lambda1}]
Consider a maximizing sequence $ \left(\Omega_n\right)_{n\geq 1}$   for \eqref{eq def lambda1}, that is, a sequence of open sets $  \left(\Omega_n\right)_{n\geq 1}$   such that
$$\lambda_1(\Omega,V)=\lim_{n\to\infty}\lambda_1(\Omega_n,V)\qquad\text{and}\qquad \Omega\subset\O_n \subset D\quad\text{for every}\quad n\ge 1.$$

\noindent {\it We first show that we can assume that $\O_n$ $\gamma$-converges to $\O$.} Let $\om_n$ be a sequence of open sets such that $\Omega\cup\omega_n$ is open and $\cp(\omega_n)\to0$. We set $\tilde{\O}_n := \O_n \cap (\O \cup \om_n)=\Omega\cup(\omega_n\cap\Omega_n)$. By \eqref{eq def lambda1} and the inclusion  $\Omega\subset\tilde{\O}_n\subset\Omega_n$ we have $\lambda_1(\O_n,V) \leq \lambda_1(\tilde{\O}_n,V) \leq \lambda_1(\O,V)$, so we get
$$\lambda_1(\Omega,V)=\lim_{n\to\infty}\lambda_1(\tilde\Omega_n,V)\qquad\text{and}\qquad \Omega\subset\tilde\O_n \subset D\quad\text{for every}\quad n\ge 1.$$
Thus, we may consider $\tilde\Omega_n$ in place of $\Omega_n$ as a maximizing sequence for \eqref{eq def lambda1}.
Finally, as in Lemma \ref{lem approx qo}, $\tilde{\O}_n$ $\gamma$-converges to $\O$ thanks to the estimate \eqref{eq estim gamma dist} applied to the sets $\tilde{\O}_n$ and $\O_n \cap \om_n$. 
\smallskip

\noindent {\it We now prove claim (1).} Indeed, suppose by contradiction that 
$$\lambda_1(\Omega,V)=\lim_{n\to\infty}\lambda_1(\Omega_n,V)=+\infty.$$ Then, by Lemma \ref{l:lambda_bound} we have that 
$$\lim_{n\to\infty}\lambda_1(\Omega_n,0)=+\infty.$$
Now, since $\Omega\mapsto\lambda_1(\Omega,0)$ is decreasing and $\Omega\subset\Omega_n$, we get that $\lambda_1(\Omega,0)=+\infty$. By the variational characterization 
$$\lambda_1(\Omega,0)=\min_{u\in H^1_0(\Omega)\setminus\{0\}}\frac{\int_\Omega|\nabla u|^2\,dx}{\int_\Omega u^2\,dx},$$
we get that $H^1_0(\Omega)=\{0\}$, which implies that $\Omega=\emptyset$ (or, equivalently, $\cp\Omega=0$),   which is absurd.  
\smallskip

\noindent {\it We now prove (2).} Let $u_n\in H^1_0(\Omega_n)\subset  H^1_0(D)$ be the (normalized) eigenfunction associated to $\lambda_1(\Omega_n,V)$. Then we have 
$$L' u_n=\big(\lambda_1(\Omega_n,V)+c\big)u_n\quad\text{in}\quad\Omega_n,\qquad u_n\in H^1_0(\Omega_n),\qquad\int_{\Omega_n}u_n^2\,dx=1.$$
Multiplying the above equation by $u_n$, integrating over $\Omega_n$ and using the estimate \eqref{eq coerci L'} we get 
$$\delta \|u_n\|_{H^1}^2\le \lambda_1(\Omega_n,V)+c\quad\text{for every}\quad n\ge 1.$$
In particular, since $\lambda_1(\Omega, V)<\infty$, we get that $  \left(u_n\right)_{n\geq 1}$   is uniformly bounded in $H^1_0(D)$ and so, up to a subsequence, we may assume that $u_n$ converges, weakly in $H^1_0(D)$ and strongly in $L^2(D)$, to a function $u\in H^1_0(D)$. Moreover, $\Omega_n$ $\gamma$-converges to $\Omega$ and so, $R_{\Omega_n}^{L'}$ converges in norm to $R_\Omega^{L'}$. Thus, 
$$u=\lim_{n\to\infty}u_n=\lim_{n\to\infty}(\lambda_1(\Omega_n,V)+c)R_{\Omega_n}^{L'}(u_n)=(\lambda_1(\Omega,V)+c)R_{\Omega}^{L'}(u),$$ 
which concludes the proof of {\it (2)}.
\smallskip

\noindent {\it Proof of (3).} Suppose that $\lambda\in \C$ is an eigenvalue of $L$ on $\Omega$ such that $\text{Re}(\lambda)<\lambda_1(\Omega,V)$. Then, $(\lambda+c)^{-1}\in \C$ is a (non-zero) eigenvalue of the compact operator $R_{\Omega}^{L'}$. Applying Lemma \ref{l:kato}, we can assume that for $n$ large enough, there is an eigenvalue $\lambda_n$ of $L$ on $\Omega_n$ such that $\text{Re}(\lambda_n)<\lambda_1(\Omega_n,V)$, which is a contradiction with   \cite[Theorem 2.3]{berestycki-nirenberg-varadhan}  . 
\end{proof}
\begin{oss}[On the sign of the first eigenfunction]\label{rk u non neg}
In particular, as a consequence of the proof of Theorem \ref{thm prop lambda1} {\it (2)}, there is an eigenfunction $u$ of $L$ on the quasi-open set $\Omega$, which is non-negative, being the limit of non-negative functions. 
We notice that $u$ does not need to be strictly positive as $\Omega$ might be disconnected. 
\end{oss}

We conclude this section with a proposition on the continuity of $\lambda_1(\cdot,V)$ with respect to the $\gamma$-convergence. 

\begin{prop}[$\gamma$-continuity of $\lambda_1(\cdot,V)$]\label{p:gamma_continuity}
Let $D\subset \R^d$ be a bounded open set, $V\in L^{\infty}(D;\R^d)$ be a fixed vector field, and $  \left(\Omega_n\right)_{n\geq 1}\subset D$   be a sequence of quasi-open sets that $\gamma$-converges to the quasi-open set $\Omega\subset D$. Then 
$$\lambda_1(\Omega,V)=\begin{cases}\ds\lim_{n\to\infty}\lambda_1(\Omega_n,V),\quad\text{if}\quad \Omega\neq\emptyset,\\
+\infty,\quad\text{if}\quad \Omega=\emptyset.
\end{cases}$$
\end{prop}
\begin{proof}
Let $\tau=\|V\|_{L^\infty(D)}$ and $\delta$ and $c$ be as in \eqref{eq coerci L'}.   Set $L'=L+c$.  

\noindent {\it   Suppose first that the sequence $\left(\lambda_1(\Omega_n,V)\right)_{n\geq 1}$ is bounded. }   Reasoning as in the proof of Theorem \ref{thm prop lambda1} {\it (2)} we get that, up to a subsequence, $\lambda_1(\Omega_n,V)$ converges to an eigenvalue $\lambda\in\R$ of $L$ on $\Omega$. Now, by the argument of Theorem \ref{thm prop lambda1} {\it (3)} and Lemma \ref{l:kato}, we have that $\lambda$ satisfies the property {\it(3)} of Theorem \ref{thm prop lambda1}, so $\lambda=\lambda_1(\Omega,V)$, which concludes the proof   since the sequence $\left(\lambda_1(\Omega_n,V)\right)_{n\geq 1}$ is bounded  . 

\noindent {\it   Next, suppose that the sequence $\left(\lambda_1(\Omega_n,V)\right)_{n\geq 1}$ is unbounded.  } Applying Lemma \ref{l:lambda_bound}, we get that,   up to a subsequence,   $\ds\lim_{n\to\infty}\lambda_1(\Omega_n,0)=+\infty$. Since $R_{\Omega_n}^{-\Delta}$ are self-adjoint compact operators, we get that 
$$\lim_{n\to\infty}\|R_{\Omega_n}^{-\Delta}\|_{\mathcal L(L^2(D))}=\lim_{n\to\infty}\frac{1}{\lambda_1(\Omega_n,0)}=0.$$
Finally, the $\gamma$-convergence gives that $R_\Omega^{-\Delta}(\Omega)\equiv 0$ and so, $H^1_0(\Omega)=\{0\}$ and $\cp(\Omega)=0$.  
\end{proof}
\begin{oss}
In view of Proposition \ref{p:gamma_continuity} we set $\lambda_1(\emptyset,V)=+\infty$. 
\end{oss}
\noindent Putting together Theorem \ref{thm prop lambda1} and Proposition \ref{p:gamma_continuity} we obtain the following result.
\begin{coro}[Equivalent definition of the principal eigenvalue on quasi-open sets]\label{cor:lambda_1}
Let $\Omega$ be a bounded quasi-open set and $V\in L^\infty(\Omega;\R^d)$. Then, there is an eigenvalue $\lambda_1(\Omega,V)\in\R$ of $L=-\Delta+V\cdot\nabla$ in $\Omega$ such that:
\begin{equation*}
\begin{array}{rl}
\lambda_1(\Omega, V)&\ds=\min\big\{\text{Re}\,\lambda\ :\ \lambda\in \C\ \text{is an eigenvalue of}\ L\ \text{on}\ \Omega\big\}\\
&\ds=\sup\big\{\lambda_1(\tilde\Omega)\ :\ \tilde \Omega\ \text{is an open set containing}\ \Omega\big\}\\
&\ds=\!\!\lim_{n\to\infty}\lambda_1(\Omega_n,V),\ \text{where}\  \left(\Omega_n\right)_{n\geq 1}\ \text{is any sequence}\\
&\qquad\qquad\qquad\qquad\qquad\qquad\text{ of (smooth) open sets $\gamma$-converging to}\ \Omega.  
\end{array}
\end{equation*}
\end{coro}
\begin{proof}
The first two inequalities are due to Theorem \ref{thm prop lambda1}. For the third one it is sufficient to note that for every quasi-open set $\Omega$ there is a sequence of smooth open sets $\gamma$-converging to $\Omega$   and to   apply Proposition \ref{p:gamma_continuity}.
\end{proof} 
\begin{oss}[Faber-Krahn with drift for quasi-open sets]\label{oss:fk_quasi-open} As further consequence of Corollary \ref{cor:lambda_1} we can extend the Hamel-Nadirashvili-Russ inequality to the class of (bounded) quasi-open sets. Precisely, for every bounded quasi-open set $\Omega\subset\R^d$   with $\left\vert \Omega\right\vert>0$   and every $\tau>0$, we have
\begin{equation}\label{e:hnr_quasi-open}
\lambda_1\left(B,\tau\frac{x}{|x|}\right)\le \lambda_1(\Omega,V)\quad\text{for every}\quad V\in L^{\infty}(\Omega;\R^d)\quad\text{with}\quad \|V\|_{L^\infty}\le\tau,
\end{equation}
where $B$ is the ball centered in zero of the same Lebesgue measure as $\Omega$. Indeed, let $\O \subset \R^d$ be a   bounded   quasi-open set and $V \in L^\infty(\Omega,\R^d)$ be such that $\|V\|_{L^\infty} \leq \tau$ (in what follows we assume that $V$ is extended by zero outside $\Omega$). Let $ \left(\O_n\right)_{n\geq 1}$   be a sequence of bounded open sets which $\gamma$-converges to $\O$ and such that $|\O_n|$ converges to $|\O|$ (see   Lemma{lem approx qo}).   Denote by $B_{r_n}$ (resp. $B$) the ball centred at $0$ whose Lebesgue measure is $|B_{r_n}| = |\O_n|$ (resp. $|B|=|\O|$). Then, since $\O_n$ is an open set, we have $\lambda_1(B_{r_n},\tau e_r) \leq \lambda_1(\O_n,v)$ thanks to \cite[Remark 6.10]{hamel-nadirashvili-russ-arxiv}. Moreover, $B_{r_n}$ $\gamma$-converges to $B$ (since $|B_{r_n}| \rightarrow |B|$ and hence $B_{r_n}$ converges to $B$ in the sense of Hausdorff; see \cite[Proposition 3.4.2]{henrot-pierre}). Therefore, Corollary \ref{cor:lambda_1} implies that $\lambda_1(B_{r_n},\tau e_r)$ converges to $\lambda_1(B,\tau e_r)$ and similarly, $\lambda_1(\O_n,V)\to\lambda_1(\O,V)$. Passing to the limit we get \eqref{e:hnr_quasi-open}.
%that $\lambda_1(B,\tau e_r) \leq \lambda_1(\O,V)$.

\end{oss}

\section{Existence of optimal domains}\label{s:existence}

In this section we prove the existence of optimal domains for the cost functional $\lambda_1(\Omega,V)$. We first consider the case when the drift $V$ is fixed, for which the existence follows by the result of the previous section and a classical theorem in shape optimization. The case when both the domain $\Omega$ and the drift $V$ may vary requires more careful analysis and the rest of the section is dedicated to the proof of Theorem \ref{thm min on omega and v}. In the end of the section (Theorem \ref{t:optPhi}) we also prove that a solution $(\Omega,V)$ exists also in the class of vector fields $V$ obtained as gradients of Lipschitz continuous functions.
\begin{teo}[Existence of optimal sets for a fixed vector field]\label{t:optO}
Let $D\subset\R^d$ be a bounded open set and $V\in L^\infty(D;\R^d)$. Then, for every $0< m\le |D|$, there is an optimal domain, solution of the problem \eqref{e:optO}. 
%\begin{equation}\label{e:optO}
%\min\Big\{\lambda_1(\Omega,V)\ :\ \Omega\subset D\ \text{quasi-open},\ |\Omega|\le m \Big\}. 
%\end{equation}
\end{teo}
\begin{proof}
By Remark \ref{oss:monotonicity} and Proposition \ref{p:gamma_continuity} we get that $\Omega\mapsto\lambda_1(\Omega,V)$ is $\gamma$-continuous and decreasing with respect to the set inclusion. The claim follows by Theorem \ref{t:budm}.
\end{proof}

\subsection{Optimal drifts on a fixed domain}\label{sub:existence_V}
Let $\Omega \subset \R^d$ be a fixed bounded quasi-open set and $\tau>0$ be given. We consider the following variational minimization problem
\begin{equation}\label{eq min omega fixed}
\min \Big\{ \lambda_1(\O,V) \ : \ V \in L^\infty(\O,\R^d), \ \|V\| _{L^\infty} \leq \tau \Big\}.
\end{equation}

\begin{teo}[Optimal vector field on a fixed quasi-open set]\label{thm min omega fixed}
The problem \eqref{eq min omega fixed} has a solution,   which satisfies   
\begin{equation}\label{e:Vopt}
V_\ast(x) =-\tau \frac{\nabla u(x)}{|\nabla u(x)|}\ \text{ if }\ |\nabla u(x)| \neq 0\ ;\qquad V_\ast(x) =0\ \text{ if }\ |\nabla u(x)| = 0\ ,
%V_\ast(x) = \begin{cases}
%\ds-\tau \frac{\nabla u(x)}{|\nabla u(x)|}\ \text{ if }\ |\nabla u(x)| \neq 0,\\
%\qquad0\qquad \text{ if }\ |\nabla u(x)| = 0,\end{cases}
\end{equation}
where $u$ is the eigenfunction of $L=-\Delta+V_\ast\cdot\nabla$ in $\O$, associated to the eigenvalue $\lambda_1(\O,V_\ast)$.
\end{teo}
\begin{proof}
Let $ \left(\O_n\right)_{n\geq 1}$   be a sequence of smooth, say of class $C^{2,\alpha}$ for some $0<\alpha<1$, open sets which $\gamma$-converges to $\O$ (see Remark \ref{lem approx qo}). 
Since $\O_n$ is smooth, we already know (see \cite[theorem 1.5]{hamel-nadirashvili-russ-annals}) that the problem \eqref{eq min omega fixed} for the fixed domain $\Omega_n$ has a solution $V_n $. Moreover, if $u_n$ is the associated eigenfunction of $-\Delta+V_n\cdot\nabla$ in $\O_n$, that is, $u_n$ is defined by
\[ -\Delta u_n + V_n \cdot \nabla u_n 
%=-\Delta u_n-\tau\left\vert \nabla u_n\right\vert
= \lambda_1(\O_n,V_n) u_n\quad \text{in}\quad \O_n, \qquad u_n \in H^1_0(\O_n), \qquad \int_{\Omega_n}u_n^2\,dx=1,\]
then the optimal vector field $V_n$ is unique and is given by
\[ V_n(x) = \left\{ 
	\begin{array}{ll}
		\ds-\tau \frac{\nabla u_n(x)}{|\nabla u_n(x)|} &\text{ if } |\nabla u_n(x)| \neq 0 ,\\
		0 &\text{ if } |\nabla u_n(x)| = 0.
	\end{array}
\right. \]
In particular, $u_n$ is a solution of 
\[ -\Delta u_n-\tau\left\vert \nabla u_n\right\vert
= \lambda_1(\O_n,V_n) u_n\quad \text{in}\quad \O_n, \qquad u_n \in H^1_0(\O_n), \qquad \int_{\Omega_n}u_n^2\,dx=1.\]
We first claim that the sequence $ \left(\lambda_1(\O_n,V_n)\right)_{n\geq 1}$   is bounded. Indeed, by optimality of $V_n$, one has $\lambda_1(\O_n,V_n)\leq \lambda_1(\O_n,0)$, which is nothing but the principal eigenvalue of $-\Delta$ on $\O_n$ with Dirichlet boundary condition. But since $\O_n$ $\gamma$-converges to $\O$, Proposition \ref{p:gamma_continuity} yields that $\lambda_1(\O_n,0)\rightarrow \lambda_1(\O,0)$ so that the sequence $\left(\lambda_1(\O_n,0)\right)_{n\geq 1}$ is bounded, proving our claim.\par

\noindent   Therefore, up to a subsequence, $\lambda_1(\O_n,V_n)$ converges to some $\lambda \in \R$ and $u_n$ has a uniformly bounded norm   in $H^1_0(D)$,   which yields a function $u\in H^1_0(D)$ such that, up to a subsequence, 
%$u_n$ converges to $u$ weakly in $H^1_0(D)$ and strongly in $L^2(D)$.
\begin{equation} \label{convun}
u_n\rightharpoonup u\mbox{ weakly in } H^1_0(D)\mbox{ and }u_n\rightarrow u\mbox{ strongly in }L^2(D).
\end{equation} 
Since the sequence $|\nabla u_n|$ is bounded in $L^2(D)$, up to a subsequence, $-\tau\left\vert \nabla u_n\right\vert\rightharpoonup z$ weakly in $L^2(D)$ for some function $z\in L^2(D)$. Therefore, $f_n:= \lambda_1(\O_n,v_n) u_n + \tau |\nabla u_n|$ weakly converges in $L^2(D)$ to $f:=\lambda u-z$. Thanks to theorem \ref{thm charac gamma conv} (applied to $-\Delta$), $u_n = R^{-\Delta}_{\O_n}(f_n)$ strongly converges in $L^2(D)$ to $R^{-\Delta}_\O(f)$. By \eqref{convun}, we have $u = R^{-\Delta}_\O(f)$ and hence $u \in H^1_0(\O)$.  Furthermore
\begin{align*}
\int_D |\nabla u|^2 \ dx &=  \displaystyle \int_D (-zu + \lambda u^2 ) \ dx \\
& =   \lim_{n\rightarrow +\infty}  \int _D (\tau|\nabla u_n|u_n + \lambda_1(\O_n,V_n) u_n^2) \ dx =   \lim_{n\rightarrow +\infty} \int_D |\nabla u_n|^2 \ dx,
\end{align*}
where the first line is due to the fact that $u\in H^1_0(\Omega)$ and $-\Delta u=\lambda u-z$ in $\Omega$. This proves that $u_n$ converges strongly in $H^1(D)$ to $u$,
% and, in particular, that $u$ is not zero. This also proves 
that $|\nabla u_n|$ strongly converges in $L^2(D)$ to $|\nabla u|$, and hence that $z= -\tau |\nabla u|$. Therefore $u$ satisfies
\[ -\Delta u + V_\ast \cdot \nabla u =-\Delta u-\tau\left\vert \nabla u\right\vert= \lambda u\quad \text{in}\quad \O, \qquad u \in H^1_0(\O), \qquad\int_\Omega u^2\,dx=1,\]
where $V_\ast \in L^\infty(D,\R^d)$ is given by \eqref{e:Vopt}.
%\[ v^\ast(x) = \left\{ 
%	\begin{array}{ll}
%		-\tau \frac{\nabla u(x)}{|\nabla u(x)|} &\text{ if } |\nabla u(x)| \neq 0 \\
%		0 &\text{ if } |\nabla u(x)| = 0,
%	\end{array}
%\right. \]
This shows that $\lambda$ is an eigenvalue of the operator $L = -\Delta +V_\ast \cdot \nabla$ in $\O$. In particular, we have $\|V_\ast\|_\infty \leq \tau$ and $\lambda_1(\Omega,V_{\ast})\leq \lambda$. On the other hand, by the minimality of $V_n$, we have $\lambda_1(\O_n,V_n) \leq \lambda_1(\O_n,V_\ast)$. Hence, letting $n\rightarrow \infty$, we get that $\lambda \leq \lambda_1(\O,V_\ast)$, which   yields   $\lambda=\lambda_1(\Omega,V_\ast)$ and concludes the proof of the theorem.
\end{proof}

\subsection{Shape optimization problem over domains and vector fields}\label{sub:existence_O_V}

Let $D \subset \R^d$ be a bounded open set, $0<m\le |D|$ and $\tau>0$. 
We consider the shape optimization problem 
\begin{equation}\label{eq min on omega and v 2}
\min\big\{\lambda_1(\O,V) \ : \ \O \subset D \ \text{quasi-open}, \ |\O| \leq m,  \ \|V\|_{L^\infty} \leq \tau \big\}.
\end{equation}

\begin{teo}[Existence of optimal sets and optimal vector fields]\label{thm min on omega and v}
Let $\tau\geq 0$ and $m \in (0,|D|)$. Then the problem \eqref{eq min on omega and v 2} has a solution   $(\O^\ast,V^\ast)$, where $V^\ast$ is given by \eqref{e:Vopt}.  
% and $u$ is the eigenfunction of $L$ in $\O^\ast$ associated to $\lambda_1(\O^\ast,v^\ast)$, then $v^\ast(x) = -\tau \nabla u(x) / |\nabla u(x)|$ if $|\nabla u(x)| \neq 0$ and $v^\ast(x) = 0$ otherwise.
\end{teo}

\begin{proof} 
Let $(\Omega_n,V_n)$ be a minimizing sequence for \eqref{eq min on omega and v 2} and let 
\begin{align*}
\underline{\lambda}&:=\inf\big\{\lambda_1(\O,V) \ : \ \O \subset D \ \text{quasi-open}, \ |\O| \leq m,  \ \|V\|_{L^\infty} \leq \tau \big\}=\lim_{n\to\infty}\lambda_1(\Omega_n,V_n),
\end{align*} 
%and let $\left(\lambda_1(\O_n,v_n)\right)_{n\geq 1}$ be a minimizing sequence for \eqref{eq min on omega and v 2}, that is, a sequence which converges to $\underline{\lambda}$ with $\O_n$ quasi-open sets such that $|\O_n| \leq m$ and $\|v_n\|_\infty \leq \tau$.
Since the topology of the weak $\gamma$-convergence is compact, we can assume that, up to a subsequence, $\O_n$ weakly $\gamma$-converges to a quasi-open set $\O \subset D$. 
Then, let $\tilde{\O}_n$ be a sequence of quasi-open sets as in Lemma \ref{lem sub seq gamma conv}. Denote by $\tilde V_n$ the optimal vector field given by Theorem \ref{thm min omega fixed} on $\tilde{\O}_n$, and let $u_n \in H^1_0(\tilde{\O}_n)$ be   a solution   of
$$-\Delta u_n + \tilde V_n \cdot \nabla u_n = \lambda_1(\tilde{\O}_n,\tilde V_n) u_n \quad \text{in}\quad \tilde{\O}_n,\qquad u_n\in H^1_0(\tilde\Omega_n),\qquad \int_D u_n^2\,dx=1.$$ 
By the minimality of $\tilde V_n$ and the inclusion $\O_n \subset \tilde{\O}_n$, we have 
\[ 0 <\lambda_1(\tilde{\O}_n,\tilde V_n) \leq \lambda_1(\tilde{\O}_n, V_n) \leq\lambda_1(\O_n, V_n)\qquad\text{for every}\quad n\ge 1. \]
Therefore, up to a subsequence, $\lambda_1(\tilde{\O}_n,\tilde V_n)$ converges to some $\tilde{\lambda}$ such that $\tilde{\lambda} \leq \underline{\lambda}$. In particular, $ \left(u_n\right)_{n\geq 1}$   is uniformly bounded in $H^1_0(D)$ and so, up to a subsequence, $u_n$ weakly converges in $H^1_0(D)$ to some $u \in H^1_0(D)$.  Now, since $\tilde{\O}_n$ $\gamma$-converges to $\O$, we can argue as in the end of the proof of Theorem \ref{thm min omega fixed} to conclude that the convergence of $u_n$ to $u$ is strong in $H^1(D)$. This yields that $u$ is not identically zero and satisfies
\[ -\Delta u + V \cdot \nabla u =\Delta u-\tau\left\vert \nabla u\right\vert= \tilde{\lambda}\, u\quad \text{in}\quad \O , \qquad u \in H^1_0(\O),\qquad \int_D u^2\,dx=1, \]
where $V \in L^\infty(D,\R^d)$ is given by \eqref{e:Vopt}.
%\[ V(x) = \left\{ 
%	\begin{array}{ll}
%		-\tau \frac{\nabla u(x)}{|\nabla u(x)|} &\text{ if } |\nabla u(x)| \neq 0 \\
%		0 &\text{ if } |\nabla u(x)| = 0.
%	\end{array}
%\right. \]
Furthermore, thanks to Lemma \ref{l:sc_meas}, we have that $|\O| \leq m$. Hence, $\underline{\lambda}\leq \tilde{\lambda}$. Thus, we get that $\tilde{\lambda} = \underline\lambda$ and hence that $\underline{\lambda}= \lambda_1(\O,V)$, which proves that the couple $(\Omega,V)$ is a solution of \eqref{eq min on omega and v 2}. 
%This concludes the proof. 
\end{proof}

\begin{oss}
If the box $D$ contains a ball $B \subset D$ such that $|B| = m$, then by Remark \ref{oss:fk_quasi-open} a solution of \eqref{eq min on omega and v 2} is given by $\lambda_1(B,\frac{\tau x}{|x|})$.
\end{oss}

We now consider a shape optimization problem in the more restrictive class of couples $(\Omega,V)$, in which the vector field $V$ is a gradient of a Lipschitz function. Precisely, given a bounded open set $D\subset\R^d$, $\tau\geq 0$ and $m \in (0,|D|)$, we consider the shape optimization problem
\begin{equation}\label{e:optPhi}
\min\Big\{\lambda_1(\Omega,\nabla\Phi)\ :\ \Omega\subset D\ \text{quasi-open},\ \Phi\in W^{1,\infty}(D),\   \left\vert \Omega\right\vert\leq m,  \ \|\nabla\Phi\|_{L^\infty}\le\tau\Big\}.
\end{equation}
In this case the argument from Theorem \ref{thm min on omega and v} does not apply since the optimal vector field from Theorem \ref{thm min omega fixed} may not be the gradient of a Lipschitz function. On the other hand, the functional $\lambda_1(\Omega,\nabla\Phi)$ is variational so we can use a more direct approach. Indeed, for every $\lambda\in\R$ and $u\in H^1_0(\Omega)$ we have 
\begin{equation*}
-\Delta u+\nabla\Phi\cdot \nabla u=\lambda u\quad \text{in} \quad \O \quad\Leftrightarrow\quad -\text{div}\,(e^{-\Phi}\nabla u)=\lambda e^{-\Phi}u\quad\text{in}\quad\Omega\,, 
\end{equation*}
and since the operator $A=-\text{div}\,(e^{-\Phi}\nabla \cdot)$ is self-adjoint, we get that 
\begin{equation}\label{e:lambdaPhi}
\lambda_1(\Omega,\nabla\Phi)=\min_{u\in H^1_0(\Omega)\setminus\{0\}}\frac{\int_D e^{-\Phi}|\nabla u|^2\,dx}{\int_D e^{-\Phi} u^2\,dx}.
\end{equation}

\begin{teo}[Existence of optimal sets and optimal potentials]\label{t:optPhi}
Let $D\subset\R^d$ be a bounded open set, $\tau\geq 0$ and $m \in (0,|D|)$. Then the problem \eqref{e:optPhi} has a solution.
\end{teo}
\begin{proof}
Suppose that $(\Omega_n, \Phi_n)$ is a minimizing sequence for \eqref{e:optPhi} and let $\lambda_n=\lambda_1(\Omega_n,\nabla\Phi_n)$. Given $x_0\in D$, we may suppose that $\Phi_n(x_0)=0$ for every $n\ge 1$. Thus, up to a subsequence, $\Phi_n$ converges uniformly in $\overline D$ to a function $\Phi\in W^{1,\infty}(D)$ such that $\Phi(x_0)=0$ and $\|\nabla \Phi\|_{L^\infty}\le \tau$. Let $u_n$ be the solution of 
$$-\Delta u_n+\nabla\Phi_n\cdot\nabla u_n=\lambda_nu_n\quad\text{in}\quad\Omega_n,\qquad u_n\in H^1_0(\Omega_n),\qquad \int_D u_n^2\,dx=1.$$
Then, $u_n$ is uniformly bounded in $H^1_0(D)$ an so, up to a subsequence, $u_n$ converges weakly in $H^1_0(D)$ and strongly in $L^2(D)$ to a function $u\in H^1_0(D)$. Thus, we have 
$$\int_De^{-\Phi}u^2\,dx= \lim_{n\to\infty} \int_De^{-\Phi_n} u_n^2\,dx\qquad\text{and}\qquad \int_De^{-\Phi}|\nabla u|^2\,dx\le \liminf_{n\to\infty}\int_De^{-\Phi_n}|\nabla u_n|^2\,dx.$$
Now, choosing $\Omega:=\{u>0\}$ and applying \eqref{e:lambdaPhi}, we get
$$\lambda_1(\Omega,\Phi)\le \frac{\int_D e^{-\Phi}|\nabla u|^2\,dx}{\int_D e^{-\Phi} u^2\,dx}\le \liminf_{n\to\infty} \frac{\int_D e^{-\Phi_n}|\nabla u_n|^2\,dx}{\int_D e^{-\Phi_n} u_n^2\,dx}= \liminf_{n\to\infty}\lambda_1(\Omega_n,\Phi_n).$$
Now, in order to conclude, it is sufficient to notice that by choosing a subsequence, we may assume that $u_n$ converges to $u$ pointwise a.e., so we get 
$$|\Omega|=|\{u>0\}|\le \liminf_{n\to\infty} |\{u_n>0\}| \le \liminf_{n\to\infty}|\Omega_n|\le m,$$
which proves that $(\Omega,\Phi)$ is a solution of \eqref{e:optPhi}.
\end{proof}

\section{Regularity of the optimal sets}\label{s:regularity}
In this section we prove Theorem \ref{thm main}. We prove the regularity of the boundary $\partial\Omega$ of the optimal sets $\Omega$ from Theorem \ref{thm main}. We only consider the case $V=\nabla \Phi$, with $\Phi\in W^{1,\infty}(D)$, since in this case the optimization problem \eqref{e:optO} is equivalent to a free boundary problem for the first eigenfunction $u$ on the optimal set $\Omega$. The regularity for a generic vector field $V\in L^\infty(D)$ remains an open problem essentially due to the lack of variational characterization of the eigenvalue $\lambda_1(\Omega,V)$. 
We start with the following lemma. 
\begin{lm}[Reduction to a free boundary problem]\label{l:fbp}
Let $D\subset\R^d$ be a bounded open set, $0<m<|D|$, $\tau>0$, $\Phi\in W^{1,\infty}(D)$, with $\|\nabla \Phi\|_{L^\infty}\le \tau$, and $V=\nabla\Phi$. Suppose that the quasi-open set $\Omega\subset D$ is a solution of \eqref{e:optO}. Then every corresponding first eigenfunction $u_\Omega$ of the operator $-\Delta+V\cdot \nabla$ on $\Omega$ is a solution to the variational problem 
\begin{equation}\label{e:fbp}
\lambda_m:=\min\Big\{\int_D|\nabla u|^2 e^{-\Phi}\,dx\ :\ u\in H^1_0(D),\ \big|\{u\neq 0\}\big|\le m, \ \int_D e^{-\Phi}u^2\,dx=1\Big\}.
\end{equation} 
Conversely, if $u$ is a solution of \eqref{e:fbp}, then the quasi-open set $\{u\neq 0\}$ is a solution of \eqref{e:optO}.   
\end{lm}
\begin{proof}
The proof is a straightforward consequence of the variational formula \eqref{e:lambdaPhi}.
\end{proof}
\begin{oss}
It turns out that if $u$ is a solution of \eqref{e:fbp}, then $u\geq 0$ in $D$ (see Lemma \ref{l:saturation} below). 
\end{oss}

The rest of this section is dedicated to the regularity of the free boundary $\partial\Omega_u\cap D$ and of the whole boundary $\partial\O_u$ if $D$ is smooth, of a solution $u$ of \eqref{e:fbp}, where   we recall that, for any function $v\in H^1_0(D)$ we denote by $\Omega_v$ the (quasi-open) set $\{v>0\}$. 

\medskip

This section is organized as follows. 

In Subsection \ref{sub:boundedness} we prove that the solutions of \eqref{e:fbp} are bounded. This is important due to the fact that $u$ solves the equation 
$$-\text{div}\,(e^{-\Phi} \nabla u)=\lambda_1(\Omega_u,\nabla\Phi)e^{-\Phi}u\quad\text{in}\quad\Omega_u,$$
and in the rest of the section we will often use the fact that the right-hand side is bounded.  

In Subsection \ref{sub:pointwise}, we prove that the solution $u$ is essentially a subharmonic function on $D$ with respect to the operator $\text{div}(e^{-\Phi}\nabla)$ (see Lemma \ref{l:radon}). In particular, this implies that $u$ and the set $\Omega_u$ are well-defined everywhere (not just up to a set of measure zero). The free boundary is thus defined as the topological boundary of the set $\Omega_u$. 
In the same subsection, in Lemma \ref{l:saturation}, we prove that the measure constraint $|\{u>0\}|\le m$ is saturated, that is, $|\Omega_u|=m$. This proves Theorem \ref{thm main} (3). 

In Subsection \ref{sub:fbp} we get rid of the integral constraint $\ds\int e^{-\Phi}u^2\,dx=1$ and we rewrite the problem \eqref{e:fbp} in terms of the functional 
$$J(u)=\int_D\big(|\nabla u|^2-\lambda_m u^2\big) e^{-\Phi}\,dx.$$

In Subsection \ref{sub:optimality_condition}, we write the Euler-Lagrange equation that arises in the  minimization of the functional $J$ under the measure constraint $|\Omega_u|\le m$. We consdider only internal variations, that is, test functions of the form $\tilde u(x)=u(x+t\xi(x))$ for smooth vector fields $\xi$.
In Subsection \ref{sub:small_scales}, we prove that at small scales $u$ is a solution (in the sense of \eqref{e:free_j_3})  to the minimization problem for the functional $v\mapsto J(v)+\Lambda_u|\Omega_v|$, where $\Lambda_u$ is the Lagrange multiplier from Subsection \ref{sub:optimality_condition}.

\medskip

In Subsection \ref{sub:lipschitz} and Subsection \ref{sub:nondegeneracy}, we use the result from Subsection \ref{sub:small_scales} to prove that the solutions of \eqref{e:fbp} are Lipschitz continuous and  non-degenerate at the free boundary; we also prove that the set $\Omega_u$ has finite perimeter. This proves Theorem \ref{thm main} (1) and (2). 

\medskip

Subsection \ref{sub:blowup} is dedicated to the compactness of the blow-up sequences and the optimality of the blow-up limits. 
In Subsection \ref{sub:weiss} we prove a (quasi-)monotoncity formula for (a variant of) the Weiss' boundary adjusted energy. As a consequence, we obtain that the blow-up limits are one-homogeneaous. 

\medskip

In Subsection \ref{sub:regularity}, we prove that the solution $u$ satisfies the optimality condition $|\nabla u|=\sqrt{\Lambda_u e^{\Phi}}$ on the free boundary $\partial \Omega_u\cap D$ in viscosity sense, we deduce the regularity of the regular part $Reg(\partial\Omega_u)$ and we show that the remaining singular set has zero  $(d-1)$-dimensional Hausdorff measure. In this subsection, we complete the proof of Theorem \ref{thm main} (4), (6) and (7).
Finally, in Subsection \ref{sub:weiss}, we give some further estimates on the Hausdorff dimension of the singular set, which complete the proof of Theorem \ref{thm main} (5).

\subsection{Boundedness of the eigenfunctions}\label{sub:boundedness}
In this subsection we give a bound on the $L^\infty$ norm of the eigenfunctions on generic bounded quasi-open sets. We first prove that if $u$ is a solution of a PDE with sufficiently integrable right-hand side, then $u$ is bounded. Then we use and iterate an interpolation argument to improve the integrability of the eigenfunctions.

\begin{lm}[Boundedness of the solutions of PDEs on quasi-open sets]\label{infttybndflem}
Let $D\subset\R^d$ be a bounded open set, $\Omega\subset D$ be a quasi-open set and $\Phi\in W^{1,\infty}(D)$. Let $f\in L^p(D)$ for some $p\in (d/2,+\infty]$ and let $u\in H^1_0(\Omega)$ be the solution of %\eqref{e:eqf}.
\begin{equation}\label{e:eqf}
-\dive\,(e^{-\Phi} \nabla u)=f\quad\text{in}\quad\Omega,\qquad u\in H^1_0(\Omega). 
\end{equation}
Then, there is a dimensional constant $C_d$ such that
%, for every $t\ge 0$, we have  
$$\|u\|_{L^\infty}\le \frac{C_d \,e^{\max\Phi}}{\sfrac2d-\sfrac1p}|\Omega|^{\sfrac2d-\sfrac1p}\|f\|_{L^p},$$
where $\max\Phi=\|\Phi\|_{L^\infty(D)}$.
%$$\|(u-t)_+\|_{L^\infty}\le \frac{C_d \,e^{\max\Phi}}{\sfrac2d-\sfrac1p}\|f\|_{L^p}|\{u>t\}|^{\sfrac2d-\sfrac1p}\qquad\text{for every}\qquad t>0.$$
\end{lm}
\begin{proof}
We first assume that $f$ is a non-negative function.
We notice that $u\ge 0$ on $\Omega$ and that $u$ is a minimum in $H^1_0(\Omega)$ of the functional
$$J(u):=\frac12\int_\Omega e^{-\Phi}|\nabla u|^2\,dx-\int_\Omega fu\,dx.$$ 
The rest of the proof follows precisely as in \cite[Lemma 3.51]{velichkov}. 
For every $0<t<\|u\|_{L^\infty}$ and $\eps>0$, we consider the test function $u_{t,\eps}=u\wedge t +(u-t-\eps)_+.$
The inequality $J(u)\le J(u_{t,\eps})$ gives that 
\begin{equation*}\label{inftybound130504e1}
\begin{array}{ll}
\ds\frac12\int_{\{t< u\le t+\eps\}}e^{-\Phi}|\nabla u|^2\,dx\le \int_{\R^d} f\left(u-u_{t,\eps}\right)\,dx\le\eps\int_{\{u>t\}}f\,dx\le \eps\|f\|_{L^p}|\{u>t\}|^{\frac{p-1}p},
\end{array}
\end{equation*}
and, using the co-area formula and passing to the limit as $\eps\to 0$, we get 
\begin{equation}\label{e:coarea1}
\int_{\{u=t\}}|\nabla u|\,d\HH^{d-1}\le  2e^{\max \Phi}\|f\|_{L^p}|\{u>t\}|^{\frac{p-1}p}.
\end{equation}
Now, setting $\vf(t):=|\{u>t\}|$ and using   the co-area formula again as well as the Cauchy-Schwarz inequality,   we obtain
\begin{equation*}
\ds \vf'(t)=-\int_{\{u=t\}}\frac{1}{|\nabla u|}\,d\HH^{d-1}\le -\left(\int_{\{u=t\}}|\nabla u|\,d\HH^{d-1}\right)^{-1}Per(\{u>t\})^2,
\end{equation*}
which, together with the isoperimetric inequality $|\{u>t\}|^{\frac{d-1}d}\le C_d Per(\{u>t\})$ and \eqref{e:coarea1}, gives
\begin{equation*}
\varphi'(t)\le -\frac{C_d}{e^{\max\Phi}\|f\|_{L^p}}\vf(t)^{\frac{d-2}{d}+\frac1p}\,,
\end{equation*}
where we recall that the dimensional constant $C_d$ may change from line to line.

\noindent Setting $\alpha=\frac{d-2}{d}+\frac1p<1$ and $C=C_d\|f\|_{L^p}^{-1}e^{-\max\Phi}$, we have $\varphi^{\prime}\leq -C\varphi^{\alpha}$. If 
$$t_{\max}:=\sup\left\{t>0;\ \varphi(s)>0\mbox{ for all }s\in [0,t)\right\}\leq +\infty,$$ 
then $\varphi^{\prime}(t)\varphi(t)^{-\alpha}\leq -C$ for all $t\in [0,t_{\max})$, so that 
$$0\leq \varphi(t)\leq \big(|\Omega|^{1-\alpha}-(1-\alpha)Ct\big)^{\frac{1}{1-\alpha}}\qquad\text{for all}\qquad t\in [0,t_{\max}).$$
%for all $t\in [0,t_{\max})$. 
This shows that $t_{\max}<+\infty$ and that 
%we consider the ODE 
%\begin{equation}\label{e:ode}
%y'(t)=-Cy(t)^\alpha\quad \text{for}\quad t>0,\qquad y(0)=|\Omega|,
%\end{equation}
%whose solution is given by $y(t)=\big(|\Omega|^{1-\alpha}-(1-\alpha)Ct\big)^{\frac{1}{1-\alpha}}$ and is such that $y(t)\ge \phi(t)\ge 0$ for every $t\ge 0$. Thus, there is some $t_{\max}$ such that $\phi(t)=0$, for every $t\ge t_{\max}$, and so we get 
$$\|u\|_{L^\infty}\le t_{\max}\le  \frac1C\frac{|\Omega|^{\sfrac2d-\sfrac1p}}{\sfrac2d-\sfrac1p},$$
which concludes the proof when $f$ is non-negative. For a general function $f$, the proof now follows by applying the estimate in Lemma \ref{infttybndflem} to both the positive and the negative parts of $f$.
\end{proof}

\begin{lm}[Boundedness of the eigenfunctions]\label{l:boundedness}
Let $D\subset\R^d$ be a bounded open set, $\Omega\subset D$ be a quasi-open set, $\Phi\in W^{1,\infty}(D)$ and $V=\nabla\Phi$. Let $R:L^2(\Omega)\to L^2(\Omega)$ be the resolvent operator of $-\Delta+V\cdot\nabla$ on $\Omega$. Then, there are constants $n\in \N$, depending only on $d$, and $C\in\R$, depending on $d$, $|\Omega|$ and $\|\Phi\|_{L^\infty}$, such that 
$$R^n (L^2(\Omega))\subset L^\infty(\Omega)\qquad \hbox{and}\qquad \|R^n\|_{\mathcal{L}(L^2(\Omega);L^\infty(\Omega))}\le C.$$
In particular, if $u$ is   a first eigenfunction   of $-\Delta+V\cdot \nabla$ on $\Omega$   normalized by $\left\Vert u\right\Vert_{L^2}=1$,   then $u\in L^\infty(\Omega)$ and 
$$\|u\|_{L^\infty}\le C\lambda_1^n(\Omega,V).$$  
\end{lm}
\begin{proof}
Let us first notice that if $d\le 3$, then   $d/2<2$   and so, taking $n=1$, the claim follows directly   by Lemma \ref{infttybndflem}.   
%If $d\ge3$, then by the Gagliardo-Nirenberg-Sobolev inequality $R$ is a continuous operator
%\begin{equation}\label{e:2star}
%R:L^2(\Omega)\to L^{2^*}(\Omega)\qquad \hbox{and}\qquad \|R\|_{\mathcal{L}(L^2(\Omega);L^{2^*}(\Omega))}\le C,
%\end{equation}
%where $\ds 2^{\ast}=\frac{2d}{d-2}$ and $C_{d,\mu}$ depends only on the dimension $d$ and torsion $T(\mu)$. 
%\item Suppose that the dimension $d$ is $4$ or $5$. Then by \eqref{22ast} and Remark \ref{osspinfty} we have that the composition $R_\mu^2=R_\mu\circ R_\mu: L^2(\R^d)\to L^2(\R^d)$ can be extended to a continuous operator  
%$$R_\mu^2:L^2(\R^d)\to L^{\infty}(\R^d),$$
%with norm bounded by a constant depending on $d$ and $T(\mu)$.
If $d>3$, then setting $2^\ast=\frac{2d}{d-2}$, we have 
$$R: L^2(\Omega)\to L^{2^\ast}(\Omega)\qquad\hbox{and}\qquad R:L^d(\Omega)\to L^\infty(\Omega).$$
Thus, interpolating between $2$ and $d$, we get   
\begin{equation}\label{interpol1}
\|R\|_{\mathcal L(L^p;L^q)}\le  C,\quad\hbox{where}\quad p\in[2,d]\quad\hbox{and}\quad q=\frac{pd}{d-p}\ge \frac{pd}{d-2},
\end{equation}
where $C$ depends only on $d$, $|\Omega|$ and $\|\Phi\|_{L^\infty}$. Now, it is sufficient to notice that $R^k\in\mathcal{L}(L^2;L^{q_k})$, where $\ds q_k= 2\Big(\frac{d}{d-2}\Big)^k$. For $k$ big enough we have that $q_k>\sfrac{d}2$ and so, $R^{k+1}\in\mathcal{L}(L^2;L^{\infty})$, which proves the first part of the claim. Finally, in order  to get the estimate on $u$, it is sufficient to notice that $R(u)=\lambda_1^{-1}(\O,V)u$ and $R^n(u)=\lambda_1^{-n}(\O,V)u$. 
\end{proof}

\subsection{Pointwise definition of the solutions}\label{sub:pointwise}
When we deal with Sobolev functions we usually reason up to a choice of certain representative of the function. Even if this representative is defined quasi-everywhere, there still might be a set of zero capacity where the function is not defined. Of course, this interferes with the notion of a free boundary   in the sense   that we cannot just consider the topological boundary of $\Omega_u$ without specifying the   representative   of $u$ that we work with. Fortunately, the eigenfunctions of the quasi-open sets are defined pointwise everywhere, that is {\it every } point is a Lebesgue point.

\begin{lm}[Subharmonicity and a mean-value formula for positive solutions of PDEs]\label{l:radon}
Let $D\subset\R^d$ be a bounded open set, $\Omega\subset D$ a quasi-open set and $\Phi\in W^{1,\infty}(D)$ a given Lipschitz function. Let $f\in L^\infty(D)$ and $u\ge 0$ be a solution to the problem \eqref{e:eqf} in $\Omega$.
%\begin{equation}\label{e:eqf}
%-\dive\,(e^{-\Phi} \nabla u)=f\quad\text{in}\quad\Omega,\qquad u\in H^1_0(\Omega). 
%\end{equation}
\begin{enumerate}
	\item Then, $\dive(e^{-\Phi} \nabla u) + f \geq 0$ in $\R^d$, in the sense of distributions. In particular, $\dive(e^{-\Phi} \nabla u)$ is a (signed) Radon measure on $\R^d$. 
	\item For any $x_0\in \R^d$, we can define the value of $u$ at $x_0$ by
	\[ u(x_0) = \lim_{r\rightarrow 0} \aver{\partial B_r(x_0)} u(x) \,d\mathcal{H}^{d-1}(x)=\lim_{r\rightarrow 0} \aver{B_r(x_0)} u(x)\,dx. \]
	Moreover, we have the identity
	\begin{multline}\label{eq id moy}
	\aver{\partial B_r(x_0)} u e^{-\Phi} d\mathcal{H}^{d-1} - u(x_0)e^{-\Phi(x_0)} = \frac{1}{d\om_d} \int_0^r s^{1-d} \dive (e^{-\Phi} \nabla u)(B_s(x_0))\,ds \\ 
	- \frac{1}{d\om_d} \int_0^r s^{1-d} \,ds\,\int_{\partial B_s}(\nabla\Phi\cdot\nu)\,u\, e^{-\Phi}\,d\mathcal{H}^{d-1},
	\end{multline} 
	where $\nu$ denotes the normal to $\partial B_s$ pointing outwards. 
		\item Let $x_0\in \R^d$ and $R>0$. Suppose that there is a constant $C>0$ such that
		\begin{equation}\label{e:delta_est}
		\big|\text{div}\,(e^{-\Phi}\nabla u)(B_r(x_0))\big|\le Cr^{d-1}\qquad\text{for every}\qquad 0<r\le R.
		\end{equation}
		Then we have the estimate
	\begin{equation}\label{e:rate}
	\Big|u(x_0)-\aver{\partial B_r(x_0)} u \, d\mathcal{H}^{d-1}\Big|\le e^{M_\Phi}\left(\frac{C}{d\om_d}+2L_{\Phi}M_ue^{M_\Phi}\right)r\qquad\text{for every}\qquad 0<r\le R,
	\end{equation} 
	where $L_\Phi:=\|\nabla\Phi\|_{L^\infty(D)}$, $M_u:=\|u\|_{L^\infty(D)}$ and $M_\Phi:=\|\Phi\|_{L^\infty(D)}$.
\end{enumerate}
\end{lm}

\begin{proof}
\textit{(1)} For $n\in\N$ define $p_n : \R \rightarrow \R$ by
\[p_n(s)=0,\  \text{ for }\ s \leq 0;\qquad p_n(s) = ns,\  \text{ for }\  s \in [0,\sfrac1n]; \qquad p_n(s) =1,\  \text{ for }\  s \geq \sfrac1n. \]
Since $p_n$ is Lipschitz continuous, we have $p_n(u) \in H^1_0(\Omega)$ and $\nabla p_n(u) = p'_n(u)\nabla u$. Let $\varphi \in C^{\infty}_0(D)$, $\varphi\ge 0$ in $D$.  Using $\varphi p_n(u)$ as a test function in \eqref{e:eqf}, we get
\[\int_D p_n(u) \nabla u \cdot \nabla \varphi \, e^{-\Phi} dx\le \int_D\big(p_n(u) \nabla u \cdot \nabla \varphi + \varphi p'_n(u) |\nabla u|^2\big) e^{-\Phi} dx = \int_D f \varphi p_n(u)\,dx. \]
which, letting $n\to\infty$, gives the first claim. 

In order to prove  \textit{(2)}, we suppose that $x_0=0$ and we calculate 
\begin{align*}
\frac{d}{ds} \aver{\partial B_s} u e^{-\Phi} d\mathcal{H}^{d-1} &= \frac{d}{ds} \aver{\partial B_1} u(s\xi)e^{-\Phi(s\xi)} d\mathcal{H}^{d-1} \\
&= \aver{\partial B_1} \big[ \xi\cdot\nabla u (s\xi)  - u(s\xi) \xi\cdot\nabla \Phi(s\xi) \big] e^{-\Phi(s\xi)} d\mathcal{H}^{d-1} \\
&=\frac{s^{1-d}}{d\omega_d} \dive(e^{-\Phi} \nabla u )(B_s) - \frac{s^{1-d}}{d\omega_d} \int_{\partial B_s}(\nabla\Phi\cdot\nu)\, u\,e^{-\Phi}\,d\mathcal{H}^{d-1}.
\end{align*} 
Then, integrating from $\rho$ to $r$  ($\rho<r$),   using the inequality from \textit{(1)} and the fact that $u \in L^\infty(D)$ by Lemma \ref{l:boundedness}, we get 
\begin{align}\label{eq id}
\aver{\partial B_r}ue^{-\Phi}d\mathcal{H}^{d-1} -\aver{\partial B_\rho}ue^{-\Phi}d\mathcal{H}^{d-1} 
&= \frac{1}{d\om_d} \int_{\rho}^r s^{1-d} \dive (e^{-\Phi} \nabla u)(B_s(x_0))\,ds \\
\nonumber &\qquad \qquad \quad- \frac{1}{d\om_d} \int_{\rho}^r s^{1-d} \,ds\,\int_{\partial B_s}(\nabla\Phi\cdot\nu)\,u\, e^{-\Phi}\,d\mathcal{H}^{d-1} \\
\nonumber&\geq -\frac{1}{2d}\|f\|_{L^\infty} \left(r^2-\rho^2\right)  - e^{-\min\Phi}\|u\|_{L^\infty}\,\|\nabla\Phi\|_{L^\infty}(r-\rho) \\ \nonumber&:=-A\left(r^2-\rho^2\right) -B(r-\rho),
\end{align} 
  where $A,B>0$. This shows that the function $\ds r\mapsto \aver{\partial B_r}ue^{-\Phi}d\mathcal{H}^{d-1}+Ar^2+Br$ is non-decreasing.   In particular, the limit $\ds \ell(x_0)=\lim_{r\rightarrow 0} \aver{\partial B_r(x_0)} u e^{-\Phi} d\mathcal{H}^{d-1}$ exists and we set $u(x_0):=e^{\Phi(x_0)}\ell(x_0)$. Now, \eqref{eq id moy} follows by letting $\rho\to0$ in \eqref{eq id}. Finally, in order to prove the claim (3), we notice that \eqref{eq id moy} implies 
  	\begin{equation}\label{e:rate2}
  \Big|\aver{\partial B_r(x_0)} u e^{-\Phi}\, d\mathcal{H}^{d-1}-u(x_0)e^{-\Phi(x_0)}\Big|\le \left(\frac{C}{d\om_d}+L_{\Phi}M_ue^{M_\Phi}\right)r\qquad\text{for every}\qquad 0<r\le R,
  \end{equation} 
  Now, by the triangular inequality we have
  	\begin{align*}
  \Big|\aver{\partial B_r(x_0)} u\, d\mathcal{H}^{d-1}-u(x_0)\Big|&= e^{\Phi(x_0)}\Big|\aver{\partial B_r(x_0)} u e^{-\Phi(x_0)}\, d\mathcal{H}^{d-1}-u(x_0)e^{-\Phi(x_0)}\Big|\\
  &\le e^{\Phi(x_0)}\Big|\aver{\partial B_r(x_0)} u e^{-\Phi}\, d\mathcal{H}^{d-1}-u(x_0)e^{-\Phi(x_0)}\Big|\\
  &\qquad+e^{\Phi(x_0)}\aver{\partial B_r(x_0)} u \Big| e^{-\Phi(x_0)}-e^{-\Phi(x)}\Big|\, d\mathcal{H}^{d-1}.
  \end{align*} 
  Thus, the claim follows since, by the Lipschitz continuity of $\Phi$, we have that for every $x\in \partial B_r(x_0)$, 
  $$\Big| e^{-\Phi(x_0)}-e^{-\Phi(x)}\Big|\le e^{\|\Phi\|_{L^\infty(D)}}\|\nabla \Phi\|_{L^\infty(D)}|x-x_0|.$$  
  
  \eqref{e:rate} is a direct consequence of \eqref{eq id moy}.
\end{proof}

As a direct consequence of Lemma \ref{l:radon} and \eqref{eq id moy} we get the following strong maximum principle. 
\begin{lm}[Strong maximum principle] \label{l:strmax}
  Let $D\subset\R^d$ be an open connected set and $u\in H^1_0(D)$ satisfy $u\geq 0$. Assume that $\dive(e^{-\Phi}\nabla u)\in L^{\infty}(D)$ satisfies $\dive(e^{-\Phi}\nabla u)\leq 0$. Then, if $u$ is not identically vanishing in $D$,   then $u$ is strictly positive in $D$. 
\end{lm}

  \begin{proof}
Set $A:=\left\{x_0\in D;\ u(x_0)=0\right\}$. If $x_0\in A$, then \eqref{eq id moy} implies that $u(x)=0$ for almost every $x\in B_r(x_0)$ whenever $B_r(x_0)\subset D$.  Therefore, for all $x\in B_r(x_0)$, since $x$ is a Lebesgue point for $u$, $u(x)=0$. Thus, $A$ is open.\par
\noindent Consider now a sequence $(x_n)_{n\geq 1}\in A$ converging to $x_0\in D$. For some $n$ large enough, there exists a ball $B_r(x_n)\subset D$ containing $x_0$. Since $u$ vanishes everywhere in $B_r(x_n)$, $u(x_0)=0$, which proves that $A$ is closed in $D$. We conclude by the connectedness of $D$.
\end{proof} 

%  Another consequence of Lemma \ref{l:radon} is that if $u\in H^1_0(D)$ is a solution of \eqref{e:fbp}, then 
%$$u(x_0)=\lim_{r\to0}\aver{\partial B_r(x_0)}u\,d\HH^{d-1}=\lim_{r\to0}\aver{B_r(x_0)}u\,dx\quad\text{for every}\quad x_0\in D.$$
  A consequence of Lemma \ref{l:radon} is the fact that   the set $\Omega_u=\{u>0\}$ and the (topological) free boundary $\partial\Omega_u \cap D$ are well defined. Below we prove that the topological boundary coincides with the measure theoretic one. 

\begin{lm}[The topological boundary coincides with the measure-theoretic one]\label{lem hypDr}
Let $u\in H^1_0(D)$, $u\ge 0$ in $D$, be a solution of \eqref{e:fbp}, $x_0 \in \partial\Omega_u$ and let $r>0$ be such that $D_r(x_0):=B_r(x_0) \cap D$ is connected. Then we have $0 < |\Omega_u \cap B_r(x_0)|$. Moreover, if $x_0 \in \partial\Omega_u\cap D$, we have $|\Omega_u \cap B_r(x_0)|< |D_r(x_0)|$.
\end{lm}

\begin{proof}
In order to prove the first inequality, suppose that $x_0\in\partial \Omega_u$ and $|B_r(x_0)\cap\Omega_u|=0$ for some $r>0$. Since every point $x\in B_r(x_0)$ is a Lebesgue point for $u$ and $u=0$ almost everywhere in $B_r(x_0)$ we have that $u\equiv0$ in $B_r(x_0)$, but this contradicts the fact that $x_0\in\partial \Omega_u$. 

In order to show the second inequality, we assume by contradiction that $|\Omega_u \cap B_r(x_0)| = |D_r(x_0)|$ for some $r>0$. We claim that $u$ is a solution of 
$$-\dive(e^{-\Phi} \nabla u) = \lambda_m u e^{-\Phi}\quad\text{in}\quad D_r(x_0),\qquad\text{where}\quad \lambda_m:=\int_D|\nabla u|^2e^{-\Phi}\,dx.$$
Indeed, let $v$ be the solution of  
$$-\dive(e^{-\Phi} \nabla v) = \lambda_m u e^{-\Phi}\quad\text{in}\quad D_r(x_0),\qquad v=u \quad\text{in}\quad D \backslash B_r(x_0).$$
Then   Lemma \ref{l:strmax} implies that $v>0$ in $D_r(x_0)$.  
Since $|\O_v | =|\O_u|$, the optimality of $u$ gives 
\begin{align*}
\frac{\int_D|\nabla v|^2e^{-\Phi}\,dx}{\int_Dv^2e^{-\Phi}\,dx}&\ge \int_D|\nabla u|^2e^{-\Phi}\,dx=\frac{\int_D|\nabla u|^2e^{-\Phi}\,dx}{\int_Dv^2e^{-\Phi}\,dx}+\lambda_m\left(1-\frac{\int_Du^2e^{-\Phi}\,dx}{\int_Dv^2e^{-\Phi}\,dx}\right),
\end{align*}
which implies 
\begin{align*}
0&\ge  \int_D\big(|\nabla u|^2-|\nabla v|^2\big)e^{-\Phi}dx+\lambda_m \int_D\big(v^2-u^2\big)e^{-\Phi}dx
=\int_D\left(|\nabla (u-v)|^2+\lambda_m(u-v)^2\right)e^{-\Phi}dx,
\end{align*}
%\begin{align*}
%0&\ge  \int_D\big(|\nabla u|^2-|\nabla v|^2\big)e^{-\Phi}\,dx+\lambda_m \int_D\big(v^2-u^2\big)e^{-\Phi}\,dx\\
%&=\int_D|\nabla (u-v)|^2 e^{-\Phi}\,dx+\lambda_m \int_D(u-v)^2e^{-\Phi}\,dx\\
%&\qquad+2\int_D\big(-\nabla v\cdot\nabla (v-u)+\lambda_m u(v-u)\big)e^{-\Phi}\,dx\\
%&=\int_D|\nabla (u-v)|^2e^{-\Phi}\,dx+\lambda_m \int_D(u-v)^2e^{-\Phi}\,dx.
%\end{align*}
where the last equality follows by the definition of $v$   and the fact that $v-u\in H^1_0(D_r(x_0))$.   This implies that $u=v$ almost everywhere and hence, by Lemma \ref{l:radon}, that $u=v$ everywhere. Therefore, we have $u>0$ in $B_r(x_0)$, which is in contradiction with $x_0 \in \partial\O_u\cap D$. 
\end{proof}

\begin{lm}[Saturation of the constraint]\label{l:saturation}
Let $D\subset\R^d$ be an open connected set, $\Phi\in W^{1,\infty}(D)$, $m$ and $\tau$ be as in Lemma \ref{l:fbp}. Then every solution $u$ of \eqref{e:fbp} is such   that   $u\ge 0$ on $D$ and $|\Omega_u|=m$ (up to a change of  sign). In particular, every solution $\Omega$ of \eqref{e:optO} is such that $|\Omega|=m$.
\end{lm}
\begin{proof}
Let $u$ be a solution of \eqref{e:fbp} and set 
$$
u_1=\frac{u_+}{\displaystyle\left(\int_{D} u_+^2e^{-\Phi}\right)^{1/2}} \mbox{ and } u_2=\frac{u_-}{\displaystyle\left(\int_{D} u_-^2e^{-\Phi}\right)^{1/2}}.
$$ 
We first prove that either $u_1$ or $u_2$ is a solution of \eqref{e:fbp}. It is obvious if $u=u_+$ or $u=u_-$. Otherwise, we have $u_+\neq 0$ and $u_-\neq 0$, and the claim follows from the estimate
$$
\inf\left(\frac{\int_{D} \left\vert \nabla u_+\right\vert^2e^{-\Phi}dx}{\int_D u_+^2e^{-\Phi}dx}, \frac{\int_{D} \left\vert \nabla u_-\right\vert^2e^{-\Phi}dx}{\int_D u_-^2e^{-\Phi}dx}\right)\leq \frac{\int_{D} \left(\left\vert \nabla u_+\right\vert^2+\left\vert \nabla u_-\right\vert^2\right)e^{-\Phi}dx}{\int_D \left(u_+^2+u_-^2\right)e^{-\Phi}dx}=\frac{\int_{D} \left\vert \nabla u\right\vert^2e^{-\Phi}dx}{\int_D u^2e^{-\Phi}dx}.
$$ 
Up to changing $u$ into $-u$, we assume that $u_1$ is a solution of \eqref{e:fbp}. Now, suppose by contradiction that $|\O_u|< m$. Then, for every ball $B_r(x_0) \subset D$ such that $|\Omega_u|+|B_r|\le m$, writing that 
$$
\int_{D} \left\vert \nabla u_1\right\vert^2e^{-\Phi}dx\leq \int_{D} \left\vert \nabla (u_1+t\varphi)\right\vert^2e^{-\Phi}dx
$$
for all functions $\varphi\in H^1_0(B_r(x_0))$ and all $t\in \R$, we easily get that $u_1$ is a solution of
$$-\dive(e^{-\Phi}\nabla u_1)=\lambda_m e^{-\Phi}u_1\quad\text{in}\quad B_r(x_0).$$
By the strong maximum principle, we get $u>0$ in $B_r(x_0)$, which is a contradiction. This proves both the saturation of the constraint and the positivity of $u$.
\end{proof}

\subsection{A free-boundary problem with measure constraint}\label{sub:fbp}
We now follow the strategy adopted in \cite{briancon-lamboley, briancon}. In particular, the proof of Theorem \ref{thm mu+-} below is very close to the one of Theorem 1.5 in \cite{briancon-lamboley}. Note that the approach is local and that a result   analogous   to Theorem \ref{thm mu+-} with perturbations in $D$ is vain (see Remark 1.6 in \cite{briancon-lamboley}). 

Let $u\in H^1_0(D)$ be a solution of \eqref{e:fbp}.
For any $v\in H^1_0(D)$ we set 
\begin{equation}\label{e:defJ}
J(v): = \int_D |\nabla v|^2 e^{-\Phi}dx - \lambda_m \int_D v^2 e^{-\Phi}dx, \end{equation}
where it is recalled that $\ds\lambda_m=\int_D|\nabla u|^2e^{-\Phi}\,dx$. 
 \begin{oss}[Removal of the integral constraint] \label{oss:jujv}
It is plain to see that, when $u\in H^1_0(D)$ is a solution of \eqref{e:fbp}, 
\begin{equation}\label{e:uminJ}
J(u)=\min\left\{J(v)\ : \ v\in H^1_0(D),\ \left\vert \Omega_v\right\vert\leq m\right\}.
\end{equation}
\end{oss}
  For a ball $B_r(x_0) \subset \R^d$ we define the admissible set
\[ \adm : = \big\{ v \in H^1_0(D) \ : \ u-v \in H^1_0(B_r(x_0)) \big\}. \]
\begin{oss}[Coercivity of $J$]\label{rem:J_bounded}
We notice that the set  $\{v\in\adm\ :\ J(v)<C\}$ is weakly compact in $H^1_0(D)$. Precisely, if $u\in H^1_0(D)$, $\Phi\in W^{1,\infty}(D)$ and $J$ be given by \eqref{e:defJ}, then there is a constant $r_0>0$, depending on $d$, $\Phi$, $\lambda_m$ and $D$ such that for all $r\leq r_0$, 	\begin{equation}\label{eq J bounded}
\int_{B_r(x_0)} |\nabla v|^2 \,dx\le 2e^{\max\Phi}J(v) + \left(1+4\lambda_m e^{\max\Phi-\min\Phi} \right)\,\|u\|_{H^1(D)}^2, \quad \forall v \in \adm .
\end{equation} 
Indeed, let $v \in \adm$ with $r \leq r_0$. We have
\begin{align*}
\int_D v^2 \,dx &\leq 2\int_D (v -u)^2 \,dx + 2\int_Du^2\,dx \leq \frac{2}{\lambda_1(B_r(x_0))} \int|\nabla(v-u)|^2 dx + 2\int_D u^2 dx \\
&\leq \frac{4 r_0^2}{\lambda_1(B_1)}  \int_D \left(|\nabla v|^2 + |\nabla u|^2 \right)\, dx + 2\int_D u^2\, dx,
\end{align*}
where the last inequality is due to the $(-2)$-homogeneity of $\lambda_1(B_r)$ and the fact that $r\le r_0$.
%, we can choose $r_0>0$ such that
%\[ 4\frac{\lambda_m + \tau}{\lambda_1(B_{r_0})} \leq \frac12. \]
Choosing $r_0$ small enough (depending only on $d$, $\lambda_m$, $\|\nabla \Phi\|_{L^\infty}$ and the diameter of $D$) we get 
\begin{align*}
\int_{B_r(x_0)} |\nabla v|^2 \,dx&\le e^{\max\Phi}J(v)+\lambda_me^{\max\Phi-\min\Phi}\int_{D}v^2\,dx\\
&\le e^{\max\Phi}J(v)+\frac12\int_{B_r(x_0)} \left(|\nabla v|^2 + |\nabla u|^2 \right)\, dx+ 2\lambda_m e^{\max\Phi-\min\Phi}\int_{D}u^2\,dx .
\end{align*} 
which concludes the proof of \eqref{eq J bounded}. 
\end{oss}
As a consequence, we obtain the following result, which gives us the existence of a solution to a local version of the minimization problem \eqref{e:uminJ} with some different measure constraint. 
\begin{lm}[Existence of local minimizers]\label{lm:localminpb}
%[Local minimality of the eigenfunctions on the optimal sets]\label{l:minJ}
Let $u\in H^1_0(D)$ be a solution of the problem \eqref{e:uminJ}.   Let $B_r(x_0)\subset \R^d$ be a fixed ball and let $\widetilde m$ be a real constant such that $\widetilde{m}>|\O_u \setminus B_r(x_0)|$. Then:
\begin{enumerate}
\item the problem 
\begin{equation}\label{e:fbpJ}
\min\Big\{J(v)\, :\ v\in \adm,\ |\Omega_v|\le \widetilde{m}\Big\}
\end{equation}
has a solution provided that $r\leq r_0$ with $r_0$ given by Remark \ref{rem:J_bounded}, 
\item if $D_r(x_0):=B_r(x_0)\cap D$ is connected and $|\O_u \cup D_r(x_0)|>\tilde{m}$, then  $\left\vert \Omega_v\right\vert=\widetilde{m}$;
\item  there exists $r_0>0$ such that, for every $r<r_0$, every solution $v$ of \eqref{e:fbpJ} is non-negative. 
\end{enumerate}
\end{lm}
\begin{proof}
  For $1$, it is enough to notice that, by Remark \ref{rem:J_bounded}, $J$ is bounded from below in $\adm$. Then, if $(v_n)_{n\geq 1}$ is a minimizing sequence for \eqref{e:fbpJ}, by \eqref{eq J bounded} $v_n$ is bounded in $H^1$
%and so a minimizer exists by the semicontinuity of $J$ (notice that, up to a subsequence, there exists $v\in L^2$ such that $v_n\rightarrow v$ strongly in $L^2$ and almost everywhere, so that $\ind_{\Omega_v}\leq \varliminf \ind_{\Omega_{v_n}}$).\par
\noindent For $2$, if $D_r(x_0)$ is connected and $|\O_u \cup D_r(x_0)|>\tilde{m}$, we argue as in the proof of Lemma \ref{l:saturation} to conclude that $\left\vert \Omega_v\right\vert=\widetilde{m}$. 
\noindent For $3$, let $v$ be a solution of \eqref{e:fbpJ}. Then, by the optimality of $v$ and the fact that $v^+\in \adm$ and $\Omega_{v^+}\subset \Omega_v$, one has 
$$
J(v^+)+J(v^-)=J(v)\leq J(v^+),
$$
which means that $J(v^-)\leq 0$. Therefore,
\begin{align*}
\int_{B_r(x_0)}|\nabla v^-|^2 e^{-\Phi}\,dx&\le \lambda_m\int_{B_r(x_0)}|v^-|^2 e^{-\Phi}\,dx\\
%&\le\lambda_me^{\max\Phi-\min\Phi}\lambda_1(B_r,0)\int_{B_r(x_0)}e^{-\Phi}|\nabla v_n^-|^2\,dx\\
&\le\lambda_m e^{2r\tau}C_d r^2 \int_{B_r(x_0)}|\nabla v^-|^2 e^{-\Phi}\,dx,
\end{align*}
where the second inequality is due to the fact that $\max_{B_r(x_0)}\Phi-\min_{B_r(x_0)}\Phi\le 2r\tau $ and the variational characterization and the scaling of $\lambda_1(B_r,0)=C_d r^{-2}$. Thus, for $r$ small enough ($r\le r_0$ with $r_0$ depending only on $\tau$, $\lambda_m$ and $d$), $v^-=0$. \par
   
%Now, if $v$ is a minimizer of \eqref{e:fbpJ} and $r_0$ is small enough (such that $\lambda_1(B_{r_0})>\lambda_m e^{\max\Phi-\min\Phi}$), then $v_-\in H^1_0(B_{r}(x_0))$ and 
%$J(v)=J(v_+)+J(v_-)\le J(v_+)$, where the last inequality is strict if $v_-\neq0$. Thus, $v\ge0$. Finally, $v$ is a solution of \eqref{e:fbp}, since
%$J(v)\le 0$ implies that $\frac{\int_D|\nabla v|^2e^{-\Phi}\,dx}{\int_Dv^2e^{-\Phi}\,dx}\le \lambda_m.$ Notice that the last inequality has to be an equality since $\lambda_m$ is the minimum of \eqref{e:fbp}. In particular, $J(v)=0$ and $u$ is also a solution of \eqref{e:fbpJ}.
\end{proof}

\subsection{An internal variation optimality condition}\label{sub:optimality_condition}
%Let $u$ be a solution of \eqref{e:fbp} in $D\subset\R^d$ and 
Let $D\subset\R^d$ be a bounded open set, $u\in H^1_0(D)$ and 
$\xi \in C^\infty_c(D;\R^d)$. The first variation $\delta J(u)[\xi]$, of $J$ at $u$ in the direction $\xi$, is given by 
$$\delta J(u)[\xi]:=\lim_{t\to0}\frac{J(u_t)-J(u)}t,\quad\text{where}\quad u_t(x):=u(x+t\xi(x)).$$
A straightforward computation gives that 
\begin{equation}\label{e:first_variation}
\delta J(u)[\xi] := \int_D \Big[ 2 D\xi(\nabla u)\cdot\nabla u + (|\nabla u|^2 - \lambda_m u^2) (\nabla \Phi \cdot \xi - \dive \xi) \Big] e^{-\Phi}\,dx.
\end{equation}

We prove in Proposition \ref{p:lagrange} the existence of an Euler-Lagrange multiplier for every solution $u$ of  \eqref{e:uminJ}. This, using a local internal variation of the boundary of the optimal set $\O_u$, we derive an optimal boundary condition for $u$ (see Lemma \ref{l:visc}). 

\begin{prop}[Euler-Lagrange equation]\label{p:lagrange}
Let $u$ be a solution of \eqref{e:uminJ}. Then, there exists $\Lambda_u > 0$ such that
%\begin{equation}\label{e:connectedness_condition2}
%D_r(x_0):=B_r(x_0)\cap D\quad\text{is connected}\qquad\text{and}\qquad 0<|D_r(x_0)\cap \Omega_u|<|D_r(x_0)|,
%\end{equation}  
%$\xi=(\xi_1,\dots,\xi_d) \in C^\infty_c(B_r(x_0);\R^d)$, we have
\begin{equation}\label{e:lagrange}
\delta J(u)[\xi] =  \Lambda_u \int_{\O_u}\!\!\dive \xi\, dx\qquad\text{for every}\qquad \xi\in C^\infty_c(D;\R^d).
\end{equation}
%Moreover, if $0<|B_r(x_0)\cap\Omega|<|B_r(x_0)|$, then $\Lambda_u>0$.
Moreover, for every $x_0 \in \partial\O_u\cap\partial D$ and every $r>0$, we have 
\begin{equation*}
\delta J(u)[\xi] \geq  \Lambda_u \int_{\O_u}\!\!\dive \xi\, dx,
\end{equation*}
for every $\xi \in C^\infty_c(B_r(x_0),\R^d)$ such that $(Id+\xi)^{-1}(D_r(x_0))\subset D_r(x_0)$.
\end{prop}
\begin{proof}
Let $\xi \in C^\infty_c(D;\R^d)$ and $u_t(x) = u(x + t\xi(x))$. Then we have
\begin{align}
|\O_{u_t}| &= |\O_u| -t\int_{\O_u} \dive\xi\,dx + o(t).\label{eq deri measure} 
\end{align}
{\it Step 1.} We first notice that if $B_r(x_0) \subset \R^d$ is a ball such that 
\begin{equation*}
D_r(x_0):=B_r(x_0)\cap D\quad\text{is connected}\qquad\text{and}\qquad 0<|D_r(x_0)\cap \Omega_u|<|D_r(x_0)|,
\end{equation*} 
then there is a vector field $\xi_0\in C^\infty_c(D_r(x_0);\R^d)$ such that $\ds \int_{\O_u} \dive\xi_0\,dx =1$. Indeed, if this is not the case, then we have 
\[ \int_{\O_u}\dive \xi\, dx=0\qquad\text{for every}\qquad \xi\in C^\infty_c (D_r(x_0);\R^d) .\]
For every ball $B_\rho(x_1) \subset D_r(x_0)$, take a vector field of the form $\xi(x)=(x-x_1)\phi_\eps(x)$ with $0\le \phi_\eps\le 1$ on $B_\rho(x_1)$, $\phi$ radially decreasing in $B_\rho(x_1)$ with $\left\vert \nabla\phi_\eps\right\vert \leq C(\rho\eps)^{-1}$, $\phi_\eps=1$ on $B_{\rho(1-\eps)}(x_1)$ and $\phi_\eps=0$ on $\partial B_\rho(x_1)$. Then we have 
$\ \ds \int_{\O_u}\big(d\phi_\eps(x)+(x-x_1)\cdot\nabla\phi_\eps(x)\big)\, dx=0$ and, passing to the limit as $\eps\to0$, we get 
$$d|\Omega_u\cap B_\rho(x_1)|-\rho\,\HH^{d-1}\big(\Omega_u\cap\partial B_\rho(x_1)\big)=0.$$ 
In particular, we get that the map $\rho\mapsto \rho^{-d}|\Omega_u\cap B_\rho(x_1)|$ is constant. Since the above identity holds for all balls $B_\rho(x_1)\subset D_r(x_0)$, we get that $|\Omega_u\cap D_r(x_0)|=0$ or $|\Omega_u\cap D_r(x_0)|=|D_r(x_0)|$, which concludes the proof of the claim. \par
%{\it We notice that instead of \eqref{e:connectedness_condition1} we can assume that:}
%\begin{equation}\label{e:connectedness_condition2}
%D_r(x_0):=B_r(x_0)\cap D\quad\text{is connected}\qquad\text{and}\qquad 0<|D_r(x_0)\cap \Omega_u|<|D_r(x_0)|.
%\end{equation} 
\noindent {\it Step 2.} We now prove the first statement of the proposition. Let $\xi_0\in C^\infty_c(D;\R^d)$ be as in Step 1 and $\xi \in C^\infty_c(D;\R^d)$.
There are two cases:

If $\ds\int_{\O_u} \dive{\xi}\,dx = 0$, define $\xi_1 = \xi + \eta \xi_0$ with $\eta>0$ so that $\ds\int_{\O_u} \dive{\xi_1}\,dx = \eta$.\\ Set $u_t(x) = u(x+t\xi_1(x))$. Then, for $t$ small enough, $u_t \in H^1_0(D)$, $|\O_{u_t}| \leq |\O_u| = m$ and
\[ J(u_t) = J(u) +t\,\delta J(u)[\xi_1] + o(t). \]
By the minimality of $u$ we have $J(u) \leq J(u_t)$ and so, $\delta J(u)[\xi_1]  \geq 0$. Therefore, 
$$\delta J(u)[\xi]  \geq -\eta\, \delta J(u)[\xi_0] \quad\text{for every}\quad \eta>0,$$ 
and hence, we get $\delta J(u)[\xi] \geq 0$. Taking $-\xi$ instead of $\xi$ we have that $\delta J(u)[\xi] = 0$, and hence \eqref{e:lagrange} holds for any $\Lambda_u \geq 0$.

If $\ds\int_{\O_u} \dive{\xi}\,dx \neq 0$, define $\ds\xi_2 := \xi - \xi_0 \int_{\O_u} \dive{\xi}\,dx$. Then $\ds\int_{\O_u} \dive{\xi_2}\,dx = 0$ and, by the preceding case, we have $\delta J(u)[\xi_2] = 0$. On the other hand, 
$$\ds\delta J(u)[\xi_2]= \delta J(u)[\xi] - \delta J(u)[\xi_0] \int_{\O_u} \dive{\xi}\,dx,$$ 
which proves \eqref{e:lagrange} with $\Lambda_u := \delta J(u)[\xi_0]$. Moreover, for $t$ small enough, $u_t(x) = u(x+t\xi(x)) \in H^1_0(D)$ and, by the minimality of $u$, we have
\[ J(u) \leq J(u_t) = J(u) +t\Lambda_u + o(t),\] 
which proves that $\Lambda_u \geq 0$. The strict inequality follows by a general result (Proposition \ref{p:density}) for minimizers of $J$ with respect to internal perturbations. 

\noindent {\it Step 3.} Let $x_0 \in \partial\O_u\cap\partial D$, $r>0$ and $\xi_0\in C^\infty_c(D;\R^d)$ be as in Step 1 so that we have $\delta J(u)[\xi_0]=\Lambda_u$. For any $\xi \in C^\infty_c(B_r(x_0),\R^d)$ such that $(Id+\xi)^{-1}(D_r(x_0))\subset D_r(x_0)$, we set $\xi_1=\xi-(1-\eta)\xi_0\int_{\O_u} \dive{\xi}\,dx$ where $\eta$ is some positive constant. Note that the vector field $\xi_1$ is such that $u_t(x)=u(x+t\xi_1(x))\in H^1_0(D)$ for small $t>0$ and $\int_{\O_u} \dive{\xi_1}\,dx = \eta>0$. Therefore, using the minimality of $u$, we have for every $t>0$ small enough 
\[ J(u) \leq J(u_t) = J(u) + t\delta J(u)[\xi_1] + o(t), \]
so that we get $\delta J(u)[\xi_1]\geq 0$. It follows that $\delta J(u)[\xi]\geq (1-\eta)\Lambda_u$ for every $\eta>0$, which concludes the proof. 
\end{proof}

In the following lemma we show that the Lagrange multipliers, associated to the solutions of variational problems with measure constraint in a fixed ball $B_r(x_0)$, are continuous with respect to variations of the measure constraint around $m$. This lemma will be used several times in the proof of the optimality of the blow-up limits.   

\begin{lm}[Convergence of the Lagrange multipliers]\label{l:u_n}
Let $D\subset\R^d$ be a bounded open set, $u\in H^1_0(D)$ be a solution of \eqref{e:uminJ} and $\Lambda_u$ be the constant from \eqref{e:lagrange}.
%Let $x_0\in D$ and $r>0$ be such that $B_r(x_0)\subset D$, $0<|B_r(x_0)\cap\Omega_u|<|B_r|$ and $u$ is a solution of \eqref{e:fbpJ} in $B_r(x_0)$. 
Let $B_r(x_0) \subset \R^d$ be a ball such that
\begin{equation*}
D_r(x_0):=B_r(x_0)\cap D\quad\text{is connected}\qquad\text{and}\qquad 0<|D_r(x_0)\cap \Omega_u|<|D_r(x_0)|.
\end{equation*} 
Let the sequence $ \left(m_n\right)_{n\geq 1}$   be such that $\ds\lim_{n\to\infty}m_n=m$. Then, for $n$ big enough, there is a solution $u_n\in\adm$ of the problem  
\begin{equation}\label{e:fbpJn_le}
\min\Big\{J(v)\, :\ v\in \adm,\ |\Omega_v|  \leq    m_n\Big\}. 
\end{equation}
%Let $\Lambda_{u_n}$ be the corresponding Lagrange multiplier in $B_r(x_0)$. 
Moreover, up to a subsequence, we have:
\begin{enumerate}[(a)]
\item for every $n$ there is a Lagrange multiplier $\Lambda_{u_n}>0$ for which \eqref{e:lagrange} holds for $u_n$ in $D_r(x_0)$;
\item for every $n$ there is a vector field $\xi_n\in C^\infty_c(D_r(x_0);\R^d)$ such that 
\begin{equation}\label{e:u_n_derivatives}
\frac{d}{dt}\Big|_{t=0}J(u_n^t)=\Lambda_{u_n}\qquad\text{and}\qquad\frac{d}{dt}\Big|_{t=0}|\Omega_{u_n^t}|=-1\qquad\text{where}\qquad u_n^t(x):=u_n(x+t\xi_n(x));
\end{equation}
\item $u_n$ converges strongly in $H^1_0(D)$ and pointwise almost everywhere to a function $u_\infty\in\adm$ which is a solution of \eqref{e:fbpJ};
\item the sequence of characteristic functions $\ind_{\Omega_{u_n}}$ converges to $\ind_{\Omega_{u_\infty}}$ pointwise almost everywhere and strongly in $L^2(D)$; 
\item if we have $0<|\O_u \setminus B_r(x_0)|<|D\setminus B_r(x_0)|$, then  $\ds\lim_{n\to\infty}\Lambda_{u_n}=\Lambda_u.$
\end{enumerate} 
Furthermore, if $D$ is of class $C^{1,1}$ and $m_n<m$ for every $n$ large enough, then all these properties still hold even if the  assumption $|\O_u\cap D_r(x_0)|<|D_r(x_0)|$ is not satisfied.
\end{lm}
\begin{proof} First of all, we notice that since $|\Omega_u\setminus D_r(x_0)|<m<|\Omega_u\cup D_r(x_0)|$, we may assume that the same holds for every $m_n$, for $n$ large enough. Thus, by Lemma \ref{lm:localminpb}, the problem \eqref{e:fbpJn_le} has a solution $u_n$ such that $|\O_{u_n}|=m_n$. Then, it follows that $u_n$ satisfies 
\begin{equation}\label{eq hypHLun}
0<|\Omega_{u_n}\cap D_r(x_0)|<|D_r(x_0)|.
\end{equation} 
%  Next, let $v_n$ be a solution of the problem \eqref{e:fbpJn_le}. Lemma \ref{l:minJ} yields some $r_0>0$ independent of $n$ such that, for all $r<r_0$, $v_n\ge 0$ in $B_r(x_0)$ and $|\Omega_{v_n}|=m_n$. In particular, the problem 
%\begin{equation}\label{e:fbpJn=}
%\min\Big\{J(v)\, :\ v\in \adm,\ |\Omega_v|=m_n\Big\}
%\end{equation}
%has a solution.
%Then $v_n\ge 0$ in $B_r(x_0)$. Indeed, if this was not the case, then the negative part $v_n^-$ would be non-zero and the quasi-open set $\{v_n^->0\}$ would be non-empty and strictly contained in $B_r(x_0)$. Moreover, by the optimality of $v_n$ we would have 
%\begin{align*}
%\int_{B_r(x_0)}e^{-\Phi}|\nabla v_n^-|^2\,dx&\le \lambda_m\int_{B_r(x_0)}e^{-\Phi}|v_n^-|^2\,dx\\
%%&\le\lambda_me^{\max\Phi-\min\Phi}\lambda_1(B_r,0)\int_{B_r(x_0)}e^{-\Phi}|\nabla v_n^-|^2\,dx\\
%&\le\lambda_m e^{2r\tau}C_d r^2 \int_{B_r(x_0)}e^{-\Phi}|\nabla v_n^-|^2\,dx,
%\end{align*}
%where the second inequality is due to the fact that $\max_{B_r(x_0)}\Phi-\min_{B_r(x_0)}\Phi\le 2r\tau $ and the variational characterization and the scaling of $\lambda_1(B_r,0)=C_d r^{-2}$. Thus, for $r$ small enough ($r\le r_1$ with $r_1$ depending only on $\tau$, $\lambda_m$ and $d$) all the solutions of \eqref{e:fbpJn_le} are non-negative. Now, the strict maximum principle and Lemma \ref{l:saturation} give that $|\Omega_{v_n}|=m_n$. In particular, we deduce that {\it the problem \eqref{e:fbpJn=} admits a solution} and {\it every solution of \eqref{e:fbpJn=} is also a solution of \eqref{e:fbpJn_le}}.

Therefore, by step 1 in the proof of Proposition \ref{p:lagrange}, there exists a vector field $\xi_n \in C^\infty_c(D_r(x_0);\R^d)$ such that $\ds\int_{\O_{u_n}} \!\!\!\dive{\xi_n}\,dx=1$, and, reasoning as in Proposition \ref{p:lagrange}, there exists $\Lambda_{u_n}> 0$ such that 
\begin{equation}\label{e:Lambda_n}
\delta J(u_n)[\xi] = \Lambda_{u_n}\int_{\O_{u_n}} \!\!\!\dive{\xi}\,dx\qquad\text{for every}\qquad \xi \in C^\infty_0(D_r(x_0),\R^d).
\end{equation}
Moreover, taking $u_n^t(x) = u_n(x+ t\xi_n(x))$, we obtain \eqref{e:u_n_derivatives}. This proves (a) and (b). We notice that the only difference with Proposition \ref{p:lagrange} is that in the present case, $u_n$ is only a solution of a variational problem in $B_r(x_0)$. 

Let now $n$ be fixed and $\xi_0\in C^\infty_c(B_r(x_0);\R^d)$ be the vector field, from the proof of Proposition \ref{p:lagrange}, associated to $u$. Then, taking $u_t(x):=u(x+t\xi_0(x))$, we have that 
$$\frac{d}{dt}\Big\vert_{t=0}|\Omega_{u_t}|=-\int_{\Omega_u}\dive\xi_0\,dx=-1,$$
and so, for $n$ large enough, there is a unique   $t_n\in \R$   such that 
$|\Omega_{u_n}|=m_n=|\Omega_{u_{t_n}}|.$
In particular, there are constants $C$ and $n_0$, depending on $u$ and $\xi_0$, but not on $n$, such that 
$$J(u_n)\le J(u_{t_n})\le C\quad\text{for every}\quad n\ge n_0.$$
Then, by Remark \ref{rem:J_bounded}, $  \left(u_n\right)_{n\geq 1}$   is uniformly bounded in $H^1_0(D)$, so up to a subsequence, $u_n$ converges weakly in $H^1$, strongly in $L^2$ and pointwise a.e. to a function $u_\infty\in \adm$. Now, since the pointwise convergence implies $\ind_{\Omega_{u_\infty}}\le\liminf \ind_{\Omega_{u_n}}$, we get that  $|\Omega_{u_\infty}|\le \liminf m_n=m$. In particular, $J(u)\le J(u_\infty)$. On the other hand, the weak $H^1$ convergence of $u_n$ gives that
$$J(u_\infty)\le \liminf_{n\to\infty} J(u_n)\le \liminf_{n\to\infty} J(u_{t_n})=J(u),$$
so, we get $J(u_\infty)=J(u)$,  $u_\infty$ is a solution  of \eqref{e:fbpJ}, $|\Omega_{u_\infty}|=m$ (by the saturation of the constraint). Moreover, $J(u_n)\to J(u_\infty)$
since we have
\[ \limsup_{n\to\infty} J(u_n) \leq \limsup_{n\to\infty} J(u_{t_n}) = J(u) \leq J(u_\infty) \leq \liminf_{n\to\infty} J(u_n). \]
But $u_n$ strongly converges in $L^2(D)$ to $u_\infty$ so that it gives $\int_D e^{-\Phi}|\nabla u_n|^2\,dx\to \int_D e^{-\Phi}|\nabla u_\infty|^2\,dx$, which means that the convergence of $u_n$ to $u$ is strong in $H^1_0(D)$. \par
\noindent   We now check that the convergence of $\ind_{\Omega_{u_n}}$ to $\ind_{\Omega_{u_\infty}}$ is strong in $L^2$. Indeed, for all non-negative function $\varphi\in L^2(D)$, the Fatou lemma shows that
\begin{equation} \label{Fatou}
\int_D \ind_{\Omega_{u_{\infty}}}\varphi\leq \int_D \varliminf \ind_{\Omega_{u_n}}\varphi \leq \varliminf \int_D\ind_{\Omega_{u_n}}\varphi.
\end{equation}
Up to a subsequence, there exists $h\in L^2(D)$ such that $\ind_{\Omega_{u_n}}\rightharpoonup h$ weakly in $L^2(D)$. Thus, \eqref{Fatou} yields $\ind_{\Omega_{u_\infty}}\leq h$. Moreover, $\left\Vert h\right\Vert_2\leq \varliminf \left\Vert \ind_{\Omega_{u_\infty}}\right\Vert_2$. As a consequence, $\left\Vert h\right\Vert_2=m^{1/2}$, which entails that $\ind_{\Omega_{u_n}}\rightarrow h$ strongly in $L^2(D)$. Since $\ind_{\Omega_{u_\infty}}\leq h$, we conclude that $\ind_{\Omega_{u_n}}\rightarrow \ind_{\Omega_{u_{\infty}}}$ strongly in $L^2(D)$,   and so, up to a subsequence $\ind_{\Omega_{u_n}}$ converges to $\ind_{\Omega_{u_\infty}}$ pointwise almost everywhere. This proves (c) and (d).

In order to prove (e), we first notice that $u$ and $u_\infty$ are both solutions of \eqref{e:uminJ} since $J(u_\infty)=J(u)$. Therefore, there is a Lagrange multiplier $\Lambda_\infty$ such that 
\begin{equation}\label{e:Lambda_infty}
\delta J(u_\infty)[\xi] =  \Lambda_\infty \int_{\O_{u_\infty}}\!\!\dive \xi\, dx\qquad\text{for every}\qquad \xi\in C^\infty_c(D;\R^d),
\end{equation}
Moreover, by (c) and (d), we get that 
$$\delta J(u_\infty)[\xi]=\lim_{n\to\infty}\delta J(u_n)[\xi]\qquad\text{and}\qquad \int_{\Omega_{u_\infty}}\dive\xi\,dx=\lim_{n\to\infty}\int_{\Omega_{u_n}}\dive\xi\,dx,$$ 
for every $\xi\in C^\infty_c(D_r(x_0);\R^d)$. Now, choosing $\xi\in C^\infty_c(D_r(x_0);\R^d)$ such that $\ds\int_{\Omega_{u_\infty}}\!\!\!\dive\xi\,dx\neq0$ and using \eqref{e:Lambda_infty} and \eqref{e:Lambda_n} we get that $\Lambda_{u_n}$ converges to $\Lambda_{\infty}$. 
Finally, if we have $0<|\O_u \setminus B_r(x_0)|<|D\setminus B_r(x_0)|$, there exists $\xi\in C^\infty_c(D \setminus B_r(x_0);\R^d)$ such that $\ds\int_{\Omega_{u_\infty}}\!\!\!\dive\xi\,dx\neq0$, so that $\Lambda_\infty=\Lambda_u$ since $u=u_\infty$ outside the ball $B_r(x_0)$.  

The proof of the last statement of the Proposition is very similar. We have $|\O_u \setminus D_r(x_0)|<m=|\O_u \cup D_r(x_0)|$ so that, since $m_n<m$, we have $|\O_u \setminus D_r(x_0)|<m_n<|\O_u \cup D_r(x_0)|$ for every $n$ large enough. It follows from Lemma \ref{lm:localminpb} that the problem \eqref{e:fbpJn_le} has a solution $u_n$ with $|\O_{u_n}|=m_n$ and such that \eqref{eq hypHLun} holds. Note also that there exists a vector field $\xi_0 \in C^\infty_0(B_r(x_0),\R^d)$ such that $(Id+t\xi_0)^{-1}(D_r(x_0))\subset D_r(x_0)$ for every small $t>0$ and $\int_{\O_u} \dive{\xi_0}\,dx = 1$ (consider a smooth extension of the normal to the boundary of $D$ on $\partial D \cap B_{\sfrac{r}{2}}(x_0)$). Moreover, we have $t_n>0$ (since $m_n<m$) and hence $u_{t_n} \in H^1_0(D)$. The rest of the proof is unchanged. 
\end{proof}

\subsection{Almost optimality of $u$ at small scales}\label{sub:small_scales}
Let $u$ be a solution of \eqref{e:fbp} in $D\subset\R^d$. For $x_0\in \R^d$ and $h>0$, we define the upper and the lower Lagrange multipliers, $\mu_-(h,x_0,r) \geq 0$ and $\mu_+(h,x_0,r) \geq 0$, by
\begin{align*}
&\mu_+(h,x_0,r) = \inf \{ \mu\geq 0 \ : \ J(u) + \mu |\O_u| \leq J(v) + \mu|\O_v|, \ \forall v \in \adm,\ m\leq |\O_v| \leq m+h \}, \\
&\mu_-(h,x_0,r) = \sup \{ \mu\geq 0 \ : \ J(u) + \mu |\O_u| \leq J(v) + \mu|\O_v|, \ \forall v \in \adm,\ m-h\leq |\O_v| \leq m \}.
\end{align*}
\begin{oss}[$\mu_-\le\Lambda_u\le \mu_+$]
	We notice that if $B_r(x_0) \subset \R^d$ is a ball such that
	%$x_0\in\partial\Omega_u\cap D$ and 
	$D_r(x_0):=D\cap B_r(x_0)$ is connected and $0<|D_r(x_0)\cap \Omega_u|<|D_r(x_0)|$, 
	 then $$\mu_-(h,x_0,r)\le \Lambda_u\le \mu_+(h,x_0,r)\qquad\text{for every} \qquad h>0.$$
	Indeed,   by Step $1$ of the proof of Proposition \ref{p:lagrange},   there is a vector field $\xi \in C^\infty_c(D_r(x_0);\R^d)$ such that $\ds\int_{\O_u}\dive{\xi}\,dx = 1$. Let $u_t(x) = u(x+t\xi(x))$. Then for $|t|$ small enough $u_t \in \adm$ and $m-h<|\Omega_{u_t}|<m+h$. Moreover, for every $\mu\ge 0$ we have 
	\begin{equation}\label{e:u_t_expansion}
	J(u_t)+\mu|\Omega_{u_t}|=J(u) +t\Lambda_u + \mu(|\Omega_u|-t) + o(t). 
	\end{equation}
     Now, if $t>0$ is small enough and $\Lambda_u<\mu$, then $m>|\Omega_{u_t}|$ and, by  \eqref{e:u_t_expansion},
	$J(u_t)+\mu|\Omega_{u_t}|< J(u) + \mu|\Omega_u|$,
	which proves that $\Lambda_u \geq \mu_-(h,x_0,r)$. Analogously, if $t<0$ and $\Lambda_u>\mu$, then $m<|\Omega_{u_t}|$ and again $ J(u_t)+\mu|\Omega_{u_t}|< J(u) + \mu|\Omega_u|$, which gives that  $\Lambda_u \leq \mu_+(h,x_0,r)$. 
\end{oss}	
\begin{oss}[Monotonicity of $\mu_+$ and $\mu_-$]\label{rem:mono} We notice that the following inclusion holds:
	$$\mathcal{A}(u,x,r) \subseteq \mathcal{A}(u,x_0,r_0)\quad\text{for every}\quad B_r(x) \subset B_{r_0}(x_0).$$ 
	In particular, for every $0<h\le h_0$ and every $B_r(x) \subset B_{r_0}(x_0)$, we have
	\[ \mu_-(h_0,x_0,r_0) \leq \mu_-(h,x,r) \qquad\text{and}\qquad \mu_+(h,x,r) \leq \mu_+(h_0,x_0,r_0).\]
\end{oss} 
\begin{teo}[Convergence of the upper and the lower Lagrange multipliers]\label{thm mu+-}
Let $u$ be a solution of \eqref{e:fbp} in the bounded open set $D\subset\R^d$ and let $\Lambda_u$ be given by Proposition \ref{p:lagrange}. Then there exists a constant $r_0>0$, which depends only on $\tau,\lambda_m$ and $d$, with the following property: for every ball $B_r(x_0)\subset \R^d$ centred at $x_0 \in \partial\O_u$ with $r\leq r_0$ and such that 
\begin{equation}\label{eq hypDr}
D_r(x_0):=B_r(x_0)\cap D \text{ is connected}\qquad\text{and}\qquad 0<|\O_u\cap D_r(x_0)|<|D_r(x_0)|,
\end{equation}
we have
%$x_0 \in \partial \O_u \cap D$. 
%, depending on $u,m,\tau$ and $D$, 
\[\lim_{h\rightarrow 0} \mu_+(h,x_0,r_0) = \lim_{h\rightarrow 0} \mu_-(h,x_0,r_0) = \Lambda_u.\]
If, moreover, $D$ is of class $C^{1,1}$, then there exists a constant $r_1>0$, which depends only on $\tau,\lambda_m,d$ and $D$, such that, for every ball $B_r(x_0)$ centred at $x_0 \in \partial\O_u\cap\partial D$ with $r\leq r_1$, we have
\[ \lim_{h\rightarrow 0} \mu_-(h,x_0,r_0) = \Lambda_u. \] 
\end{teo}
\begin{proof}[Proof of Theorem \ref{thm mu+-}:]
Let $x_0 \in \partial\O_u$ be such that \eqref{eq hypDr} holds and let $h>0$ be small. We set for simplicity $r=r_0$, $B_r(x_0) = B_r$, $\mu_+(h):=\mu_+(h,x_0,r)$ and  $\mu_-(h):=\mu_-(h,x_0,r)$. 
%We first notice that 
%%since $x_0\in\partial\Omega_u$, we have $0<|B_r\cap\Omega_u|<|B_r|$. Thus, 
%for $h$ small enough we have 
%\begin{equation}\label{eq hyp EL}
%\!\!0<|\O_v \cap D_r(x_0)| < |D_r(x_0)|, \quad \forall v \in \adm\quad\text{such that}\quad m - |\O_u \cap B_r| \leq |\O_v| \leq m+h.
%\end{equation}
We proceed in three steps.

\noindent \textbf{Step 1.} We first prove that $\mu_+(h)$ is finite.
Let, for any $n\in\N$, $v_n \in \adm$ be a solution of the variational problem
\begin{equation}\label{eq min vn}
\min \big\{J(v)+n(|\O_{v}| - m)_+ \ : \  v \in \adm, \ |\O_v| \leq m+h \big\}.
\end{equation}
If there exists $n$ such that $|\O_{v_n}| \leq m$, then $\mu_+(h) \leq n$ and hence $\mu_+(h)$ is finite. Indeed, by the minimality of $u$ and the definition of $v_n$, we have for every $v \in \adm$ such that $m\leq|\O_v|\leq m+h$
\begin{equation*}
J(u) + n|\O_u| \leq J(v_n)+n|\O_u| \leq J(v) + n|\O_v|,
\end{equation*}
so that $\mu_+(h) \leq n$ and the inequality   $\mu_+(h)<\infty$   holds. 

\noindent Suppose, by contradiction, that $|\O_{v_n}|>m$ for every $n$. First notice that since $J(v_n)$ is bounded from below (see Remark \ref{rem:J_bounded}) and $J(v_n)+n(|\Omega_{v_n}|-m)\le J(u)$, we have that $|\Omega_{v_n}|\to m$ as $n\to\infty$.
Since $v_n$ is  a solution of \eqref{e:fbpJn_le} with $m_n:=|\O_{v_n}|$, there is a Lagrange multiplier $\Lambda_{v_n}$ such that \eqref{e:lagrange} holds for $v_n$ and a vector field $\xi_n$ such that \eqref{e:u_n_derivatives} holds for $v_n^t(x) = v_n(x+ t\xi_n(x))$. For $t>0$ small enough, $v_n^t \in \adm$ and $m<|\O_{v_n^t}|<m+h$. Then, by the minimality of $v_n$ we have
\[ J(v_n) + n(|\O_{v_n}|-m) \leq J(v_n^t) + n(|\O_{v_n^t}|-m) =J(v_n) + t\Lambda_{v_n} + n(|\O_{v_n}| -t -m) +o(t), \]
which implies $n \leq \Lambda_{v_n}$, in contradiction with $\ds\lim_{n\rightarrow\infty} \Lambda_{u_n}= \Lambda_u$ from Lemma \ref{l:u_n}. 

%We now prove that $\mu_+(h)>0$. Indeed, if $\mu_+(h)=0$, then we get
%\[ J(u) \leq J(u+t\varphi), \quad \text{for every}\quad \varphi \in C^\infty_c(B_r)\quad\text{such that}\quad |\{\varphi>0\}|<h. \] 
%This implies that, for $s>0$ small enough, $u$ is a solution of the equation $-\dive(e^{-\Phi} \nabla u) = \lambda_m e^{-\Phi} u$ in $B_r$. Since $u\geq 0$ in $D$, by the strong maximum principle, we get that $u>0$ in $B_r$. This contradicts Lemma \ref{lem J bounded}.

\noindent \textbf{Step 2.} $\ds\lim_{h\rightarrow0}\mu_+(h) = \Lambda_u$.
%We now prove that $\ds\lim_{h\to0}\mu_+(h) = \Lambda_u$. 
Let   $\left(h_n\right)_{n\geq 1}$   be a decreasing sequence such that $h_n\to0$. Since $\Lambda_u \leq \mu_+(h)$ and $h \mapsto \mu_+(h)$ is non-decreasing, it is sufficient to prove that $\ds\lim_{n\to\infty}\mu_+(h_n)=\Lambda_u$. Fix $\varepsilon \in (0, \Lambda_u)$ and let $0<\alpha_n:=\mu_+(h_n) - \varepsilon < \mu_+(h_n).$
Let $u_n$ be the solution 
%(which exists due to \eqref{eq J bounded}) 
of the problem
\[ \min \big\{ J(v) + \alpha_n(|\O_v|-m)^+ \ : \ v \in \adm, \ |\O_v| \leq m+h_n \big\}. \]
Notice that $|\O_{u_n}|>m$, since otherwise we would have $J(u) \leq J(u_n) + \alpha_n(|\O_{u_n}|-m)^+$, which contradicts the definition of $\mu_+(h_n)$.
%The proof is now very similar to the one in the previous part. 
For $n$ large enough, \eqref{eq hypDr} holds with $u_n$, and since $u_n$ is solution of \eqref{e:fbpJn_le} with $m_n=|\Omega_{u_n}|$,
%\[ J(u_n) = \min \big\{ J(v) \ : \ v \in \adm,\ \ |\O_v| \leq |\O_{u_n}| \big\}, \]
by Proposition \ref{p:lagrange}, there is a Lagrange multiplier $\Lambda_{u_n}\geq0$ and a vector field $\xi_n$ such that \eqref{e:u_n_derivatives} holds for $u_n^t(x):=u_n(x+t\xi_n(x))$. By the minimality of $u_n$, for $t>0$ small enough, we have
\[ J(u_n) + \alpha_n(|\O_{u_n}|-m) \leq J(u_n^t) + \alpha_n(|\O_{u_n^t}|-m) =J(u_n) +t\Lambda_{u_n} + \alpha_n(|\O_{u_n}|-t-m) + o(t), \]
which shows that $\Lambda_{u_n}\geq \alpha_n$. By Lemma \ref{l:u_n} we have 
\[ 
\lim_{n\to\infty}\mu_+(h_n) - \varepsilon = \lim_{n\to\infty}\alpha_n \leq \lim_{n\to\infty}\Lambda_{u_n} = \Lambda_u,
\]
which proves the claim since $\eps>0$ is arbitrary.

\noindent \textbf{Step 3.} $\ds\lim_{h\rightarrow0}\mu_-(h)=\Lambda_u$.
We prove this result for any $x_0 \in \partial \O_u$, which will conclude the proof of the Theorem. Note that the smoothness of $D$ implies that there exists a constant $c_D>0$ such that $D_r(x_0)$ is connected for every $r\leq r_D$ and every $x_0 \in \partial\O_u\cap\partial D$. 

Let $\varepsilon>0$ and $(h_n)_{n\in\N}$ be a decreasing infinitesimal sequence. We will show that 
%Since $\mu_-(h) \leq \Lambda_u$ and $h\mapsto \mu_-(h)$ is non-increasing, it is sufficient to show that 
$\ds\Lambda_u - \varepsilon \leq \lim_{n\to\infty}\mu_-(h_n)$.
Let $u_n$ be a solution of the problem
\begin{equation}\label{eq min vn 2}
\min \big\{J(v) + (\mu_-(h_n)+\varepsilon)(|\O_{v}|-(m-h_n))^+ \ : v \in \adm, \ |\O_v| \leq m \big\}.
\end{equation}
Up to replacing $u_n$ by   $u_n^+$,   we can assume that $u_n \geq 0$ in $B_r$   (the argument is similar to the proof of Lemma \ref{lm:localminpb}).   We claim that 
\begin{equation}\label{e:bla}
m-h_n \leq |\O_{u_n}| <m.
\end{equation}
Suppose that $|\O_{u_n}|= m$. By the minimality of $u$ and $u_n$ we get
\[ J(u) + (\mu_-(h_n)+\varepsilon)|\O_u| \leq J(u_n) + (\mu_-(h_n)+\varepsilon)|\O_{u_n}| \leq J(v) + (\mu_-(h_n) + \varepsilon)|\O_v|, \]
for every $v \in \adm$ such that $m-h_n \leq |\O_v|\leq m$, which contradicts the definition of $\mu_-(h_n)$.
Now, if $|\O_{u_n}|<m-h_n$, we have $J(u_n) \leq J(u_n+t\varphi)$ for every $\varphi \in C^\infty_c(D_r(x_0))$ with sufficiently small compact support.
% such that $|\{ \varphi\neq0 \}|<a-h_n-|\O_{u_n}|$. 
Thus $u_n$ solves the PDE $-\dive(e^{-\Phi}\nabla u_n) = \lambda_m e^{-\Phi}u_n$ in $D_r(x_0)$. Since $u_n\geq 0$ in $D_r(x_0)$, by the strong maximum principle, we have that either $u_n\equiv0$ or $u_n >0$ in $D_r(x_0)$, in contradiction with \eqref{eq hypHLun}. Thus, we proved \eqref{e:bla}.

We have that $u_n$ is solution of \eqref{e:fbpJn_le} with $m_n:=|\Omega_{u_n}|$ which converges to $m$ as $n\to\infty$.
%\[ J(u_n) = \min \big\{ J(v)  : \ v \in \adm, \ |\O_v| \leq |\O_{u_n}| \big\}, \]
By Lemma \ref{l:u_n}, we have an Euler-Lagrange equation for $u_n$ in $B_r$ for some $\Lambda_{u_n}$. Let $\xi_n \in C^\infty_c(D_r(x_0);\R^d)$ be the vector field from Lemma \ref{l:u_n} (b) and let $u_n^t(x) = u_n(x+t\xi_n(x))$. For negative $t<0$ and $|t|$ small enough, $u_n^t \in \adm$ and  $|\O_{u_n}|\le |\O_{u_n^t}|<m$. Thus, by the minimality of $u_n$, we get
\[ J(u_n) + (\mu_-(h_n) + \eps)(|\O_{u_n}|-(m-h_n)) \leq J(u_n) +\Lambda_{u_n}t + (\mu_-(h_n) + \eps)(|\O_{u_n}|-t-(m-h_n)) + o(t), \]
which implies that $\Lambda_{u_n} \leq \mu_-(h_n) + \eps$. Now, by Lemma \ref{l:u_n}, we get
\[ \Lambda_u=\lim_{n\to\infty}\Lambda_{u_n} \leq \lim_{n\to\infty}\mu_-(h_n) + \varepsilon, \] 
which conclude the proof.
\end{proof}

\begin{oss}[Quasi-minimality at small scales]\label{rem:uniform} Suppose that $D\subset\R^d$ is just a bounded open set. By the monotonicity of $\mu_+$ and $\mu_-$ with respect to the inclusion (Remark \ref{rem:mono}) and a covering argument we get that for every compact set $\mathcal K\subset D$ there is $r(\mathcal K)>0$ such that: for every $\eps>0$ there is $h>0$ such that 
	$$\mu_+(h,x,r)-\eps \le\Lambda_u \le \mu_-(h,x,r)+\eps \quad\text{for every}\quad x \in \mathcal K\cap\partial\Omega_u\quad\text{and every}\quad 0<r \leq r(\mathcal K).$$
If, moreover, $D$ is of class $C^{1,1}$, then 
%that for every $x_0\in\overline D$ there is $r_0=r(x_0)>0$ such that for every $0<r\le r_0$ the set $D_r(x_0):=B_r(x_0)\cap D$ is connected, 
%\begin{align}
%\text{for every $x_0\in\partial\O_u$}&\text{ there is $r_0=r(x_0)>0$ such that}\notag\\
%&\text{for every $0<r\le r_0$ the set $D_r(x_0):=B_r(x_0)\cap D$ is connected},\label{e:rem:regD}
%\end{align}
then exists $r_D>0$ such that, for every $\eps>0$ there exists $h>0$ such that: for every $0<r\leq r_D$ and every $x_0\in\partial\O_u$ we have
\begin{align*}
\mu_+(h,x,r)-\eps \le \Lambda_u \le \mu_-(h,x,r)+\eps \quad&\text{if } |\O_u\cap D_r(x_0)|<|D_r(x_0)|, \\
\Lambda_u \le \mu_-(h,x,r)+\eps \quad&\text{otherwise.}
\end{align*}
\end{oss} 

\subsection{Lipschitz continuity of the eigenfunctions on the optimal sets}\label{sub:lipschitz}
In this subsection we prove that the solutions of \eqref{e:fbp} are (locally) Lipschitz continuous in $D$. For $\delta>0$ we set $D_\delta=\{ x \in D \ : \ d(x,\partial D)>\delta\}$ and let $\mu>0$ be fixed. By Theorem \ref{thm mu+-} and Remark \ref{rem:uniform} we get that if $u$ is a solution of \eqref{e:fbp} and $\mu>\Lambda_u$, then there is $r_0>0$ such that, for every $x_0\in\partial\O_u\cap D_\delta$, we have
\begin{equation}\label{e:supersolution}
J(u)+\mu|\Omega_u|\le J(v)+\mu |\Omega_v|\quad\text{for every}\quad v\in \mathcal{A}(u,x_0,r_0)\quad\text{such that}\quad |\Omega_v|\ge |\Omega_u|.
\end{equation}
Note that the condition $|\O_v|\leq |\O_u|+h$ can be dropped by choosing $r_0$ such that $|B_{r_0}|\leq h$.
We will prove that if $u\in H^1(B_{r_0})$ is bounded, nonnegative and satisfies \eqref{eq id moy} and \eqref{e:supersolution}, then $u$ is Lipschitz in $D_\delta$. 
%We also show that if $D$ is sufficiently regular, then the optimal eigenfunctions are globally Lipschitz continuous in $\overline D$. 
\begin{prop}[Lipschitz continuity of the eigenfunctions on the optimal sets]\label{p:lipschitz}
Let $D\subset\R^d$ be a bounded open set. Let $\tau\geq 0$, $m \in (0,|D|)$ and $\Phi \in W^{1,\infty}(D)$.
	Then, every solution of \eqref{e:fbp} is locally Lipschitz continuous in $D$. More precisely, it is Lipschitz in $D_{\delta}$ for all $\delta>0$. 
	Moreover, if the box $D$ is of class $C^{1,1}$, then $u$ (extended by $0$ outside $D$) is Lipschitz in $\R^d$.
	%Then the shape optimization problem \eqref{eq min lambda} has a solution $\O^\ast \subset D$ which is an open set. Moreover, every of $L$ in $\O^\ast$ associated to $\lambda_1(\O^\ast,v)$ is locally Lipschitz continuous in $D$.
\end{prop}

%In the next lemma we will show that we can furtherly restrict our attention to the case when $\lambda_m=0$ at least when we consider the local problem \eqref{e:fbpJ}. Before we give the precise statement we introduce several notations.  First of all we recall that every $\tau$-Lipschitz function defined on $D$ can be extended to a $\tau$-Lipschitz function on $\R^d$. Thus, we may suppose that the potential $\Phi$ is defined on $\R^d$. Next, we notice that for every $x_0\in \R^d$, there is some $R=R(x_0)>0$ depending on $\Phi$ such that $\lambda_1(B_R(x_0),\nabla\Phi)=\lambda_m$. Also, the corresponding first eigenfunction $\varphi\in H^1_0(B_R(x_0))$ is strictly positive and $C^{2,\alpha}$ regular in $B_R(x_0)$. 

%Now, for $x_0$ and $R=R(x_0)$ as above and every $v\in H^1(B_{R/2}(x_0))$,  we set 
%$$\tilde J(v)=\int_{B_{R/2}(x_0)}|\nabla v|^2e^{-\Psi}\,dx,$$
%where $\Psi:B_R(x_0)\to\R$ is such that $e^{-\Psi}=\varphi^2e^{-\Phi}$. 

%Next, for every $v\in H^1(B_{R/2}(x_0))$ we denote by $\tilde v\in H^1(B_{R/2}(x_0))$ the function $\tilde v=v/\varphi$. A straightforward computation gives that 
%$$\tilde J(\tilde v)=J(v)\qquad\text{for every}\qquad v\in H^1(B_{R/2}(x_0)).$$

The proof is based on the following lemma, whose (more general) two-phase counterpart can be found for instance in \cite{alt-caffarelli-friedman}, \cite{briancon-hayouni-pierre} and \cite{bucur-mazzoleni-pratelli-velichkov}.
\begin{lm}[A bound on the measure $\text{\it div}\,(e^{-\Phi}\nabla u)$]\label{l:laplacian}
	Let $u$ be a solution of \eqref{e:uminJ} and let $r_0>0$ be such that $u$ satisfies \eqref{e:supersolution} for some $\mu>\Lambda_u$. Then, there is a constant $C>0$ such that for every $x_0 \in \partial\O_u\cap D_\delta$ we have
	%for every ball $B_{2r}(x) \subset B_{r_0}(x_0)$, we have
		\begin{equation}\label{eq esti div}
		|\dive(e^{-\Phi} \nabla u)|(B_r(x)) \leq Cr^{d-1}\qquad\text{for every ball}\qquad B_{2r}(x) \subset B_{r_0}(x_0).
		\end{equation}
\end{lm}

\begin{proof}
	%Without loss of generality, we assume $x_0=0$. Let $\lambda>0$ and the functional $J_\lambda$ be as in Lemma \ref{lm min u relax}.
	Let $x=0$ and $\eta \in {C}^\infty_c(B_{2r})$ be such that
	\[ 0 \leq \eta \leq 1\ \text{ in }\ B_{2r}\ ,\qquad \eta = 1\ \text{ in }\ B_r\ , \qquad \|\nabla \eta\|_{L^\infty} \leq \frac{C_d}r\ .\]
	Using $u+t\eta$ as a test function for $J$, and setting $\ds\langle f,g\rangle:=\int_D fg\,dx$, we get
	\begin{align*}
	2 \langle \dive(e^{-\Phi} \nabla u) + \lambda_m u e^{-\Phi},\eta \rangle &\le tJ(\eta)+\frac{\mu}t |B_{2r}|
	%&\le 2\lambda_m \int_{B_{2r}} u \varphi e^{-\Phi} dx + t \int_{B_{2r}} |\nabla\varphi|^2 e^{-\Phi} dx + \frac{\lambda}{t}|\O_\varphi| \\
	 \le C\left( t\|\nabla\eta\|_{L^2}^2  + \frac{r^d}{t}\right)
	\end{align*}
	where the constant $C>0$ depends on $d$, $\Phi$ and $\mu$. Now, minimizing over $t>0$ and using the estimate $\|\nabla\eta\|_{L^2} \leq C_d r^{\frac{d}{2}-1}$, we get 
	\begin{align*}
	\langle \dive(e^{-\Phi} \nabla u) + \lambda_m u e^{-\Phi},\eta \rangle \le Cr^{d-1}.
	\end{align*}
	By Lemma \ref{l:radon}, we have that $\dive(e^{-\Phi} \nabla u) + \lambda_m u e^{-\Phi}$ is a positive Radon measure. Thus, the inequality $  \eta\ge \ind_{B_r}$   and the boundedness of $u$ imply
	\begin{align*}
	|\dive(e^{-\Phi} \nabla u)|(B_r)&\le \lambda_m \int_{B_r}u e^{-\Phi}\,dx+\langle \dive(e^{-\Phi} \nabla u) + \lambda_m u e^{-\Phi},\ind_{B_r} \rangle \le Cr^{d-1}.\qedhere
	\end{align*}
\end{proof}

The main ingredients of the proof of Proposition \ref{p:lipschitz} will be Lemma \ref{l:laplacian} and the classical gradient estimate that we recall in the lemma below. 

\begin{lm}[Gradient estimate]\label{l:grad_est}
Let $p>d$. Let $\mathcal U$ be an open subset of $\R^d$ and let $u\in W^{2,p}_{loc}(\mathcal U)\cap L^p(\mathcal U)$ be a (strong) solution to the equation 
$$\sum_{i,j=1}^d a_{ij}(x)\partial_{ij}u+\sum_{i=1}^d b_{i}(x)\partial_iu=f\qquad \text{in}\qquad \mathcal U,$$
where $f\in L^p(\mathcal U)$ and we suppose that: 
\begin{enumerate}[(a)]
\item the functions $a_{ij}:\mathcal U\to\R$ are H\"older continuous, that is, there are constants $C_a>0$ and $\delta_a>0$ such that 
$$|a_{ij}(x)-a_{ij}(y)|\le C_a|x-y|^{\delta_a}\qquad\text{for every}\qquad x,y\in\mathcal U\,;$$
\item there is a constant $M>0$ such that 
$$\|a_{ij}\|_{L^\infty(\mathcal U)}\le M\quad\text{and}\quad \|b_{i}\|_{L^\infty(\mathcal U)}\le M\qquad\text{for every}\qquad 1\le i,j\le d\,;$$
\item the matrix is $(a_{ij})_{ij}$ is uniformly elliptic, that is, there is a constant $c_a>0$ such that 
$$\sum_{i,j=1}^da_{ij}(x)\xi_i\xi_j\ge c_a|\xi|^2\qquad\text{for every}\qquad x\in\mathcal U\quad\text{and}\quad\xi=(\xi_1,\dots,\xi_d)\in\R^d\,.$$
\end{enumerate}
Then, for any domain $\mathcal U'\subset\subset\mathcal U$, we have 
$$\|u\|_{C^{1,\gamma}(\mathcal U')}\le C\left(\|u\|_{L^p(\mathcal U)}+\|f\|_{L^p(\mathcal U)}\right)\qquad\text{where}\qquad \gamma=1-\frac{d}{p},$$
and $C$ is a constant depending on $d$, $p$, $M$, $C_a$, $\delta_a$, $c_a$, $\mathcal U'$ and $\,\mathcal U$. 
\end{lm}
\begin{proof}
First notice that by \cite[Theorem 9.11]{GT}, there is a constant $C'$ such that 
$$\|u\|_{W^{2,p}(\mathcal U')}\le C'\left(\|u\|_{L^p(\mathcal U)}+\|f\|_{L^p(\mathcal U)}\right).$$
Now, the claim follows by the Sobolev inequality (see \cite[Section 5.6, Theorem 6]{evans}).
\end{proof}

\begin{proof}[Proof of Proposition \ref{p:lipschitz}]
Let $u$ be a solution of \eqref{e:fbp}. We proceed in four steps.

\noindent{\it Step 1. $\Omega_u$ is open.} Let $\bar x\in \partial\Omega_u\cap D$. We will prove that $u(\bar x)=0$. Let $r_1>0$ be such that $B_{r_1}(\bar x)\subset D$ and let $x_n\in B_{\sfrac{r_1}2}(\bar x)$ be a sequence converging to $\bar x$ such that $u(x_n)=0$ (such a sequence exists by Lemma \ref{lem hypDr}). By Lemma \ref{l:laplacian}   and Lemma \ref{l:radon} \eqref{e:rate},   for every $n$ 
%any point $\bar x\in B_{r_0/2}(x_0)$ 
and every $r\le r_1/2$ we have 
$$\aver{\partial B_r(x_n)}u\,d\HH^{d-1}\le u(x_n)+Cr=Cr,$$
where the constant $C$ does not depend on $n$.
Passing to the limit as $n\to\infty$, we get that 
$$\aver{\partial B_r(\bar x)}u\,d\HH^{d-1}\le Cr\quad\text{for every}\quad r\le r_1/2,$$
which, passing to the limit as $r\to 0$, proves that $u(\bar x)=0$.

\noindent{\it Step 2. Gradient estimate in $\Omega_u$.} We claim that, for every ball $B_r(\bar x)\subset \Omega_u$, there is a constant $C_2$, depending only on $\Phi$, $d$ and $\lambda_m$, such that 
\begin{equation}\label{e:gradest}
\|\nabla u\|_{L^\infty(B_{\sfrac{r}2}(\bar x))}\le \frac{C_2}{r}\|u\|_{L^\infty(B_r(\bar x))}.
\end{equation} 
Indeed, suppose that $\bar x=0$ and set $\Phi_r(x):=\Phi(rx)$ and $u_r(x)=u(rx)$. Then $u_r$ is a solution of 
$$\dive(e^{-\Phi_r}\nabla u_r)+\lambda_m e^{-\Phi_r}u_r=0\quad\text{in}\quad B_1,$$
which can be re-written as 
$$\Delta u_r-\nabla \Phi\cdot \nabla u_r=-\lambda_m u_r\quad\text{in}\quad B_1.$$
Applying Lemma \ref{l:grad_est} with $\mathcal U=B_1$, $\mathcal U'=B_{\sfrac12}$, $a_{ij}=\delta_{ij}$, $b=\nabla\Phi$, $M=1+\|\nabla\Phi\|_{L^\infty}$, $f=-\lambda_mu$ and any $p>d$, we get  
$$\|\nabla u_r\|_{L^\infty(B_{\sfrac{1}2})}\le\|u_r\|_{C^{1,\alpha}(B_{\sfrac12})}\le  C_2\|u_r\|_{L^\infty(B_{1})},$$ 
%and so, the interior Schauder estimate (see \cite[Theorem 6.2 and Theorem 9.19]{GT}) gives 
%$$\|\nabla u_r\|_{L^\infty(B_{\sfrac{1}2})}\le\|u_r\|_{C^{2,\alpha}(B_{\sfrac12})}\le  C_2\|u_r\|_{L^\infty(B_{1})},$$ 
which, after rescaling, is precisely \eqref{e:gradest}. 
	
\noindent{\it Step 3. Proof of the local Lipschitz continuity.} Let $\bar x\in \Omega_u\cap D_\delta$ and set $r:=\text{dist}(\bar x,\partial\Omega_u)$. Let $r_0 \in (0,\sfrac\delta2)$ be such that $u$ satisfies \eqref{e:supersolution} for every point $x_0$ on $\partial\O_u\cap D_{\sfrac\delta2}$ (such an $r_0$ exists by a standard covering argument). We now consider two cases 

Case 1. If $r\geq \sfrac{r_0}6$, then the estimate \eqref{e:gradest} gives $|\nabla u(\bar x)| \leq C_{r_0}$.

Case 2. If $r\leq \sfrac{r_0}6$, let $\bar y$ be the projection of $\bar x$ on $\partial\Omega_u$, that is, $\bar y\in\partial\Omega_u$ and $r=|\bar x-\bar y|$. Notice that in this case we have that $\bar y \in \partial\O_u\cap D_{\sfrac{\delta}2}$. Now, take any $\bar z \in B_r(\bar x)$. By Lemma \ref{l:laplacian} and by the estimate \eqref{e:rate} of Lemma \ref{l:radon}, we have 
$$u(\bar z) \leq \aver{\partial B_s(\bar z)} u\,d\mathcal{H}^{d-1} + Cs\qquad\text{for every}\qquad 0<s\le r,$$
where $C$ is the constant in the right-hand side of \eqref{e:rate}.
Now, multiplying by $s^{d-1}$ and then integrating from $0$ to $r$ the above inequality, we get
\begin{align*}
\nonumber u(\bar z) \leq \aver{B_r(\bar z)} u\,dx +Cr&\leq 3^d \aver{B_{3r}(\bar y)} u\,dx +Cr \\
&= \frac{d}{r^d}\int_0^{3r}s^{d-1}\,ds\,\aver{\partial B_s(\bar y)} u\,d\mathcal{H}^{d-1} +Cr.
\end{align*}
Using again \eqref{e:rate}, this time for $\bar y$ (at which $u(\bar y)=0$ by Step 1 of the proof), we get that 
$$u(\bar z)\le 3^{d+1}Cr\qquad\text{for every}\qquad \bar z \in B_r(\bar x).$$
Finally, using the estimate \eqref{e:gradest} this gives
\begin{equation}
|\nabla u(\bar x)| \leq \|\nabla \bar u\|_{L^\infty(B_{\sfrac{r}2}(\bar x))} \leq \frac{C_2}{r} \|u\|_{L^\infty(B_r(\bar x))} \leq 3^{d+1}C_2C.
\end{equation}
This proves that $|\nabla u|$ is bounded in $D_\delta$ without assuming any regularity of $D$.
	
\noindent{\it Step 4. Global Lipschitz estimate.} We first notice that since $D$ is $C^{1,1}$ regular, the radius $r_0$ for which \eqref{e:supersolution} holds does not depend on the point $x_0\in \partial \O_u$. Now, let $\bar x\in \O_u \setminus D_{r_0}$ and set $r:=\text{dist}(\bar x,\partial\Omega_u\cap D)$. We consider the projection $\bar y$ of $\bar x$ on $\partial\Omega_u$ and we distinguish two cases. If $6r\leq \text{dist}(\bar x,\partial D)$, then we apply the estimate from Step 3 and we get that $|\nabla u(\bar x)|\le C$. If $6r\ge \text{dist}(\bar x,\partial D)$, we consider the solution $w$ to the problem
$$-\dive(e^{-\Phi}\nabla w)=1\quad\text{in}\quad D,\qquad w\in H^1_0(D),$$
which is Lipschitz continuous in $\R^d$ since $D$ is of class $C^{1,1}$ (see for example \cite[Theorem 9.13]{GT}). Moreover, by the strong maximum principle, we have that $u\le Cw$ for some constant $C$ depending on $\lambda_m$, $d$ and $\Phi$. Therefore, setting $r_1=\text{dist}(\bar x,\partial D)$, we have for every $\bar z \in B_{r_1}(\bar x)$,
\[ u(\bar z) \leq Cw(\bar z) \leq C|\bar z -\bar y| \leq C r_1, \]
and we conclude by the gradient estimate \eqref{e:gradest}. 
	%Now the bound on $|\nabla u|$ on any compact set $K\subset D$ follows by a covering argument, Theorem \ref{thm mu+-} and the fact that \eqref{e:gradest} holds away from the free boundary. 
\end{proof}

\subsection{Non-degeneracy of the eigenfunctions and finiteness of the perimeter of $\Omega_u$}\label{sub:nondegeneracy}
Let $u$ be a solution of \eqref{e:fbp} in the bounded open set $D\subset\R^d$. Let $x_0\in\partial \Omega_u$ and $r_0(x_0)$ be such that for every $0<r\le r(x_0)$ the set $D_r(x_0):=B_r(x_0)\cap D$ is connected. Notice that such an $r(x_0)$ trivially exists if $x_0\in \partial\O_u \cap D$, while in the general case it is sufficient to assume some a priori regularity of the box $D$. Then, by Remark \ref{rem:uniform}, for every $\mu<\Lambda_u$ there is some $r_0>0$ such that, for every $x_0 \in \partial\O_u$, we have
\begin{equation}\label{e:subsolution}
J(u)+\mu|\Omega_u|\le J(v)+\mu|\Omega_v|\quad\text{for every}\quad v\in\mathcal{A}(u,x_0,r_0)\quad\text{such that}\quad |\Omega_v|\le|\Omega_u|. 
\end{equation}
This property was first exploited by Alt and Caffarelli to prove the non-degeneracy of the solutions. More recently, it was exploited by Bucur  who introduced the notion of a {\it shape subsolution}  which found application to several shape optimization problems (see for example \cite{bucur}  and \cite{bucur-velichkov}).  

Lemma \ref{l:nondegeneracy} below is a fundamental step in the proof of the regularity of the free boundary since it allows to prove that the blow-up limits (see Subsection \ref{sub:blowup}) are non trivial. It is the analogue of the non-degeneracy estimate from \cite{alt-caffarelli} and the proof is based on the same idea. Before we state it, we recall the following boundary estimate for solutions to elliptic PDEs. 
\begin{lm}[Boundary gradient estimate]\label{l:boundary_grad_est}
Let $p>d$. Let $\mathcal U$ be a bounded connected open subset of $\R^d$ with $C^{1,1}$ boundary. Let $T_1,\dots, T_k$ be the connected components of the boundary $\partial\mathcal U$ and let $c_1,\dots,c_k$ are given constants. Let $u\in W^{2,p}_{loc}(\mathcal U)\cap L^p(\mathcal U)$ be a (strong) solution to the problem 
$$\sum_{i,j=1}^d a_{ij}(x)\partial_{ij}u+\sum_{i=1}^d b_{i}(x)\partial_iu=f\quad \text{in}\quad \mathcal U,\qquad u=c_i\quad \text{in}\quad T_i\, ,\quad i=1,\dots,k\,,$$
where $f\in L^p(\mathcal U)$ and we suppose that $A=(a_{ij})_{ij}$ and $b=(b_1,\dots, b_d)$ satisfy the conditions (a), (b) and (c) of Lemma \ref{l:boundary_grad_est}. Then, we have 
$$\|u\|_{C^{1,\gamma}(\mathcal U)}\le C\left(\|u\|_{L^p(\mathcal U)}+\|f\|_{L^p(\mathcal U)}\right)\qquad\text{where}\qquad \gamma=1-\frac{d}{p},$$
and $C$ is a constant depending on $d$, $p$, $M$, $C_a$, $\delta_a$, $c_a$ (defined in Lemma \ref{l:grad_est}), and $\,\mathcal U$. 
\end{lm}
\begin{proof}
By \cite[Theorem 9.13]{GT}, there is a constant $C'$ such that 
$$\|u\|_{W^{2,p}(\mathcal U)}\le C'\left(\|u\|_{L^p(\mathcal U)}+\|f\|_{L^p(\mathcal U)}\right).$$
The claim follows by the Sobolev inequality (see for instance \cite[Section 5.6, Theorem 6]{evans}).
\end{proof}

\begin{lm}[Non-degeneracy of the eigenfunctions on the optimal sets]\label{l:nondegeneracy}
Let $u$ be a solution of \eqref{e:fbp} in the bounded open set $D\subset\R^d$. Suppose that $x_0\in\partial\O_u$, $0<\mu<\Lambda_u$ and $r_0>0$ are such that \eqref{e:subsolution} holds. Then there are constants $c>0$ and $r_1>0$ which depend only on $\tau,\lambda_m,\mu$ and $d$, such that for every ball $B_{2r}(x)\subset B_{r_0}(x_0)$ with $r\leq r_1$, we have that:
$$\text{If}\quad \|u\|_{L^\infty(B_{2r}(x))}\le cr\, ,\quad\text{then}\quad u=0\quad\text{in} \quad B_{r}(x).$$    
     %Density estimates:
    %\[ c \leq \frac{|\O_u \cap B_r(x_0)|}{|B_r(x_0)|} \leq 1-c. \]
\end{lm}

\begin{proof}
Let $r,x$ be such that $B_{2r}(x)\subset B_{r_0}(x_0)$ and $\|u\|_{L^\infty(B_{2r}(x))} < cr$.  Assume for simplicity that $x=0$. Let $\eta \in H^1(B_{2r})$ be the solution of the problem
$$-\dive(e^{-\Phi} \nabla\eta) = \beta e^{-\Phi}\quad\text{in}\quad B_{2r}\backslash B_{r},\qquad  
\eta = 0\quad\text{in}\quad B_{r},\qquad\eta = c r\quad\text{in}\quad D \backslash B_{2r},$$
where $\beta>0$ will be chosen later.
Consider the test function $\tilde u\in H^1_0(D)$ defined as
$$\tilde u=u \wedge \eta \quad\text{in}\quad B_{2r},\qquad  \tilde u=u\quad\text{in}\quad D \backslash B_{2r}.$$
By \eqref{e:subsolution}, we get 
\begin{equation}\label{e:ter}
\int_D|\nabla u|^2e^{-\Phi}dx - \lambda_m\int_Du^2e^{-\Phi}dx +\mu|\O_u| \leq \int_D|\nabla \tilde u|^2e^{-\Phi}dx - \lambda_m\int_D\tilde u^2e^{-\Phi}dx +\mu|\O_{\tilde u}| . 
\end{equation}
Let
$\ds E(u,r):=\int_{B_{r}}|\nabla u|^2 e^{-\Phi}dx + \mu|\O_u \cap B_{r}|$. Since $\tilde u\equiv0$ in $B_r$ we have $|\O_u| - |\O_{\tilde{u}}| = |\O_u \cap B_{r}|$ and $\ds\int_{B_r}\tilde u^2\,dx=\int_{B_r}|\nabla\tilde u|^2\,dx=0$. Thus, we can rewrite \eqref{e:ter} in the form 
\begin{equation}\label{eq 6}
 E(u,r) \leq \int_{B_{2r} \backslash B_{r}} \big(|\nabla\tilde{u}|^2 - |\nabla u|^2\big)e^{-\Phi}dx +4cr\lambda_m\int_{B_{2r} \backslash B_{r}} (u-\tilde{u})e^{-\Phi}dx + \lambda_m\int_{B_{r}} u^2 e^{-\Phi}dx,
\end{equation}
where in the estimate of the second term we used that in $B_r$
$$u^2-\tilde u^2=(u+\tilde u)(u-\tilde u)\le 2u(u-\tilde u)\le 2cr(u-\tilde u).$$
Next, we estimate the first term of the right-hand side of \eqref{eq 6}. We have
\begin{equation}\label{eq 7}
 |\nabla \tilde{u}|^2 - |\nabla u|^2 = -|\nabla(\tilde{u}-u)|^2 + 2\nabla\tilde{u} \cdot \nabla(\tilde{u}-u) \leq 2\nabla\tilde{u} \cdot \nabla(\tilde{u}-u).
\end{equation}
Integrating by parts and using that $(u-\eta)_+ =0$ on $\partial B_{2r}$, we get
\begin{align}
 \int_{B_{2r} \backslash B_{r}} \nabla \tilde{u} \cdot \nabla (\tilde{u} - u)e^{-\Phi}dx &= -\int_{B_{2r} \backslash B_{r}} \nabla \eta \cdot \nabla[(u-\eta)_+]e^{-\Phi}dx \label{eq 8}\\
&   \leq   -\beta \int_{B_{2r} \backslash B_{r}} (u-\eta)_+ e^{-\Phi}dx + \|\nabla\eta\|_{L^\infty(\partial B_{r})} \int _{\partial B_{r}} ue^{-\Phi}d\mathcal{H}^{d-1}. \notag
\end{align}
%where in the last inequality we have integrated by parts and used that $(u-\eta)_+ =0$ on $\partial B_{2r}$. 
We now set $\beta =2cr\lambda_m$ so that, combining \eqref{eq 6}, \eqref{eq 7} and \eqref{eq 8} we have
$$E(u,r) \leq 2\|\nabla\eta\|_{L^\infty(\partial B_{r})}\int _{\partial B_{r}} ue^{-\Phi} d\mathcal{H}^{d-1} + \lambda_m\int_{B_{r}} u^2 e^{-\Phi}dx.$$
Now, for every $s \in (0,r]$, we have by the $W^{1,1}$ trace inequality in $B_s$ 
\begin{align*}
\int_{\partial B_{s}} ue^{-\Phi} d\mathcal{H}^{d-1} &\leq  e^{-\min\Phi} C_d \left( \int _{B_{s}} |\nabla u|\,dx + \frac{1}{s}\int_{B_{s}} u\,dx \right) \\
& \leq  e^{-\min\Phi}C_d\left( \frac12 \int_{B_{s}} |\nabla u|^2 dx + \frac12 |\O_u \cap B_{s}| + c |\O_u \cap B_{s}| \right) \\
& \leq C \left( \int_{B_s} |\nabla u|^2 e^{-\Phi}dx + \mu |\O_u \cap B_s| \right) \leq C E(u,s)\leq C E(u,r),
\end{align*}
where we have set $C = e^{-\min\Phi} C_d \max \left\{ e^{\max\Phi}, \frac{1}{\mu}(1+2c) \right\}$.
Moreover, since the above inequality holds for every $s \in (0,r]$, we have
\[ \int_{B_{r}} ue^{-\Phi} dx = \int_0^{r} ds \int_{\partial B_s} ue^{-\Phi} d\mathcal{H}^{d-1} \leq r C E(u,r). \]
Finally, using the bound \eqref{eq bound grad eta}, we get
$$E(u,r) \leq \left(2\|\nabla\eta\|_{L^\infty(\partial B_{r})} + r^2 c \lambda_m\right)C E(u,r).$$
Thus, the claim will follow, if we can choose $c$ such that 
$$\left(2\|\nabla\eta\|_{L^\infty(\partial B_{r})} + r^2 c \lambda_m\right)C<1.$$
We  now estimate $\|\nabla\eta\|_{L^\infty(\partial B_{r})}$. Notice that in $B_{2r}\setminus B_r$
$$\eta(x)=r^2w(\sfrac{x}r)+crh(\sfrac{x}r),$$
where $w:B_2\setminus B_1\to\R$ and $h:B_2\setminus B_1\to\R$ are the solutions to 
$$-\Delta h+\nabla\Phi_r\cdot\nabla h = 0\quad\text{in}\quad B_{2}\backslash B_{1},\qquad  
h = 0\quad\text{on}\quad \partial B_{1},\qquad h = 1\quad\text{on}\quad \partial B_{2},$$
$$-\Delta w+\nabla\Phi_r\cdot\nabla w = \beta \quad\text{in}\quad B_{2}\backslash B_{1},\qquad  
w = 0\quad\text{on}\quad \partial B_1 \cup \partial B_2,$$
where $\Phi_r(x)=\Phi(rx)$. Thus, we have 
$$\|\nabla \eta\|_{L^\infty}\le r\|\nabla w\|_{L^\infty}+c\|\nabla h\|_{L^\infty},$$
and so, it is sufficient to estimate $\nabla w$ and $\nabla h$. First, applying Lemma \ref{l:boundary_grad_est} to $h$, we have 
$$\|\nabla h\|_{L^\infty(B_2\setminus B_1)}\le C\|h\|_{L^\infty(B_2\setminus B_1)}\le C,$$
where the last inequality follows by the maximum principle ($0\le h\le 1$ in $B_2\setminus B_1$). Next, applying Lemma \ref{l:boundary_grad_est} to $w$, we get 
$$\|\nabla w\|_{L^\infty(B_2\setminus B_1)}\le C\left(\|w\|_{L^\infty(B_2\setminus B_1)}+\beta\right)\le C\beta,$$
where the last inequality follows by Lemma \ref{infttybndflem}. Combining the above estimates, we get 
	\begin{equation}\label{eq bound grad eta}
	\|\nabla \eta \|_{L^\infty(B_{r})} \leq C\left(r\beta +c\right),
	\end{equation} 
which, for $c$ and $r$ small enough, implies that $E(u,r)=0$ and concludes the proof.
\end{proof}

Another consequence of   property   \eqref{e:subsolution} is that the optimal sets have finite perimeter. This fact is of independent interest but it can also be used to estimate the dimension of the singular set of the free boundary   (see Subsection \ref{sub:regularity}).   The local finiteness of the perimeter was also obtained in \cite{alt-caffarelli} in the case of the Laplacian by a different argument. Here we use the more direct approach from \cite{mazzoleni-terracini-velichkov-2}, which is also the local version of an estimate that was used in \cite{bucur} to prove that some optimal shapes have finite perimeter. 
{\begin{lm}[Local finiteness of the perimeter]\label{l:finite_perimeter}
Let $D\subset\R^d$ be a bounded open set and $u$ a solution of \eqref{e:fbp}. Then $\Omega_u$ is a set of locally finite perimeter in $D$. Moreover, if $D$ is of class $C^{1,1}$, then $\Omega_u$ is a set of finite perimeter. 
%\[ \mathcal{H}^{d-1}(\text{Sing}(\O^\ast)) = 0. \]
\end{lm}

\begin{proof}
Let $x_0\in \partial\Omega_u$ and $0<\mu<\Lambda_u$ be fixed. Let $r>0$ be such that \eqref{e:subsolution} holds in $D_r(x_0):=B_r(x_0)\cap D$.   
Assume $x_0 = 0$ and $r_0=r$. In the sequel we denote by $C>0$ any constant, which does not depend on $t$ or $x_0$. 
Let $t\in(0,1)$ and $\eta \in C^\infty_c(B_r)$ be such that
\[ 0 \leq \eta \leq 1, \qquad \eta = 1 \quad\text{in}\quad B_{\sfrac{r}2}, \qquad \eta= 0 \quad\text{in}\quad \R^d \backslash B_r, \qquad |\nabla \eta | \leq \frac{C}{r}. \]
We set 
$$u_t := \eta (u-t)_+ + (1-\eta)u =\left\{
\begin{split}
(1-\eta)u,\quad \mbox{if }u<t,\\
u-t\eta,\quad \mbox{if }u\geq t.
\end{split}
\right.$$
We can now compute on $\{u\geq t\}$ 
\begin{align*}
|\nabla u|^2-|\nabla u_t|^2&=|\nabla u|^2-|\nabla (u-t\eta)|^2
=t\big(2 \nabla u\cdot\nabla\eta-t|\nabla \eta|^2\big);\\
u^2-u_t^2&=u^2-(u-t\eta)^2=t\big(2u\eta+t\eta^2\big).
\end{align*}
%where $C_1$ depends only on $\|\nabla\phi\|_{L^\infty}$ and $\|\nabla U\|_{L^\infty}$.
Next, on the set $\{u<t\},$ we compute

\begin{align*}
u^2-u_t^2&=(2\eta-\eta^2)u^2\le t^2(2\eta-\eta^2);\\
|\nabla u|^2-|\nabla u_t|^2&=|\nabla u|^2-|\nabla(u-u\eta)|^2=2\nabla u\cdot\nabla (u\eta)-|\nabla (u\eta)|^2\\
%&=2\eta|\nabla u|^2+2u\nabla u\cdot \nabla \eta - |\nabla u|^2\eta^2-|\nabla\eta|^2u^2-2u\eta\nabla u\cdot\nabla\eta\\
&=(2\eta-\eta^2)|\nabla u|^2+2u(1-\eta)\nabla u\cdot \nabla \eta -|\nabla\eta|^2u^2\\
&\ge\ind_{\{\eta=1\}}|\nabla u|^2-2t(1-\eta)|\nabla u\cdot \nabla \eta|.
\end{align*}
Notice that $u_t\in \adm$. Thus, by the optimality of $u$, we have
$$\int_{B_r} \left(|\nabla u|^2-\lambda_mu^2\right)e^{-\Phi}dx+\mu|\Omega_u\cap B_r|\le \int_{B_r} \left(|\nabla u_t|^2-\lambda_mu_t^2\right)e^{-\Phi}dx+\mu|\Omega_{u_t}\cap B_r|.$$
By the above estimates, there is a constant $C$, depending only on $\mu$, $r$, $\ds\lambda_m=\int_D|\nabla u|^2e^{-\Phi}\,dx$ and $\|\Phi\|_{L^\infty(D)}$ such that, for every $t\le 1$, we have 
\begin{align*}
\int_{\{0<u<t\} \cap B_{\sfrac{r}2}} |\nabla u|\,dx&\le \int_{\{0<u<t\} \cap B_{\sfrac{r}2}} \big(|\nabla u|^2+1\big)\,dx\le \max\{1,\sfrac1\mu\}\int_{\{0<u<t\} \cap B_{\sfrac{r}2}} \big(|\nabla u|^2+\mu\big)\,dx\\
&= \max\{1,\sfrac1\mu\}\Big(\int_{B_{\sfrac{r}2}} \big(|\nabla u|^2-|\nabla u_t|^2\big)\,dx+\mu\big(|\Omega_u\cap B_r|-|\Omega_{u_t}\cap B_r|\big)\Big)\le Ct.
\end{align*}
%$$\int_{\{0<u<t\} \cap B_{\sfrac{r}2}} |\nabla u|^2e^{-\Phi} dx + \mu|\{0<u<t\} \cap B_{\sfrac{r}2}|\leq C t.$$
%Therefore, using the inequality $2|\nabla u|\leq |\nabla u|^2+1$ and the boundedness of $\Phi$, we obtain
%$$\int_{\{0<u<t\} \cap B_{\sfrac{r}2}} |\nabla u|\,dx\le C t.$$
We now use the co-area formula to rewrite the above inequality as
\[ \frac1t\int_0^t Per\left(\{u>s\} ; B_{\sfrac{r}2}\right)\,ds\le C. \]
Hence, there is a sequence $t_n\to0$ such that $Per\left(\{u>t_n\} ; B_{\sfrac{r}2}\right)\,ds\le C$, which implies that $Per\left(\Omega_u ; B_{\sfrac{r}2}\right)\,ds\le C$. The last claim of the lemma follows by a standard covering argument. 
%Therefore, if $D_0$ is an open set strongly included in $D$, then the compact set $\partial \O_w \cap \overline{D}_0$ can be covered by a finite number of balls $B_r(x_0) \subset D$ centred at $x_0 \in \partial \O_w \cap \overline{D}_0$ such that $P(\O_w; B_r(x_0)) \leq C$, which shows that $\O_w$ is a set of finite perimeter in $D_0$. Finally, the last property is a consequence of the Federer Theorem (see e.g. \cite[Theorem 16.2]{Maggi}) and the fact that $\partial^\ast\O_w \subset \text{Reg}(\O_w)$, where $\partial^\ast \O_w$ stands for the reduced boundary of $\O_w$ (see \cite[Theorem 15.5]{Maggi}).
\end{proof}}

\subsection{Blow-up sequences and blow-up limits}\label{sub:blowup}
Let $u$ be a solution of \eqref{e:fbp} in the bounded open set $D\subset\R^d$. For $r>0$ and $x_0 \in \partial\O_u$, we define the rescaled function
$$u_{x_0,r}(x):= \frac{1}{r} u(x_0+rx).$$ 
Now since $u$ is Lipschitz continuous in some ball $B_{r_0}(x_0)$ (assume some regularity of the box if $x_0 \in \partial D$) we get that every sequence   $\left(u_{x_0,r_n}\right)_{n\geq 1}$   such that $r_n\to 0$ admits a subsequence (still denoted by $r_n$) that converges to a function $u_0:\R^d\to\R$ uniformly on every compact set $K\subset\R^d$. We say that $u_0$ is a blow-up limit of $u$ at $x_0$ and we use the notation $\mathcal{BU}_u(x_0)$ for the family of all blow-up limits of $u$ at $x_0$. 
We notice that, due to the non-degeneracy of $u$, the blow-up limits are non-trivial. Precisely, $u_0\neq 0$ and there is a constant $c>0$ such that $\|u_0\|_{L^\infty(B_r)}\ge cr$. 

The following proposition is standard. For a detailed proof we refer for example to  \cite[Proposition 4.5]{mazzoleni-terracini-velichkov}.  
\begin{prop}[Convergence of the blow-up sequences]\label{p:blowup}
Let $u$ be a solution of \eqref{e:fbp} and let $x_0 \in \partial\O_u$. Assume moreover that $D$ is of class $C^{1,1}$ if $x_0 \in \partial D$. Let $u_0\in\mathcal{BU}_u(x_0)$ and $u_n:=u_{x_0,r_n}$ be a blow-up sequence such that $u_n\to u_0$ locally uniformly in $\R^d$ as $n\to\infty$. Then 
\begin{enumerate}
	\item The sequence   $\left(u_n\right)_{n\geq 1}$   converges to $u_0$ strongly in $H^1_{\text{loc}}(\R^d)$.
	\item The sequence of characteristic functions  $\left(\ind_{\O_{u_n}}\right)_{n\geq 1}$   converges to $\ind_{\O_{u_0}}$ in $L^1_{\text{loc}}(\R^d)$.
	\item The sequences of closed sets   $\left(\overline{\O}_n\right)_{n\geq 1}$   and   $\left(\O_n^c\right)_{n\geq 1}$   Hausdorff converge locally in $\R^d$ to $\overline{\O}_0$ and $\O_0^c$, respectively.
	\item If $x_0 \in \partial\O_u\cap D$, then $u_0$ is a non-trivial global minimizer of the one-phase Alt-Caffarelli functional with $\Lambda=\Lambda_u e^{\Phi(x_0)}$ (see Definition \ref{def:one-phase} below). 
	
	\noindent If $x_0 \in \partial\O_u\cap\partial D$, then, up to a rotation, $u_0$ is a non-trivial global minimizer of the one-phase constrained Alt-Caffarelli functional with $\Lambda=\Lambda_u e^{\Phi(x_0)}$. 
	\end{enumerate}
\end{prop}	

\begin{definition}[Global minimizers of the one-phase problem]\label{def:one-phase}
 Let $\Lambda>0$ and $u\in H^1_{loc}(\R^d)$ be a non-negative function.
\begin{itemize}
\item We say that $u$ is a global minimizer of the one-phase Alt-Caffarelli functional with $\Lambda$, if  
\begin{equation}\label{e:OP AC func}
\int_{B} |\nabla u|^2 dx + \Lambda |\{u>0\}\cap B|\le \int_{B} |\nabla v|^2 dx + \Lambda |\{v>0\}\cap B|,
\end{equation}
for every ball $B\subset\R^d$ and every function $v \in H^1(B)$ such that $u-v\in H^1_0(B)$.
\item We say that $u$ is a global minimizer of the one-phase constrained Alt-Caffarelli functional with $\Lambda$, if $\Omega_{u} \subset \{x_d>0\}$ and \eqref{e:OP AC func} holds for every ball $B\subset\R^d$ and every function $v \in H^1(B)$ such that $u-v\in H^1_0(B)$ and $\Omega_v \subset \{x_d>0\}$. 
\end{itemize}
\end{definition} 

The optimality of the blow-up limit at points $x_0\in D$ (Proposition \ref{p:blowup}, claim (4)) follows by a standard argument based on our analysis in Subsection \ref{sub:small_scales}. Below, we give the proof in the case when $x_0$ lies on the boundary of $D$. The idea is to straighten out the boundary of the box and to show that the function $u$ in the new coordinates satisfies an almost-minimality condition. We only give the proof of Proposition \ref{p:blowup} (4) in order to show how to deal with the fact that on different scales $r$ the inclusion constraint on the set $\Omega_{u_r}$ changes and that at the limit the box $D$ becomes the half-space $\{x_d>0\}$.

Let $x_0=0 \in \partial\O_u \cap \partial D$.
Since $D$ is $C^{1,1}$ regular, there exist $\delta>0$ and a function $g : (-\delta,\delta)^{d-1} \rightarrow \R$ such that
\[ D \cap SQ_\delta = \{ (x',x_d) \in SQ_\delta \ : \ g(x')<x_d \}, \]
where $SQ_\delta = (-\delta,\delta)^d \subset B_{r_0}$. Moreover, up to a rotation, we can assume that the differential $Dg_0$ of $g$ at $0$ is zero. Let $\psi : SQ_\delta \subset \R^d \rightarrow \R^d$ be the function that straightens out the boundary of $D$ and let $\phi:=\psi^{-1} : \psi(SQ_\delta)\subset \R^d \rightarrow \R^d$ be its inverse:
\[ \psi(x',x_d) = (x',x_d-g(x')), \quad \phi(x',x_d) = (x',x_d+g(x')). \]
We define the matrix-valued function $A=(a_{ij})_{ij} : SQ_\delta \rightarrow \text{Sym}^+_d(\R)$ by
\[ A_x = (D\phi_x)^{-1} ({}^tD\phi_x)^{-1}, \quad \text{for every} \quad x\in SQ_\delta, \]
where ${}^tD\phi_x$ stands for the transpose of the Jacobian matrix of $\phi$ at $x$. Note that the coefficients $a_{ij}$ are Lipschitz continuous functions and that $A_x$ are symmetric positive definite matrices since they are small variations of $A_0=Id$. For $v \in H^1(\R^d)$ and $r>0$ we define the functional 
\[ \tilde{J}(v,r) = \int_{B_r} \Big(a_{ij}(x)\frac{\partial v}{\partial x_i}\frac{\partial v}{\partial x_j} -\lambda_m v^2\Big)e^{-\tilde{\Phi}}\,dx, \]
where we have set $\tilde{\Phi}=\Phi \circ \phi$. Moreover, we set $H=\{ (x',x_d) \in \R^d \ : \ x_d>0 \}$. With an elementary change of variables we get the following result.

\begin{lm}[Minimality of $u$ in the straightened coordinates]\label{l:coeff}
	Let $u$ be a solution of \eqref{e:fbp} and $x_0=0 \in \partial\O_u \cap \partial D$.  Let $h>0$. There exist $c>0$ and $r_0>0$ such that $B_{2r_0} \subset \psi(SQ_\delta)$ and the function $\tilde{u} = u \circ \Phi$ satisfies the minimality condition: for every  $r \in (0,r_0)$ we have
\[ \tilde{J}(\tilde{u},r) + \mu|\O_{\tilde{u}}\cap B_r| \leq \tilde{J}(\tilde{v},r) + \mu|\O_{\tilde{v}}\cap B_r| \]
for every $\tilde{v}\in H^1(B_{2r_0})$ such that $\tilde{u}=\tilde{v}$ on $B_{2r_0} \setminus B_r$, $\O_{\tilde{v}} \subset H$ and where 
\begin{equation}\label{e:defmu}
\mu=\begin{cases}\mu_+(h,0,cr) \quad\text{if}\quad |\O_{\tilde{u}}\cap B_r| \leq |\O_{\tilde{v}} \cap B_r| \leq |\O_{\tilde{u}}\cap B_r| + h, \\
\mu_-(h,0,cr) \quad\text{if}\quad |\O_{\tilde{u}}\cap B_r|-h \leq |\O_{\tilde{v}}\cap B_r| \leq |\O_{\tilde{u}}\cap B_r|.
\end{cases}
\end{equation}
\end{lm}

\begin{proof}
Let $r_0>0$ be such that $B_{2r_0} \subset \psi(SQ_\delta)$, $r \in (0,r_0)$ and $\tilde{v}$ such that $\tilde{u}=\tilde{v}$ on $B_{2r_0} \setminus B_r$, $\O_{\tilde{v}} \subset H$. Assume that $|\O_{\tilde{u}}\cap B_r| \leq |\O_{\tilde{v}} \cap B_r| \leq |\O_{\tilde{u}}\cap B_r| + h$. We define $v\in H^1_0(D)$ by $v = \tilde{v} \circ \psi$ in $\phi(B_{2r_0})$ and $v=u$ otherwise. 
Let $c$ be a positive constant depending only on $\phi$ such that $\phi(B_r)\subset B_{c r}$. Then, it follows that $u=v$ on $D \setminus B_{c r}$. Moreover, since $\det(D\phi_x)=1$ we have $|\O_u| \leq |\O_v|\leq |\O_u|+h$. Therefore, up to chosing $r_0>0$ smaller (depending on $c$), we get
\[ J(u) + \mu_+(h,0,c r)|\O_u| \leq J(v) + \mu_+(h,0,c r)|\O_v|. \]
Since we have $u=v$ on $\phi(B_r(x_0))$, this can rewrite as
\[ \int_{\phi(B_r)}\big(|\nabla u|^2-\lambda_m u^2\big)e^{-\Phi}\,dx + \mu|\O_u\cap \phi(B_r)| 
\leq \int_{\phi(B_r)}\big(|\nabla v|^2-\lambda_m v^2\big)e^{-\Phi}\,dx + \mu|\O_v\cap \phi(B_r)|, \]
where we have set $\mu =\mu_+(h,0,cr)$.
Now, a change of variables gives
\begin{align*}
\tilde{J}(\tilde{u},r) + \mu|\O_{\tilde{u}}\cap B_r| &= \int_{\phi(B_r)} \big(|\nabla u|^2-\lambda_m u^2\big)e^{-\Phi}\,dx + \mu|\O_u\cap \phi(B_r)| \\
& \leq \int_{\phi(B_r)}\big(|\nabla v|^2-\lambda_m v^2\big)e^{-\Phi}\,dx + \mu|\O_v\cap \phi(B_r)| =\tilde{J}(\tilde{v},r) + \mu|\O_{\tilde{v}}\cap B_r|.
\end{align*}
This concludes the proof.
\end{proof}

The next Lemma states that $\tilde u$ is an almost-minimizer also of the functional $J$.

\begin{lm}\label{l:quasi-min}
		Let $u$ be a solution of \eqref{e:fbp} and $x_0=0 \in \partial\O_u \cap \partial D$.  
Let $h>0$ be a given constant and let $c>0$ and $r_0>0$ be as in Lemma \ref{l:coeff}. Then there exists a constant $C>0$ such that $\tilde{u} = u \circ \phi$ satisfies the following almost-minimality condition: for every  $r \in (0,r_0)$ we have
\[ J(\tilde{u},r) + \mu|\O_{\tilde{u}}\cap B_r| \leq (1+Cr)J(\tilde{v},r) + \mu|\O_{\tilde{v}}\cap B_r| +Cr\int_{B_r} \tilde{v}^2\,dx + Cr^{d+1} \]
for every $\tilde{v}\in H^1(B_{2r_0})$ such that $\tilde{u}=\tilde{v}$ on $B_{2r_0} \setminus B_r$, $\O_{\tilde{v}} \subset H$ and where $\mu$ is as in \eqref{e:defmu}.
%\begin{align*}
%&\mu=\mu_+(h,0,cr) \quad\text{if}\quad |\O_{\tilde{u}}\cap B_r| \leq |\O_{\tilde{v}} \cap B_r| \leq |\O_{\tilde{u}}\cap B_r| + h \\
%&\mu=\mu_-(h,0,cr) \quad\text{if}\quad |\O_{\tilde{u}}\cap B_r|-h \leq |\O_{\tilde{v}}\cap B_r| \leq |\O_{\tilde{u}}\cap B_r|.
%\end{align*}
\end{lm}

\begin{proof}
Using the Lipschitz continuity of $A$ and $\Phi$, we estimate 
\begin{multline*}
J(\tilde{u},r) = \tilde{J}(\tilde{u},r) + \int_{B_r}(a_{ij}(x_0)-a_{ij}(x))\frac{\partial\tilde{u}}{\partial x_i}\frac{\partial\tilde{u}}{\partial x_j} e^{-\Phi}\,dx  \\ 
+ \int_{B_r}(a_{ij}(x)\frac{\partial\tilde{u}}{\partial x_i}\frac{\partial\tilde{u}}{\partial x_j} - \lambda_m\tilde{u}^2)(e^{-\Phi}-e^{-\tilde{\Phi}})\,dx \leq \tilde{J}(\tilde{u},r) + Cr^{d+1},
\end{multline*}
for some positive constant $C$ that does not depend on $r$. Analogously, we get the following estimate from below
\begin{align*}
J(\tilde{v},r) &= \tilde{J}(\tilde{v},r) + \int_{B_r}(a_{ij}(x_0)-a_{ij}(x))\frac{\partial\tilde{v}}{\partial x_i}\frac{\partial\tilde{v}}{\partial x_j} e^{-\Phi}\,dx  \\ 
&\qquad\qquad\qquad\qquad\qquad\qquad + \int_{B_r}(a_{ij}(x)\frac{\partial\tilde{v}}{\partial x_i}\frac{\partial\tilde{v}}{\partial x_j} - \lambda_m\tilde{v}^2)(e^{-\Phi}-e^{-\tilde{\Phi}})\,dx \\
&\geq \tilde{J}(\tilde{v},r) - Cr\Big(\int_{B_r}a_{ij}(x)\frac{\partial\tilde{v}}{\partial x_i}\frac{\partial\tilde{v}}{\partial x_j}e^{-\Phi}\,dx + \int_{B_r} \tilde{v}^2\,dx \Big) \\
&\geq (1-Cr)\tilde{J}(\tilde{v},r) -Cr\int_{B_r} \tilde{v}^2\,dx.
\end{align*}
Now, using Lemma \ref{l:coeff} and then combining the above estimates we get
\begin{align*}
J(\tilde{u},r) + \mu|\O_{\tilde{u}}\cap B_r| &\leq \tilde{J}(\tilde{u},r) + \mu|\O_{\tilde{u}}\cap B_r| + Cr^{d+1} \leq \tilde{J}(\tilde{v},r) + \mu|\O_{\tilde{v}}\cap B_r| + Cr^{d+1} \\
&\leq \frac{1}{1-Cr} \Big(J(\tilde{v},r) + Cr\int_{B_r} \tilde{v}^2\,dx \Big) + \mu|\O_{\tilde{v}}\cap B_r| + Cr^{d+1},
\end{align*}
which concludes the proof.
\end{proof}

We are now in position to prove the claim (4) of Proposition \ref{p:blowup} in the constrained case.

\begin{proof}[Proof of Proposition \ref{p:blowup} (4)]
Let $x_0=0 \in \partial\O_u \cap \partial D$ and let $u_0 \in \mathcal{BU}_u(x_0)$ be the blow-up limit of the sequence $u_n(x)=r_n^{-1}u(r_nx)$, where $(r_n)_n$ is some fixed sequence decreasing to $0$. Let $v\in H^1(B_r)$ be such that $u_0-v \in H^1_0(B_r)$ and $\O_v \subset H$. We define $v_n=v+\tilde{u}_n-u_0 \in H^1(\R^d)$, where we have set $\tilde{u}_n(x)=r_n^{-1}\tilde{u}(r_nx)$. Note that the sequence $(u_n)_n$ and $(\tilde{u}_n)_n$ converge to the same limit $u_0$ since the function $\phi$ is $C^{1,1}$ regular. Moreover, since $u_0-v \in H^1_0(B_r)$ we have $\tilde{u}_n-v_n \in H^1_0(B_r)$ and hence $\tilde{u}-v_n^{r_n} \in H^1_0(B_{rr_n})$, where we write $v_n^{r_n}(x)=r_n v_n(x/r_n)$. Note that we have $\O_{v_n^{r_n}} \subset H$. We set $h_n=|B_{rr_n}|$ and assume that $|\O_{\tilde{u}}\cap B_{rr_n}| \leq |\O_{v_n^{r_n}} \cap B_{rr_n}| \leq |\O_{\tilde{u}}\cap B_{rr_n}| + h_n$. Now we set $\mu=\mu_+(h_n,0,r)$ and $\Phi_n(x)=\Phi(r_nx)$ and we apply Lemma \ref{l:quasi-min} to the test function $v_n^{r_n}$ to estimate
\begin{align*}
\int_{B_r}(&|\nabla \tilde{u}_n|^2-r_n^2\lambda_m\tilde{u}_n^2)e^{-\Phi_n(x)}\,dx + \mu|\O_{\tilde{u}_n} \cap B_r| = \frac{1}{r_n^d}\Big(J(\tilde{u},rr_n) + \mu|\O_{\tilde{u}} \cap B_{rr_n}|\Big) \\
& \leq\frac{1}{r_n^d}\Big( (1+Cr r_n)J(v_n^{r_n},r r_n) + \mu|\O_{v_n^{r_n}}\cap B_{r r_n}| +Cr r_n\int_{B_{r r_n}} (v_n^{r_n})^2\,dx + C(r r_n)^{d+1}  \Big) \\
&=(1+Cr r_n)\int_{B_r}(|\nabla v_n|^2-r_n^2\lambda_mv_n^2)e^{-\Phi_n(x)}\,dx + \mu|\O_{v_n}\cap B_r| + Cr r_n^3\int_{B_r}v_n^2\,dx + Cr^{d+1}r_n.
\end{align*}
By Proposition \ref{p:blowup} (1) the sequence $\tilde{u}_n$ (resp. $v_n$) strongly converges in $H^1(B_r)$ to $u_0$ (resp. $v$) and the sequence of characteristic functions $\left(\ind_{\O_{\tilde{u}_n}}\right)_{n\geq 1}$ (resp. $\left(\ind_{\O_{v_n}}\right)_{n\geq 1}$) converges to $\ind_{\O_{u_0}}$ (resp. $\ind_{\O_v}$). Moreover, $\mu=\mu_+(h_n,0,r)$ tends to $\Lambda_u$ as $r_n \rightarrow 0$ by Theorem \ref{thm mu+-}. Therefore, passing at the limit in the above inequality and then multiplying by $e^{\Phi(x_0)}$ gives the claim.
\end{proof}

%\begin{oss}[Optimality of the blow-up sequence]
%The last claim of Proposition \ref{p:blowup} follows by a standard argument and by Theorem \ref{thm mu+-}. 
%Notice that, Theorem \ref{thm mu+-} implies that if $x_0\in\partial\Omega_u \cap D$ and $B\subset\R^d$ are fixed, then for every $\eps>0$ there is $r_0>0$ such that for every $0<r\le r_0$ we have  (setting $\Phi_r(x)=\Phi(x_0+rx)$)
%$$\int_B\left(|\nabla u_r|^2-r^2\lambda_mu_r^2\right)e^{-\Phi_r}dx+(\Lambda_u-\eps)|\Omega_{u_r}\cap\,B|\le\int_B\left(|\nabla v|^2-r^2\lambda_m v^2\right)e^{-\Phi_r}dx+(\Lambda_u-\eps)|\Omega_{v}\cap\,B|,$$
%for every $v\in H^1(B)$ such that $u_r-v\in H^1_0(B)$ and $|\Omega_v\cap B|\le |\Omega_{u_r}\cap B|$;
%$$\int_B\left(|\nabla u_r|^2-r^2\lambda_mu_r^2\right)e^{-\Phi_r}dx+(\Lambda_u+\eps)|\Omega_{u_r}\cap\,B|\le\int_B\left(|\nabla v|^2-r^2\lambda_m v^2\right)e^{-\Phi_r}dx+(\Lambda_u+\eps)|\Omega_{v}\cap\,B|,$$
%for every $v\in H^1(B)$ such that $u_r-v\in H^1_0(B)$ and $|\Omega_v\cap B|\ge |\Omega_{u_r}\cap B|$.

%Passing (formally) to the limit (as $r\to 0$) the two conditions above, we get that any blow-up limit $u_0$ satisfies 
%$$e^{-\Phi(x_0)}\int_B|\nabla u_0|^2\,dx+\Lambda_u|\Omega_{u_0}\cap\,B|\le e^{-\Phi(x_0)}\int_B|\nabla v|^2\,dx+\Lambda_u|\Omega_{v}\cap\,B|,$$
%for every $v\in H^1(B)$ such that $u_0-v\in H^1_0(B)$. 

%If $D$ is $C^{1}$ and $x_0 \in \partial\O_u \cap \partial D$, then the same holds with test functions $v\in H^1(B)$ such that (up to a rotation of the coordinate axes) $\O_v \subset \{x_d>0\}$.
%\end{oss}

\begin{oss}[Lebesgue density on the free boundary]\label{rem:density}
For every 
%$\Omega\subset\R^d$ and every 
$\gamma\in[0,1]$ we define
$$\ds \Omega_u^{(\gamma)}:=\Big\{x\in\R^d\ :\ \lim_{r\to0}\frac{|\Omega_u\cap B_r(x)|}{|B_r|}=\gamma\Big\}.$$
We notice that, as a consequence of Proposition \ref{p:blowup}, we get that $$\partial\Omega_u\cap D\cap\Omega_u^{(0)}=\emptyset\qquad\text{and}\qquad \partial\Omega_u\cap D\cap\Omega_u^{(1)}=\emptyset.$$ 
The first equality follows by the non-degeneracy of $u$, while the second one follows from the fact that all the blow-up limits vanish in zero and are global solutions of the Alt-Caffarelli problem.
\end{oss}

\subsection{Regularity of the free boundary}\label{sub:regularity}
In this section we prove Theorem \ref{thm main} (4) and (6), and Theorem \ref{t:th2} (4) and (6). We first show that the optimality condition $|\nabla u|^2=\Lambda_u e^{\Phi}$ on the free boundary $\partial\Omega_u\cap D$ and $|\nabla u|^2\geq\Lambda_u e^{\Phi}$ on $\partial\O_u\cap\partial D$ holds in the viscosity sense (Lemma \ref{l:visc}). We will then decompose the free boundary into regular and singular parts (Definition \ref{def:reg-sing}) and we will show that the regular part is $C^{1,\alpha}$ regular (Proposition \ref{p:regularity}).

\begin{definition}[Optimality condition in viscosity sense]
	Let $D$ be an open set and $u:D\to\R$ be continuous, that is, $u \in C(D)$. 
	
	$\bullet$ We say that $\varphi \in C(D)$ touches $u$ by below (resp. by above) at $x_0 \in D$ if $\varphi(x_0) = u(x_0)$ and $\varphi \leq u$ (resp. $\varphi \geq u$) in a neighborhood of $x_0$. 
	
	%$$\varphi \leq w \quad (\text{resp. } \varphi \geq w) \quad \text{in a neighbourhood of } x_0. $$
	
	$\bullet$ Let $\Lambda$ be a non-negative function on $\overline D$ and assume that $u$ is non-negative. We say that $u$ satisfies the boundary condition
	\[ |\nabla u| = \sqrt{\Lambda}\quad \text{on}\quad \partial \O_u \cap D \]
	in viscosity sense if, for every $\varphi \in C^2(D)$ such that $\varphi^+$ touches $u$ by below (resp. by above) at some $x_0 \in \partial \O_u \cap D$, we have 
	$|\nabla \varphi|(x_0) \leq \sqrt{\Lambda}$ (resp. $|\nabla \varphi|(x_0) \geq \sqrt{\Lambda})$).
	
	\noindent Analogously, we say that $u$ satisfies the boundary condition
	\[ |\nabla u| \geq \sqrt{\Lambda}\quad \text{on}\quad \partial \O_u \cap \partial D \]
	in viscosity sense if, for every $\varphi \in C^2(D)$ such that $\varphi^+$ touches $u$ by above at some $x_0 \in \partial \O_u \cap \partial D$, we have
	$|\nabla \varphi|(x_0) \geq \sqrt{\Lambda}$. 
\end{definition}

\begin{lm}[Optimality condition on the free boundary]\label{l:visc}
	Let $D\subset \R^d$ be a bounded open set of class $C^{1,1}$ and let $u$ be a solution of \eqref{e:fbp}. Then $u$ is a solution of the problem
	\begin{equation}\label{eq opt cond}
	\left\{
	\begin{aligned}
	& -\dive(e^{-\Phi}u)=\lambda_m ue^{-\Phi}\ \text{ in }\ \O_u, \\
	& \ |\nabla u| = \sqrt{\Lambda_ue^{\Phi}}\ \text{ on }\ \partial \O_u \cap D, \\
	& \ |\nabla u| \geq \sqrt{\Lambda_ue^{\Phi}}\ \text{ on }\ \partial \O_u \cap \partial D, \\
	\end{aligned} 
	\right.
	\end{equation} 
	where the boundary conditions hold in viscosity sense.
\end{lm}

To prove the optimality condition we will need the following result.

\begin{lm}\label{l:u0>0=xd0>0}
Let $u \in H^1(\R^d)$ be a non-trivial, continuous and one-homogeneous function (in the sense that $u(tx)=tu(x)$ for every $t>0$) such that $u(0)=0$. Assume moreover that $u$ is harmonic in the set $\O_u$. If $\O_u\subset \{x_d>0\}$ then $\O_u=\{x_d>0\}$, while if $\O_u\supset\{x_d>0\}$ then either $\O_u=\{x_d>0\}$ or $\O_u = \{x_d\neq 0\}$.
\end{lm}

\begin{proof}
Set $S=\O_u\cap\partial B_1$ and denote by $C_S=\{ r\theta \ : \ \theta\in S,\ r>0\}$ the cone generated by $S$. Since $u$ is a one-homogeneous function and is solution of
\[ \Delta u = 0 \quad\text{in}\quad C_S,\qquad u=0\quad\text{on}\quad \partial C_S, \]
it follows that the trace $\varphi=u_{|_{\partial B_1}}$ is a solution of
\[ -\Delta_{\mathbb{S}^{d-1}}\varphi = (d-1)\varphi\quad\text{in}\quad S,\qquad\varphi=0\quad\text{on}\quad\partial S. \]
Therefore $\varphi$ is a first eigenfunction of $-\Delta_{\mathbb{S}^{d-1}}$ in $S$ (because $\varphi>0$ in $S$) and hence $\lambda_1(S)=d-1$. Note also that $\varphi_1=(x_d)_+$ is the first eigenfunction on the set $S_+=\{x_d>0\}\cap\partial B_1$ with eigenvalue $\lambda_1(S_+)=d-1$.

Firstly, assume that $S\subset S_+$. Then $u\in H^1_0(S_+)$ and by the variational characterization of $\lambda_1(S_+)$ we have
\[ \frac{\int_{\partial B_1}|\nabla \varphi|^2\,d\mathcal{H}^{d-1}}{\int_{\partial B_1}\varphi^2\,d\mathcal{H}^{d-1}}=\lambda_1(S)=\lambda_1(S_+)\leq \frac{\int_{\partial B_1}|\nabla (\varphi+t\psi)|^2\,d\mathcal{H}^{d-1}}{\int_{\partial B_1}(\varphi+t\psi)^2\,d\mathcal{H}^{d-1}} \]
for every $\psi \in H^1_0(S_+)$ and $t\in\R$. This gives that $\varphi$ is solution of 
\[ -\Delta_{\mathbb{S}^{d-1}}\varphi=\lambda_1(S_+)\varphi\quad\text{in}\quad S_+,\qquad \varphi=0\quad\text{on}\quad \partial S_+, \]
that is, $\varphi$ is the first eigenfunction on $S_+$. Since $\lambda_1(S_+)$ is simple (because $S_+$ is connected) is follows that $\varphi=c(x_d)_+$ for some $c>0$. In particular, $\{\varphi>0\}=S_+$ and hence $\O_u=\{x_d>0\}$ by one-homogeneity of $u$.

Assume now that $S\supset S_+$ and write $S=S_0\sqcup S_1$, where $S_0$ is the connected component of $S$ which contains $S_+$. If $S_1\neq\emptyset$, then it follows by the preceding step that $S_1 = S_-:= \{x_d<0\}\cap\partial B_1$; hence $S= S_+\cup S_-$ and $\O_u= \{x_d\neq 0\}$. Now, if $S_1=\emptyset$, then $S=S_0$ is connected. Moreover, $\varphi_1 \in H^1_0(S)$ and using the variational characterization of $\lambda_1(S)$ it follows that $\varphi_1$ is the first eigenfunction in $S$. Then $\varphi_1>0$ in $S$ (since $S$ is connected) which proves that $S=S_+$.
\end{proof}

\begin{proof}[Proof of Lemma \ref{l:visc}]
	From Proposition \ref{p:lipschitz} it follows that $u$ is continuous in $D$. We only have to prove that $u$ satisfies the two boundary conditions in the viscosity sense. We first show that $|\nabla u| = \sqrt{\Lambda_ue^{\Phi}}$ holds on $\partial \O_u \cap D$. Let $\varphi \in C^2(D)$ a function such that $\varphi^+$ touches $u$ by below at $x_0\in \partial \O_u \cap D$. Let $r_n$ be an infinitesimal sequence and 
	\begin{equation}\label{e:u_nphi_n}
	u_n(x) = \frac{1}{r_n}u(x_0 +r_nx) \quad \text{and} \quad \varphi_n(x) = \frac{1}{r_n}\varphi(x_0 + r_n x).
	\end{equation}
	Up to a subsequence, $u_n$ converges locally uniformly to some $u_0\in \mathcal{BU}_{u}(x_0)$, while $\varphi_n$ converges to $\varphi_0(x):=x\cdot \nabla\varphi(x_0)$. Up to a change of coordinates, we may suppose that  $\nabla\varphi(x_0)=|\nabla\varphi(x_0)|e_d$. If $|\nabla\varphi(x_0)|=0$, then $|\nabla\varphi(x_0)| \leq \sqrt{\Lambda}$, where we have set $\Lambda = \Lambda_u e^{\Phi(x_0)}$, and we are done. Otherwise, we have $u_0>0$ in the half-space $\{x_d>0\}$  since $u_0\geq \varphi_0$. Moreover, $u_0$ is a one-homogeneous function by Lemma \ref{l:homogeneity} and it follows that $\O_{u_0}=\{x_d>0\}$ by Lemma \ref{l:u0>0=xd0>0}, because the case $\O_{u_0}=\{x_d\neq 0\}$ is ruled out (by \textit{(4)} in Proposition \ref{p:blowup} or Remark \ref{rem:density}).	
	Moreover, $u_0$ is a local minimizer of the Alt-Caffarelli functional for $\Lambda$  by Proposition \ref{p:blowup} and hence satisfies (in the classical sense) the optimality condition 
	\[ |\nabla u_0| = \sqrt{\Lambda}\quad\text{on}\quad \{x_d=0\} \]
	(see \cite[Theorem 2.5]{alt-caffarelli}). This implies that $u_0=\sqrt{\Lambda}\,x_d^+$. To see this, note that the boundary condition implies that $v$ defined by $v=u_0$ in $\{x_d>0\}$ and $v=\sqrt{\Lambda}\,x_d$ in $\{x_d\leq 0\}$ is harmonic in $\R^d$; hence $v=\sqrt{\Lambda}\,x_d$ by uniqueness of the solution to Cauchy problem for the Laplacian. Finally, since $u_0\geq \varphi_0$ we have $\sqrt\Lambda\ge |\nabla \varphi_0|(0)=|\nabla \varphi|(x_0)$.
	The proof if now $\varphi^+$ touches $u$ from above is similar. In this case we have $\O_{u_0}\subset\{x_d>0\}$ since $u_0\leq\varphi_0^+$, and hence $\O_u=\{x_d>0\}$ by Lemma \ref{l:u0>0=xd0>0}. Notice that $u_0$ is a non trivial function by the non-degeneracy property in Lemma \ref{l:nondegeneracy}.
	
	Suppose now that $\varphi^+$ touches $u$ from above at $x_0\in \partial \Omega_u\cap\partial D$ and consider $u_n$ and $\varphi_n$ defined in \eqref{e:u_nphi_n}. By Proposition \ref{p:blowup}, $u_n$ converges to a local minimizer $u_0$ of the Alt-Caffarelli functional with $\Lambda=\Lambda_ue^{\Phi(x_0)}$ in the half-space $\{x_d>0\}$ (up to a change of coordinates). Therefore, the sequence $u_{0n}:=r_n^{-1}u_0(r_nx)$ converges to a limit $u_{00}$ which is a solution of the constrained Alt-Caffarelli problem in $\{x_d>0\}$ and is one-homogenous (see Remark \ref{r:homogconstr}). Therefore, $\O_{u_{00}}=\{x_d>0\}$ by Lemma \ref{l:u0>0=xd0>0} and $u_{00}$ satisfies the optimality condition 
		\[ |\nabla u_{00}| \geq \sqrt{\Lambda}\quad\text{on}\quad \{x_d=0\}. \]
	 This implies that $u_{00}(x)=\alpha x_d^+$, for some $\alpha\ge \sqrt\Lambda$, and thus that  $\sqrt\Lambda\le \alpha\le |\nabla \varphi_0|(0)=|\nabla \varphi|(x_0)$ (since $u_{00}\le\varphi_0^+$).
\end{proof}

\begin{oss}\label{r:homogconstr}
The homogeneity of the blow-up limits of the (local) minimizers of the Alt-Caffarelli functional was first obtained by Weiss in \cite{weiss}. In the case of the constrained problem, when the solution $u$ is optimal only among the functions with support in $\{x_d>0\}$, the Weiss formula can still be applied because the one-homogeneous extensions of $u$ are admissible competitors. Thus, the blow-up limits in this case are still one-homogeneous. We refer for instance to \cite[Proposition 4.3]{spolaor-trey-velichkov} and Lemma \ref{l:weiss} below. 
\end{oss}

\begin{oss}[On the Alt-Caffarelli optimality condition]
	Using an argument based on an internal variation of the boundary as in \cite[Theorem 2.5]{alt-caffarelli} we can get in a weak sense the optimality boundary condition given in Lemma \ref{l:visc}, namely: for every $x_0 \in \partial\O_u \cap D$, $r>0$ such that $B_r(x_0)\subset D$ and $\xi \in C_0^\infty(B_r(x_0),\R^d)$ we have
	\begin{equation*}
	\lim_{\varepsilon\downarrow 0} \int_{\partial\{u>\varepsilon\}}\big(|\nabla u|^2 - \Lambda_u e^\Phi\big)\,e^{-\Phi}\,\xi\cdot\nu\,d\mathcal{H}^{d-1} = 0,
	\end{equation*}
	while for every $x_0 \in \partial\O_u\cap\partial D$, $r>0$ such that $D_r(x_0)$ is connected and every $\xi \in C_0^\infty(B_r(x_0),\R^d)$ such that $(Id+\xi)^{-1}(D_r(x_0))\subset D_r(x_0)$ we have
	\begin{equation*}
	\lim_{\varepsilon\downarrow 0} \int_{\partial\{u>\varepsilon\}}\big(|\nabla u|^2 - \Lambda_u e^\Phi\big)\,e^{-\Phi}\,\xi\cdot\nu\,d\mathcal{H}^{d-1} \geq 0.
	\end{equation*}
\end{oss} 

\begin{definition}[Regular and singular parts of the free boundary]\label{def:reg-sing}
We say that $x_0\in\partial\Omega_u$ is a {\emph regular point}
	if there exists a blow-up $u_0\in\mathcal{BU}_u(x_0)$ of the form 
	\begin{align}
	&u_0(x)=\sqrt{\Lambda_ue^{\Phi(x_0)}}\,(x\cdot \nu)_+ &\text{if}\quad &x_0 \in \partial\O_u\cap D, \label{e:blowup1}\\
	&u_0(x)=q\,(x\cdot \nu)_+ &\text{if}\quad &x_0 \in \partial\O_u\cap \partial D, \label{e:blowup2}
	\end{align}
	where $\nu\in\partial B_1$ is some unit vector and $q$ is a constant such that $q\ge \sqrt{\Lambda_ue^{\Phi(x_0)}}$. 
	
\noindent We denote by $Reg(\partial\Omega_u\cap D)$ the set of regular points (the regular part of the free boundary) in $D$, and by $Reg(\partial\Omega_u)$ the set of all regular points of $\partial\Omega_u$. 
We define the singular part of the boundary as $\singu:=\partial\Omega_u\setminus\regu$ and $Sing(\partial\Omega_u\cap D)=\partial\Omega_u\setminus Reg(\partial\Omega_u\cap D)$.
\end{definition} 

%\begin{oss}Notice that by definition we have 
%$$Reg(\partial\Omega_u\cap D)\subset  Reg(\partial\Omega_u)\qquad\text{and}\qquad Sing(\partial\Omega_u)\subset  Sing(\partial\Omega_u\cap D).$$\end{oss}

\begin{prop}[Regularity of the free boundary]\label{p:regularity}
	Suppose that $u$ is a solution of \eqref{e:fbp} in the bounded open set $D\subset\R^d$. Then, we have:
	\begin{enumerate}
	\item $Reg(\partial\Omega_u\cap D)$ is locally the graph of a $C^{1,\alpha}$ function for any $\alpha<1$;
	\item the reduced boundary $\partial^\ast\Omega_u\cap D$ is contained in $Reg(\partial\Omega_u\cap D)$; 
	\item $\HH^{d-1}(Sing(\partial\Omega_u\cap D))=0$; moreover,   if $d\le 4$, then $Reg(\partial\Omega_u\cap D)=\emptyset$. 		\end{enumerate}
If $D$ is a $C^{1,1}$ regular domain, then:
\begin{enumerate}
\item[(4)] $Reg(\partial\Omega_u)$ is locally the graph of a $C^{1,\sfrac12}$ regular function;
\item[(5)] $\partial\Omega_u\cap \partial D\subset Reg(\partial\Omega_u)$; 
\item[(6)] $Reg(\partial\Omega_u\cap D)\subset  Reg(\partial\Omega_u)$ and $Sing(\partial\Omega_u)=  Sing(\partial\Omega_u\cap D).$
\end{enumerate}
\end{prop}
\begin{proof}
By Lemma \ref{l:visc}, $u$ is a viscosity solution of \eqref{eq opt cond}. Let $x_0\in\partial \Omega_u$ be a regular point. Then, for some $r>0$ small enough the function $u_{r,x_0}=\frac1ru(x_0+rx)$ is also a viscosity solution and is $\eps$-flat in the sense of \cite{de-silva}. Applying the results of De Silva \cite{de-silva} (in the case when $x_0\in D$) and Chang-Lara-Savin \cite{chang-lara-savin} (if $x_0\in\partial D$), we get the claims (1) and (4). 

We next prove (2) and (3). Let $x_0\in \partial \Omega_u\cap D$ and $u_n:=u_{x_0,r_n}$ be a blow-up sequence at $x_0$ converging to some $v\in \mathcal{BU}_u(x_0)$ such that    $\ind_{\Omega_{u_n}}$   converges in $L^1_{loc}(\R^d)$ to   $\ind_{\Omega_{v}}$.   If $x_0\in \partial^\ast \Omega_u\cap D$, then $\Omega_{v}$ is a half-plane of the form $H=\{x\in\R^d\ :\ x\cdot\nu>0\}$ for some $\nu\in\partial B_1$. Without loss of generality, we assume $\nu=e_d$. On the other, hand $v$ is a solution to the Alt-Caffarelli problem with $\Lambda=\Lambda_ue^{\Phi(x_0)}$. Thus $v$ is harmonic in $H$ and zero on $\partial H$, so it is smooth up to the boundary of $H$. Now, the optimality condition $|\nabla v|=\sqrt\Lambda$ on $\partial H$ and the unique continuation of harmonic functions in the half-plane imply that $v(x)=\sqrt{\Lambda} x_d^+$, which  proves (2). Now, since $\Omega_u$ has (locally) finite perimeter in $D$, the Federer Theorem and Remark \ref{rem:density} give that
	$$\HH^{d-1}\big(\partial\Omega_u\cap D\setminus(\partial^\ast \Omega_u)\big)=\HH^{d-1}\left(\partial\Omega_u\cap D\setminus\big(\partial^\ast \Omega_u\cup \Omega_u^{(0)}\cup \Omega_u^{(1)}\big)\right)=0,$$
	which proves that $\HH^{d-1}(Sing(\partial\Omega_u\cap D))=0$. We now prove the second claim of (3). As above, let $x_0\in \partial \Omega_u\cap D$ and $v=\lim_{n\to\infty} u_{r_n,x_0}$ be a blow-up limit of $u$ at $x_0$. Then $v$ is a solution of the Alt-Caffarelli problem problem. Let $\rho_n\to 0$ and $v_{\rho_n}(x)=\frac1{\rho_n}v(x\rho_n)$ be a sequence that converges locally uniformly to a function $v_0$. Since $d\le 4$, we have that the free boundary $\partial\Omega_v$ is $C^{1,\alpha}$ and $v_0$ is of the form $v_0(x)=\sqrt{\Lambda} (x\cdot\nu)_+$ for some $\nu\in\partial B_1$ (see \cite{alt-caffarelli} for $d=2$, \cite{caffarelli-jerison-kenig} for $d=3$ and \cite{jerison-savin} for $d=4$). Now since, for fixed $n>0$, we have that $v_{\rho_n}=\lim_{m\to\infty}u_{\rho_nr_m,x_0}$, we can choose a diagonal sequence $u_{R_n,x_0}$, where $R_n=\rho_n r_{m(n)}$, such that $v_0=\lim_{n\to\infty}u_{R_n,x_0}$. This proves that $x_0$ is a regular point. 
	
	The claim (5) follows by the same diagonal sequence argument. This time $x_0\in\partial D$ and the blow-up $v$ is a solution of the constrained Alt-Caffarelli problem in $\{x_d>0\}$. Thus, the blow-up $v_0$ of $v$ is one-homogeneous solution of the constrained problem for $\Lambda=\Lambda_ue^{\Phi(x_0)}$. This implies (in any dimension) that $v_0(x)=q x_d^+$ for some $q\ge \sqrt{\Lambda}$ (see \cite[Proposition 4.3]{spolaor-trey-velichkov}).   
	
	Finally, (6) follows by the definition of the regular part and claim (5). 	
%	In order to prove the last claim we recall that every blow-up $u_0\in \mathcal{BU}_u(x_0)$ is a solution of the one-phase Alt-Caffarelli problem. Thus, by \cite{alt-caffarelli} (for $d=2$), \cite{caffarelli-jerison-kenig} (for $d=3$) and \cite{jerison-savin} (for $d=4$), the free boundary $\partial \Omega_{u_0}$ is locally a graph of a smooth function and so the blow-up $u_{00}$ of $u_0$ in $0$ is of the form \eqref{e:blowup1}. Now since $u_{00}\in \mathcal{BU}_u(x_0)$ we get that $x_0\in\regu$.  
%		We first notice that, since $D$ is smooth, every point $x_0 \in \partial\O_u\cap \partial D$ is flat and hence every blow-up at $x_0 \in \partial\O_u\cap \partial D$ is of the form \eqref{e:blowup2}; in particular, $\partial\O_u \cap \partial D \subset \regu$.  
%	Now, the regularity of the regular part of the free boundary $\regu\cap D$ follows by Lemma \ref{l:visc} and the improvement of flatness Theorem from \cite{de-silva}, while the regularity of $\regu\cap \partial D$ follows from \cite{chang-lara-savin}. Thus, we only need to prove the estimate on $\singu$. 
\end{proof}

\begin{oss}[On the higher regularity of the free boundary]
	The smoothness of the free boundary can be improved under an additional regularity assumption on $\Phi$. Indeed, if $\nabla \Phi \in C^{k+1,\alpha}(D;\R^d)$ for some $k\geq 1$ and $\alpha \in (0,1)$, then by \cite[Theorem 1]{kinderlehrer-nirenberg}, $\reg\cap D$ is locally a graph of a $C^{k+1,\alpha}$ function.
\end{oss}

\subsection{Monotonicity formula and some further estimates on the dimension of the singular set}\label{sub:weiss}
This section is dedicated to the estimates on the dimension of the singular set (Theorem \ref{thm main} (5) and Theorem \ref{t:th2} (5)). The main ingredient is a monotonicity formula that implies the homogeneity of the blow-up limits at any free boundary point $x_0\in\partial\Omega_u\cap D$. 

%We prove in this section a monotonicity formula for a Weiss-type boundary adjusted energy of the solution $u$ to \eqref{e:fbp} at every point of the free boundary $\partial\O_u\cap D$. 
Let $u$ be a solution \eqref{e:fbp} and $\Lambda_u$ be the constant given by Theorem \ref{p:lagrange}.
We define the Weiss-type boundary adjusted energy as
\[ W(u,\Phi,x_0,r) = \frac{1}{r^d}\int_{B_r(x_0)}|\nabla u|^2e^{-\Phi}\,dx - \frac{1}{r^{d+1}}\int_{\partial B_r(x_0)}u^2e^{-\Phi}\,d\HH^{d-1} + \frac{\Lambda_u}{r^d}|\Omega_u\cap B_r(x_0)|.\] 

\begin{lm}[Weiss monotonicity formula]\label{l:weiss}
Let $u$ be a solution \eqref{e:fbp} in the bounded open set $D$. Then, for every $x_0\in\partial\Omega_u\cap D$ and every $0<r<\text{dist}(x_0,\partial D)$, the function $W$ satisfies the differential inequality 
\begin{equation}\label{e:weiss}
\frac{d}{dr}W(u,\Phi,x_0,r)\ge \frac{2e^{-\max\Phi}}{r^{d+2}}\int_{\partial B_r(x_0)}\left|\nabla u\cdot x -u\right|^2\,d\HH^{d-1}-C,
\end{equation} 
where $C>0$ is a constant depending only on $\lambda_m,\Phi,L:=\|\nabla u\|_{L^\infty}$ and the dimension $d$. 
%\noindent If, moreover, $D$ is of class $C^{1,1}$, then there exists a constant $r_0>0$ depending only on $D$, such that, for every $x_0\in\partial\O_u\cap\partial D$ and every $0<r<r_0$ the inequality \eqref{e:weiss} holds for some constant $C>0$ which  depends also on $D$.
\end{lm}

\begin{proof}
We first prove the claim when $x_0 \in \partial\O_u\cap D$. Assume $x_0=0$. We set 
$$H(r):=\int_{\partial B_r}u^2e^{-\Phi}d\HH^{d-1}\qquad\text{and}\qquad D(r):=\int_{B_r}|\nabla u|^2e^{-\Phi}dx,$$
$$H_\Phi(r):=\int_{\partial B_r}(\nu\cdot\nabla\Phi)\,u^2e^{-\Phi}d\HH^{d-1}\qquad\text{and}\qquad D_\Phi(r):=\int_{B_r}\left(|\nabla u|^2-\lambda_mu^2\right)(x\cdot\nabla\Phi)e^{-\Phi}dx,$$
where $\nu(x)=\sfrac{x}r$ is the exterior normal to the sphere $\partial B_r$ at $x$.
  As in Proposition   \ref{p:density} (notice that in Proposition \ref{p:density}  $D_\Phi$ is defined differently) we have  
$$D'(r)=\int_{\partial B_r}|\nabla u|^2e^{-\Phi}d\HH^{d-1}\quad\text{and}\quad H'(r)=\frac{d-1}{r}H(r)+2D(r)-2\lambda_m\int_{B_r}u^2e^{-\Phi}dx-H_\Phi(r).$$
 Let $\phi_\eps$ be a radially decreasing function such that 
 \begin{equation}\label{e:phi_eps}
 0\le \phi_\eps\le 1\ \, \text{in}\ \,  B_r,\quad  \phi_\eps=1\ \,  \text{in}\ \,  B_{r(1-\eps)},\quad \phi_\eps=0\ \,  \text{on}\ \, \partial B_r\quad \text{and}\quad|\nabla\phi_\eps|\leq C(r\eps)^{-1}. 
 \end{equation}
As in Step 2 of the proof of Proposition \ref{p:density}, the optimality condition   $\ \ds\delta J(u)[\xi]=\Lambda_u\int_{\Omega_u}\dive\xi\,dx\ $,   applied to the vector field $\xi(x)=x\phi_\eps(x)$, gives that 
\begin{align*}
\Lambda_u\left(d|\Omega_u\cap B_r|-r\HH^{d-1}(\Omega_u\cap\partial B_r)\right)&=-(d-2)D(r)+rD'(r)-2r\int_{\partial B_r}\left(\partial_\nu u\right)^2e^{-\Phi} d\HH^{d-1}\\
&\quad+\lambda_m\left(d\int_{B_r}u^2e^{-\Phi}dx-r\int_{\partial B_r}u^2e^{-\Phi}d\HH^{d-1}\right)+ rD_\Phi(r),
\end{align*}
where $\partial_\nu u:=\nu\cdot\nabla u$. We now calculate 
\begin{align*}
\frac{d}{dr}W(u,\Phi,x_0,r)&
=\frac{1}{r^d}D'(r)-\frac{d}{r^{d+1}}D(r)-\frac{1}{r^{d+1}}H'(r)+\frac{d+1}{r^{d+2}}H(r)\\
&\qquad+\frac{\Lambda_u}{r^{d+1}}\left(r\HH^{d-1}(\Omega_u\cap \partial B_r)-d|\Omega_u\cap B_r|\right)\\
%&=-\frac{4}{r^{d+1}}D(r)+\frac{2}{r^d}\int_{\partial B_r}\!\left(\frac{\partial u}{\partial n}\right)^2\!e^{-\Phi} d\HH^{d-1}+\frac{2}{r^{d+1}}H(r)+\frac{D_\Phi(r)+H_\Phi(r)}{r^{d+1}}\\
%&\qquad+\frac{\lambda_m}{r^{d+1}}\left((d+2)\int_{B_r}u^2e^{-\Phi}dx-r\int_{\partial B_r}u^2e^{-\Phi}d\HH^{d-1}\right)\\
& =\frac{2}{r^{d+2}}\int_{\partial B_r}\left|\nabla u\cdot x -u\right|^2 e^{-\Phi}d\HH^{d-1}+\frac{1}{r^{d+1}}H_\Phi(r)-\frac{1}{r^d}D_\Phi(r) \\
&\qquad-\frac{\lambda_m}{r^{d+1}}\left((d+2)\int_{B_r}u^2e^{-\Phi}dx-r\int_{\partial B_r}u^2e^{-\Phi}d\HH^{d-1}\right) \\
& \ge\frac{2e^{-\max\Phi}}{r^{d+2}}\int_{\partial B_r}\left|\nabla u\cdot x-u\right|^2d\HH^{d-1}-C,
\end{align*} 
which gives the claim if $x_0 \in \partial\O_u\cap D$ and $r<\text{dist}(x_0,\partial D)$. 
\end{proof}

\begin{lm}[Homogeneity of the blow-up limits]\label{l:homogeneity}
Let $u$ be a solution \eqref{e:fbp} in the bounded open set $D$ and let  $x_0\in\partial\Omega_u$. Then every blow-up limit $u_0\in \mathcal{BU}_u(x_0)$ is one-homogeneous. 
\end{lm}
\begin{proof}
Let $x_0=0$ and $W(u,\Phi,r):=W(u,\Phi,x_0,r)$. Recall that $u_r(x)=\frac1ru(rx)$ and $\Phi_r(x)=\Phi(rx)$. We first notice that for every $r>0$ and $s>0$ such that $rs\le \text{dist}(x_0,\partial D)$ we have 
$$W(u_r,\Phi_r,s)=W(u,\Phi,rs).$$ 
Moreover, since the function $r\mapsto W(u,\Phi,t)+Cr$ is monotone, the limit $$W(u,\Phi,0):=\lim_{r\to0^+}W(u,\Phi,r)$$
exists (and is finite due to the Lipschtz continuity of $u$). On the other hand, for every blow-up sequence $u_{r_n}$ with blow-up limit $u_0$, we have 
$$W(u_0,\Phi(0),s)=\lim_{n\to\infty}W(u_{r_n},\Phi_{r_n}(0),s)=\lim_{n\to\infty}W(u,\Phi(0),r_ns)=W(u,\Phi(0),0).$$ 
Thus, the function 
$$\ds s\mapsto \frac1{s^d}\int_{B_s}|\nabla u_0|^2\,dx-\frac{1}{s^{d+1}}\int_{\partial B_s}u_0^2\,d\HH^{d-1}+\frac{\Lambda_ue^{\Phi(0)}}{s^d}|\Omega_{u_0}\cap B_s|,$$
is constant. Now, by \cite{weiss} (or, simply by applying \eqref{e:weiss} to $u=u_0$, $\lambda_m=0$ and $\Phi=0$), we have that $u_0$ is one-homogeneous. 
\end{proof}

%\subsection{Further estimates on the dimension of the singular set}\label{sub:singular}
%This subsection is dedicated to some more refined estimates on the dimension of the singular set. 
%Precisely, the optimal sets for the problem \eqref{e:optO} are $d^\ast$-regular, where $d^\ast$ is defined below.

\begin{definition}\label{def:dstar}
We define $d^*$ as the smallest dimension which admits one-homogeneous global minimizers of the one-phase Alt-Caffarelli problem with (isolated) singularity in zero.
\end{definition}
By \cite{jerison-savin} and \cite{de-silva-jerison} we know that $d^*\in\{5,6,7\}$. Weiss was first to prove that the monotonicity formula implies the dimension estimate 
$$\dim_\HH \singu:=\inf\{\alpha\ge0\ :\ \HH^\alpha(\singu)=0\}\le d-d^*,$$ 
for every $d>d^*$ (see also \cite{mazzoleni-terracini-velichkov} for an argument using only the monotonicity of $W$). 
%Recently, using the innovative approach of Naber and Valtorta \cite{naber-valtorta}, Edelen and Engelstein \cite{edelen-engelstein} showed the the monotonicity formula implies the stronger estimate $\HH^{d-d^\ast}(\singu)<\infty$. 
Thus, as a consequence of Lemma \ref{l:weiss}, Lemma \ref{l:homogeneity} and the results from \cite{weiss} and \cite{mazzoleni-terracini-velichkov}, we get 
\begin{prop}[On the dimension of the singular set]\label{p:singular}
Let $u$ be a solution of \eqref{e:fbp} in the bounded open set $D\subset\R^d$. Then 

$\bullet$ $\singu=\emptyset$ if $d<d^\ast$; 

$\bullet$ $\singu$ is a discrete (locally finite) set if $d=d^\ast$;

$\bullet$ $\dim_\HH \singu<d-d^\ast$ if $d>d^\ast$.
\end{prop}
\begin{proof}
The proof of this proposition is standard, once we have the monotonicity of $W$ (Lemma \ref{l:weiss}) and the consequent homogeneity of the blow-up limits (Lemma \ref{l:homogeneity}). We refer to \cite[Section 4]{weiss} and \cite[Section 5.5]{mazzoleni-terracini-velichkov}.  
\end{proof}

\begin{oss}
Recently, using the innovative approach of Naber and Valtorta \cite{naber-valtorta}, Edelen and Engelstein \cite{edelen-engelstein} showed the monotonicity formula of Weiss can be used to obtain the (local) estimate $\HH^{d-d^\ast}(\singu)<\infty$, which in particular implies that $\dim_\HH \singu<d-d^\ast$.
\end{oss}

\appendix
\section{Extremality conditions and Lebesgue density }
In this section we prove Proposition \ref{p:density}, which we use in Proposition \ref{p:lagrange} to show that the Lagrange multiplier $\Lambda_u$ is strictly positive, but the result is of independent interest. For instance, it applies to optimal partition problems (see, for example, \cite{conti-terracini-verzini} and \cite{caffarelli-lin}). We first show that a function which is critical for the functional 
\begin{equation}\label{e:app:defJ}
J(u):=\int_D|\nabla u|^2e^{-\Phi}\,dx-\lambda \int_Du^2e^{-\Phi}\,dx,
\end{equation}
with respect to internal variations that is 
$$\ds\delta J(u)[\xi]:=\lim_{t\to0}J(u(x+t\xi(x)))=0\quad\text{for every vector field}\quad \xi\in C^\infty_c(D;\R^d),$$
satisfies a monotonicity formula for the associated Almgren frequency function $N(r)$. Now, by the argument of Garofalo and Lin (see \cite{garofalo-lin}) the monotonicity of the frequency function implies that $u$ cannot decay too fast around the free boundary points. If, in addition, $u$ is a solution of $-\dive(e^{-\Phi}\nabla u)=\lambda u e^{-\Phi}$ on the positivity set $\Omega_u=\{u>0\}$, we can use a Caccioppoli inequality to show that if the Lebesgue density of $\Omega_u$ is too small, then the decay of $u$ on the balls of radius $r$ should be very fast. This, in combination with the monotonicity of the Almgren's frequency function, shows that the Lebesgue density of $\Omega_u$ should be bounded from below \emph{everywhere} (and not only on the boundary of $\Omega_u$). In particular, there cannot be points of zero Lebesgue density for $\Omega_u$ in $D$.  
%\begin{equation}\label{e:first_variation_app}
%	\delta J(u)[\xi] := \int_D \Big[ 2 D\xi(\nabla u)\cdot\nabla u + \big(|\nabla u|^2 - \lambda u^2\big) (\nabla \Phi \cdot \xi - \dive \xi) \Big] e^{-\Phi} dx.
%\end{equation}
\begin{prop}\label{p:density}
Let $D\subset\R^d$ be a bounded open set and $\Phi\in W^{1,\infty} (D)$. Suppose that $\lambda\ge0$ and $u\in H^1(D)$ is a nonnegative (non-identically-zero) function such that
\begin{enumerate}[(a)]
\item $u$ is a solution of the equation 
\begin{equation}\label{e:app:equ}
-\dive(e^{-\Phi}\nabla u)=\lambda e^{-\Phi}u\quad\text{in}\quad \Omega_u=\{u>0\};
\end{equation}
\item $u$ satisfies the extremality condition 
$$\delta J(u)[\xi]=0\quad\text{for every}\quad \xi\in C^\infty_c(D;\R^d),$$
where $J$ is given by \eqref{e:app:defJ}
%$$J(u)=\int_D|\nabla u|^2e^{-\Phi}\,dx-\lambda \int_Du^2e^{-\Phi}\,dx,$$
and its first variation in the direction $\xi$ is given by 
\begin{equation}\label{e:first_variation_app}
\delta J(u)[\xi] := \int_D \Big[ 2 D\xi(\nabla u)\cdot\nabla u + \big(|\nabla u|^2 - \lambda u^2\big) (\nabla \Phi \cdot \xi - \dive \xi) \Big] e^{-\Phi} dx.
\end{equation}
\end{enumerate}
Then, $|D\setminus\Omega_u|=0$. 
\end{prop}
\subsection{Reduction to the case $\lambda=0$}
In this section we will show that it is sufficient to prove Proposition \ref{p:density} for $\lambda=0$. The general case will then follow by an elementary substitution argument. In the next lemma we deal with the first variation of the functional $J$. 
\begin{lm}\label{l:app:Ja}
Suppose that $D\subset\R^d$ is a bounded open set, $a:D\to\R$ is a given Lipschitz function such that $0<\eps\le a\le \eps^{-1}$ on $D$. Let $\lambda>0$ and let $\varphi\in H^2(D)$ be such that 
$$-\dive(a\nabla\varphi)=\lambda a\varphi\quad\text{in}\quad D,\qquad{\varphi\ge \eps>0}\quad\text{on}\quad D.$$  
For any $u\in H^1(D)$, we set  $\tilde a(x):=\varphi^2(x)a(x)$, $\tilde u:=\sfrac{u}\varphi$, 
%For any $v\in H^1(D)$ and any  $\xi\in C^\infty_c(D;\R^d)$, we set   
$$ J(u):=\int_{D}\left(|\nabla u|^2-\lambda u^2\right)a(x)\,dx\qquad\text{and}\qquad\tilde J(u):=\int_{D}|\nabla u|^2\tilde a(x)\,dx,$$
$$\delta J(u)[\xi]:=\int_D \Big[ 2a D\xi(\nabla u)\cdot\nabla u - \left(|\nabla u|^2-\lambda u^2\right) \dive (a\xi ) \Big] dx,$$
$$\delta \tilde J(u)[\xi]:=\int_D \Big[ 2\tilde a D\xi(\nabla u)\cdot\nabla u - |\nabla u|^2 \dive (\tilde a\xi ) \Big]  dx\quad\text{for any}\quad\xi\in C^\infty_c(D;\R^d).$$
Then, for every $u\in H^1(D)$ and every $\xi\in C^\infty_c(D;\R^d)$, we have  
\begin{equation}\label{e:app:Ja}
\delta \tilde J (\tilde u)[\xi]=\delta J(u)[\xi]-2\int_D\nabla \left(u \xi\cdot\nabla (\ln\varphi) \right)\cdot\nabla u\,a\,dx+2\int_D\left(u \xi\cdot\nabla (\ln\varphi) \right)\lambda a u\,dx.
\end{equation}
\end{lm}
\begin{proof}
Notice that we may assume $u\in C^\infty(D)$. First we notice that an integration by parts gives
\begin{align*}
\delta \tilde J(\tilde u)[\xi] &=\int_D 2\,\partial_i\xi_j\,\partial_i \tilde u\, \partial_j\tilde u\,\tilde a\,dx-\int_D|\nabla \tilde u|^2\dive(\tilde a\xi)\,dx\\
&=-\int_D 2\,\xi_j\,\partial_i(\tilde a\,\partial_i \tilde u)\,\partial_j\tilde u\,dx-\int_D 2\,\xi_j\,\partial_i \tilde u\,\partial_{ij}\tilde u\,\tilde a\,dx-\int_D|\nabla \tilde u|^2\dive(\tilde a\xi)\,dx\\
&=-\int_D 2\,\xi_j\,\partial_i(\tilde a\,\partial_i \tilde u)\,\partial_j\tilde u\,dx-\int_D \dive(\tilde a|\nabla \tilde u|^2\xi)\,dx=-\int_D 2\,\xi_j\,\partial_i(\tilde a\,\partial_i \tilde u)\,\partial_j\tilde u\,dx\\
&=-\int_D 2(\xi\cdot\nabla\tilde u)\dive(\tilde a\nabla\tilde u)\,dx .
\end{align*}
and, analogously,  
\begin{align*}
\delta J(u)[\xi] =-\int_D 2(\xi\cdot\nabla u)\dive( a\nabla u)\,dx+\lambda\int_Du^2\dive(a\xi)\,dx.
\end{align*}
Now, since 
$$\dive(\tilde a\nabla\tilde u)=\dive(a(\varphi\nabla u-u\nabla\varphi))=\varphi\dive(a\nabla u)-u\dive(a\nabla\varphi)=\varphi(\dive(a\nabla u)+\lambda a u),$$ 
we get 
\begin{align*}
\delta \tilde J(\tilde u)[\xi]&=-2\int_D\xi\cdot(\nabla u-\frac{u}\varphi\nabla\varphi)\big(\dive(a\nabla u)+\lambda a u\big)\,dx\\
&=2\int_D\xi\cdot\nabla \varphi \frac{u}\varphi\big(\dive(a\nabla u)+\lambda a u\big)\,dx-2\int_D(\xi\cdot\nabla u)\big(\dive(a\nabla u)+\lambda a u\big)\,dx\\
&=-2\int_D\nabla \left(\frac{\xi\cdot\nabla \varphi}\varphi u\right)\cdot\nabla u\,a\,dx+2\int_D\left(\frac{\xi\cdot\nabla \varphi}\varphi u\right)\lambda a u\,dx+\delta J(u)[\xi],
\end{align*}
which is precisely \eqref{e:app:Ja}.\end{proof}
Let now $D\subset\R^d$ and $u\in H^1(D)$ be as in Proposition \ref{p:density} for some $\lambda>0$. In order to prove that $|D\setminus \Omega_u|=0$, it is sufficient to prove that $|(D\cap B)\setminus \Omega_u|=0$ for any (small) ball $B\subset D$. Let now $x_0\in D$ and let $R>0$ be such that $\lambda_1(B_R(x_0),\nabla\Phi)=\lambda$. Such a radius exists, since the map $f(r):= \lambda_1(B_r(x_0),\nabla\Phi)$ is continuous, $f(0)=+\infty$ and $f(+\infty)=0$. Notice also that we may assume $\Phi$ to be defined on the entire space $\R^d$. Let $\varphi$ be the first eigenfunction on $B_R(x_0)$ and let $r=R/2$. Then, we can apply Lemma \ref{l:app:Ja} in the set $D\cap B_r(x_0)$ with $a=e^{-\Phi}$. Moreover, since $u$ satisfies \eqref{e:app:equ}, we get that 
$$\delta \tilde J(\tilde u)[\xi]=\delta J(u)[\xi]=0,\quad\text{for every}\quad \xi\in C^\infty_c(D\cap B_r(x_0);\R^d),$$
which proves that $\tilde u=\sfrac{u}\varphi$ satisfies hypothesis (b) for $\lambda=0$. 
Finally, in order to prove that  $\tilde u$ satisfies hypothesis (a), we notice that on $\Omega_u=\Omega_{\tilde u}$ we have (in a weak sense)
$$\dive(\tilde a\nabla\tilde u)=\varphi\dive(a\nabla u)-u\dive(a\nabla\varphi)=\varphi\left(\dive(a\nabla u)+\lambda a u\right)=0.$$

\subsection{Proof of Proposition \ref{p:density} in the case $\lambda=0$}
%\begin{proof}
Let $\lambda=0$. Then we have 
\begin{equation}\label{e:app:defJ0}
J(u):=\int_D|\nabla u|^2e^{-\Phi}\,dx,
\end{equation}
\begin{equation}\label{e:first_variation_app0}
\delta J(u)[\xi] := \int_D \Big[ 2 D\xi(\nabla u)\cdot\nabla u + |\nabla u|^2 (\nabla \Phi \cdot \xi - \dive \xi) \Big] e^{-\Phi} dx.
\end{equation}
Let $x_0=0\in D$ and $\tau=\|\nabla \Phi\|_{L^\infty(D)}$. We set 
$$H(r):=\int_{\partial B_r}u^2e^{-\Phi}d\HH^{d-1},\qquad D(r):=\int_{B_r}|\nabla u|^2e^{-\Phi}dx\qquad\text{and}\qquad N(r):=\frac{rD(r)}{H(r)}.$$
{\it Step 1. Derivative of $H$.} We calculate 
\begin{align*}
H'(r)&=\frac{d-1}{r}H(r)+r^{d-1}\frac{d}{dr}\int_{\partial B_1}u^2(rx)e^{-\Phi(rx)}d\HH^{d-1}(x)\\
&=\frac{d-1}{r}H(r)+2\int_{\partial B_r}u\frac{\partial u}{\partial n}e^{-\Phi}d\HH^{d-1}-\int_{\partial B_r}u^2(n\cdot\nabla\Phi)e^{-\Phi}d\HH^{d-1}\\
&=\frac{d-1}{r}H(r)+2\int_{B_r}|\nabla u|^2e^{-\Phi}dx-\int_{\partial B_r}u^2(n\cdot\nabla\Phi)e^{-\Phi}d\HH^{d-1},
\end{align*}
which we rewrite as
\begin{equation}\label{e:app:derH}
H'(r)=\frac{d-1}{r}H(r)+2 D(r)-H_\Phi(r).
\end{equation}
where we have set
$$H_\Phi(r):=\int_{\partial B_r}u^2(n\cdot\nabla\Phi)e^{-\Phi}d\HH^{d-1}\qquad\text{and}\qquad|H_\Phi(r)|\le \tau H(r).$$
{\it Step 2. Equidistribution of the energy.} Let $\phi_\eps$ be a radially decreasing function such that $0\le \phi_\eps\le 1$ on $B_r$, $\phi_\eps=1$ on $B_{r(1-\eps)}$, $\phi_\eps=0$ on $\partial B_r$ and $|\nabla\phi_\eps|\leq C(r\eps)^{-1}$. The vector field $\xi(x):=x\phi_\eps(x)$ satisfies $\dive\xi(x)=d\phi_\eps(x)+x\cdot\nabla\phi_\eps$ and $\partial_i\xi_j=\delta_{ij}\phi_\eps(x)+x_j\partial_i\phi_\eps(x)$. Since $\lambda=0$ we have
\begin{align*}
\delta J(u)[\xi]&= \int_D \Big[ 2 D\xi(\nabla u)\cdot\nabla u + |\nabla u|^2 (\nabla \Phi \cdot \xi - \dive \xi) \Big] e^{-\Phi} dx\\
&= \int_D \Big[ 2 |\nabla u|^2\phi_\eps+2(x\cdot\nabla u)(\nabla\phi_\eps\cdot\nabla u) - |\nabla u|^2  (d\phi_\eps(x)+x\cdot\nabla\phi_\eps) \Big] e^{-\Phi} dx\\
&\qquad+\int_D|\nabla u|^2(\nabla \Phi \cdot x)\phi_\eps e^{-\Phi} dx,
\end{align*}
and passing to the limit as $\eps\to0$, rearraging the terms and using the property (b), we get
\begin{align*}
0&=-(d-2)\int_{B_r}|\nabla u|^2e^{-\Phi}dx+r\int_{\partial B_r}|\nabla u|^2e^{-\Phi} d\HH^{d-1}\\
&\qquad-2r\int_{\partial B_r}\left(\frac{\partial u}{\partial n}\right)^2e^{-\Phi} d\HH^{d-1}+\int_{B_r}|\nabla u|^2 (\nabla \Phi \cdot x) e^{-\Phi} dx,
\end{align*}
which we rewrite as 
\begin{align*}
-(d-2)D(r)&+rD'(r)= 2r\int_{\partial B_r}\left(\frac{\partial u}{\partial n}\right)^2e^{-\Phi} d\HH^{d-1}-rD_\Phi(r),
\end{align*}
where 
$$D_\Phi(r):=\frac1r\int_{B_r}|\nabla u|^2 (\nabla \Phi \cdot x) e^{-\Phi} dx\qquad\text{and}\qquad|D_\Phi(r)|\le \tau D(r).$$

\noindent {\it Step 3. The derivative of $N$.} We notice that $N(r)$ is only defined for $r$ such that $H(r)>0$. In what follows we fix $r_0>0$ such that $B_{r_0}(x_0)\subset D$ and $H(r_0)>0$. Since $u\in H^1(D)$, there is an interval $(a,b)\ni r_0$, on which $H>0$.   
\begin{align}
N'(r)&=\frac{D(r)H(r)+rD'(r)H(r)-rD(r)H'(r)}{H^2(r)}\notag\\
&=  \frac{D(r)H(r)+rD'(r)H(r)-r D(r)\left(\frac{d-1}{r}H(r)+2D(r)- H_\Phi(r)\right)}{H^2(r)}\notag\\
&=\frac{-(d-2)D(r)H(r)+rD'(r)H(r)-2rD^2(r)+r D(r)H_\Phi(r)}{H^2(r)}\notag\\
&=\frac{2r}{H^2(r)}\left(H(r)\int_{\partial B_r}\left(\frac{\partial u}{\partial n}\right)^2e^{-\Phi} d\HH^{d-1}-D^2(r)\right)+\frac{r \left(D(r)H_\Phi(r)-D_\Phi(r)H(r)\right)}{H^2(r)}\label{e:app:derN}
\end{align}
Now we notice that, since $u$ solves \eqref{e:app:equ} on $\Omega_u$, we have
\begin{align*}
D(r)&=\int_{B_r}|\nabla u|^2e^{-\Phi}dx=\int_{\partial B_r}u\frac{\partial u}{\partial n}e^{-\Phi} d\HH^{d-1},
\end{align*}
and so, by the Cauchy-Schwarz inequality and \eqref{e:app:derN} we obtain
\begin{equation}\label{e:app:der_N_final_est}
N'(r)\ge\frac{r \left(D(r)H_\Phi(r)-D_\Phi(r)H(r)\right)}{H^2(r)}\ge -2\tau N(r).
\end{equation}

\noindent {\it Step 4. A bound on $N(r)$.} Using the estimate \eqref{e:app:der_N_final_est} from the previous step we get that the function 
$\,r\mapsto e^{2\tau r}N(r)\,$
is non-decreasing in $r$ and so
$$N(r)\le e^{2\tau (r_0-r)}N(r_0)\le e^{2\tau r_0}N(r_0)\quad\text{for every}\quad a<r\le r_0.$$

\noindent {\it Step 5. Strict positivity and doubling inequality for $H(r)$.} By the step $4$ we have 
\begin{equation}\label{e:der logH}
\frac{d}{dr}\left[\log\left(\frac{H(r)}{r^{d-1}}\right)\right]=2\frac{N(r)}r-\frac{H_\Phi(r)}{H(r)}\le \frac{2e^{2\tau r_0}N(r_0)}{r}+\tau,
\end{equation}
and integrating we get 
$$\log\left(\frac{H(r_0)}{r_0^{d-1}}\right)-\log\left(\frac{H(r)}{r^{d-1}}\right)\le \log\Big(\frac{r_0}{r}\Big)\,2e^{2\tau r_0}N(r_0)+\tau r_0,\quad\text{for every}\quad a<r\le r_0.$$
In particular, $H>0$ on every interval $[\eps r_0,r_0]$ and so, $H>0$ on $(0,r_0]$ and we might take $a=0$. Moreover, integrating once again the inequality \eqref{e:der logH} from $r<\sfrac{r_0}2$ to $2r$, we get
\begin{equation*}
\log\left(\frac{H(2r)}{H(r)}\right)\le ((d-1)\log 2+\tau r_0)+2\log 2\,e^{2\tau r_0}N(r_0)\qquad\text{for every}\qquad 0<r\le \frac{r_0}2.
\end{equation*}
Taking $r_0\le 1$, there is a constant $C$, depending only on $d$ and $\tau$, such that
\begin{equation}\label{e:app:doubling0}
H(2r)\le C\exp (CN(r_0))H(r)\qquad\text{for every}\qquad 0<r\le \frac{r_0}2.
\end{equation}
Integrating once more in $r$ we get 
\begin{equation}\label{e:app:doubling}
\int_{B_{2r}}u^2e^{-\Phi}\,dx\le C\exp (CN(r_0))\int_{B_{r}}u^2e^{-\Phi}\,dx\qquad\text{for every}\qquad 0<r\le \frac{r_0}2.
\end{equation}
\noindent {\it Step 6. Caccioppoli inequality and conclusion.} 
Let $r\in(0,\sfrac{r_0}2]$ and let $\phi\in C^\infty_0(B_{2r})$ be such that $\phi=1$ in $B_r$, $\phi=0$ on $\partial B_{2r}$, $0\le\phi\le 1$ and $|\nabla \phi|\le \sfrac{2}{r}$ on $B_{2r}\setminus B_r$. Using the fact that $u$ is a solution of $-\dive(e^{-\Phi}\nabla u)=0$ in $\Omega_u$, we get the following Caccioppoli inequality: 
\begin{align}
\int_{B_r}|\nabla u|^2e^{-\Phi}\,dx&\le \int_{B_{2r}}|\nabla (u\phi)|^2e^{-\Phi}\,dx=\int_{B_{2r}}\left(u^2|\nabla\phi|^2+\nabla u\cdot\nabla (u\phi^2)\right)e^{-\Phi}\,dx\notag\\
&= \int_{B_{2r}}u^2|\nabla\phi|^2e^{-\Phi}\,dx-\int_{B_{2r}} u\phi^2\dive \left(e^{-\Phi}\nabla u)\right)\,dx= \int_{B_{2r}}u^2|\nabla\phi|^2e^{-\Phi}\,dx.\notag\\
&\le \frac4{r^2}\int_{B_{2r}}u^2e^{-\Phi}\,dx.\label{e:app:caccioppoli}
\end{align}
On the other hand, there are dimensional constants $C_d$ and $\eps_d>0$ such that, if $|\Omega_u\cap B_r|\le \eps_d|B_r|$, then the following inequality does hold (see \cite[Lemma 4.4]{buttazzo-velichkov-stoch}) 
\begin{equation*}
\int_{B_r}u^2\,dx\le C_d r^2\bigg(\frac{|\Omega_u\cap B_r|}{|B_r|}\bigg)^{2/d}\int_{B_r}|\nabla u|^2\,dx,
\end{equation*}
which, taking $C:=C_d\exp(\max\Phi-\min\Phi)$, implies 
\begin{equation*}
\int_{B_r}u^2e^{-\Phi}\,dx\le C  r^2\bigg(\frac{|\Omega_u\cap B_r|}{|B_r|}\bigg)^{2/d}\int_{B_r}|\nabla u|^2e^{-\Phi}\,dx.
\end{equation*}
This, together with \eqref{e:app:caccioppoli} and the doubling inequality \eqref{e:app:doubling}, gives that there are constants $C_1$ and $C_2$, depending only on $d$ and $\tau$ such that 
$$\min\left\{\eps_d,C_1\exp (-C_2N(r_0))\right\}\le \frac{|\Omega_u\cap B_r|}{|B_r|}\qquad\text{for every}\qquad 0<r\le\frac{r_0}{2},$$
where to be precise we recall that we assumed $r_0\le 1$. In particular, we have a lower density bound for $\Omega_u$ at \emph{every} point of $D$, which implies that $|D\setminus\Omega_u|=0$ and concludes the proof. \qed

\bigskip\bigskip
\noindent {\bf Acknowledgments.} 
The authors have been partially supported by Agence Nationale de la Recherche (ANR) by the projects GeoSpec (LabEx PERSYVAL-Lab, ANR-11-LABX-0025-01). The third author was also partially supported by the project ANR CoMeDiC (ANR-15-CE40-0006).

\bibliographystyle{plain}
\bibliography{bib-rtv}

\end{document}